\documentclass[11pt,reqno]{amsart}
\usepackage[a4paper,margin=30mm,headheight=14pt]{geometry}
\usepackage{amsmath,amssymb,amsthm,mathtools,mathrsfs,booktabs,array}
\usepackage{xcolor,enumitem,etoolbox,microtype}
\usepackage[T1]{fontenc}
\usepackage{lmodern}
\usepackage{xurl}
\usepackage{hyperref}
\usepackage{aliascnt}
\usepackage[nameinlink,capitalize,noabbrev]{cleveref}
\hypersetup{
  colorlinks=true,
  linkcolor=blue!45!black,
  citecolor=blue!45!black,
  urlcolor=blue!45!black,
  pdftitle={Weighted potential and Green systems with Dini coefficients and primitive growth},
  pdfauthor={Hongtao Hu, Zhouyu Long and Wenming Zou},
  pdfsubject={Liouville theorems for weighted potential and Green systems},
  pdfkeywords={Lane--Emden systems, Liouville theorems, polyharmonic equations, Riesz potentials, Dini continuity, integral representation}
}
\numberwithin{equation}{section}
\newtheorem{theorem}{Theorem}[section]
\theoremstyle{plain}
\newaliascnt{lemma}{theorem}
\newtheorem{lemma}[lemma]{Lemma}
\aliascntresetthe{lemma}
\theoremstyle{plain}
\newaliascnt{proposition}{theorem}
\newtheorem{proposition}[proposition]{Proposition}
\aliascntresetthe{proposition}
\theoremstyle{plain}
\newaliascnt{corollary}{theorem}
\newtheorem{corollary}[corollary]{Corollary}
\aliascntresetthe{corollary}
\theoremstyle{definition}
\newaliascnt{definition}{theorem}
\newtheorem{definition}[definition]{Definition}
\aliascntresetthe{definition}
\theoremstyle{remark}
\newaliascnt{remark}{theorem}
\newtheorem{remark}[remark]{Remark}
\aliascntresetthe{remark}
\theoremstyle{remark}
\newaliascnt{example}{theorem}
\newtheorem{example}[example]{Example}
\aliascntresetthe{example}
\theoremstyle{plain}
\crefname{theorem}{Theorem}{Theorems}
\Crefname{theorem}{Theorem}{Theorems}
\crefname{lemma}{Lemma}{Lemmas}
\Crefname{lemma}{Lemma}{Lemmas}
\crefname{proposition}{Proposition}{Propositions}
\Crefname{proposition}{Proposition}{Propositions}
\crefname{corollary}{Corollary}{Corollaries}
\Crefname{corollary}{Corollary}{Corollaries}
\crefname{definition}{Definition}{Definitions}
\Crefname{definition}{Definition}{Definitions}
\crefname{remark}{Remark}{Remarks}
\Crefname{remark}{Remark}{Remarks}
\crefname{example}{Example}{Examples}
\Crefname{example}{Example}{Examples}
\newcommand{\fint}{\mathop{\mathchoice{\displaystyle\mathrlap{\mkern3mu-}\int}{\textstyle\mathrlap{\mkern2mu-}\int}{\scriptstyle\mathrlap{\mkern1.5mu-}\int}{\scriptscriptstyle\mathrlap{\mkern1mu-}\int}}\nolimits}
\setlist{nosep}
\AtBeginEnvironment{thebibliography}{\small\raggedright\setlength{\itemsep}{0pt}}
\title[Weighted potential and Green systems]{Weighted potential and Green systems\\with Dini coefficients and primitive growth}
\author[Hongtao Hu]{\bfseries Hongtao Hu$^{*}$}
\author[Zhouyu Long]{\bfseries Zhouyu Long$^{\dagger}$}
\author[Wenming Zou]{\bfseries Wenming Zou$^{\ddagger}$}
\thanks{$^{*}$Department of Mathematical Sciences, Tsinghua University,
Beijing 100084, China.\newline
E-mail: \href{mailto:huht23@mails.tsinghua.edu.cn}{\nolinkurl{huht23@mails.tsinghua.edu.cn}}}
\thanks{$^{\dagger}$Department of Mathematical Sciences, Tsinghua University,
Beijing 100084, China.\newline
E-mail: \href{mailto:longzy25@mails.tsinghua.edu.cn}{\nolinkurl{longzy25@mails.tsinghua.edu.cn}}}
\thanks{$^{\ddagger}$Department of Mathematical Sciences, Tsinghua University,
Beijing 100084, China.\newline
E-mail: \href{mailto:zou-wm@mail.tsinghua.edu.cn}{\nolinkurl{zou-wm@mail.tsinghua.edu.cn}}}

\makeatletter
\patchcmd{\@setauthors}{\centering\footnotesize}{\centering\large}{}%
  {\ClassError{amsart}{Unable to set author font}{Check the amsart author layout.}}
\patchcmd{\@setauthors}{\MakeUppercase{\authors}}{\authors}{}%
  {\ClassError{amsart}{Unable to preserve author case}{Check the amsart author layout.}}
\def\author@andify{\nxandlist{\unskip,\quad}{\unskip,\quad}{\unskip,\quad}}
\let\addresses\@empty
\makeatother
\date{}

\begin{document}
\raggedbottom
\maketitle
\vskip0.6in
\begin{center}
\begin{minipage}{140mm}
\textbf{Abstract:} We prove a nonexistence theorem for positive
potential systems with continuous coefficients satisfying distributional
Euler lower bounds and a uniform Dini modulus on normalized logarithmic
dilations. The nonlinearities have positive lower logarithmic
derivatives and positive primitive-to-endpoint ratios. The common radial kernel is
comparable with a Riesz kernel. A localized
transport identity, a primitive pressure estimate and implicit amplitude
normalization give uniform local energy bounds and a dilation
contradiction. And we treats the genuine Green kernel of
$-\Delta+V$ for nonnegative potentials under annular and Newton
composition budgets and a strict gap corrected by the relative Green dilation
defect. Relative endpoint differences and an exact virtual-energy
cancellation retain local nonlinear energy as the only initial energy
hypothesis. Independent representation arguments give
these consequences, including even integer orders without initial tails.
A compact family of normalized responses also yields universal interior
bounds without regular variation.

\vskip0.231in
\textbf{MSC Classification (2020):} Primary 35B53; Secondary 35J48, 35R11, 45G15.
\vskip0.23in
\end{minipage}
\end{center}
\newpage

\section{Introduction}\label{sec:intro}
Let $0<\sigma<n$ and write $L_\sigma=(-\Delta)^{\sigma/2}$ and
\[
 K_\sigma(x)=c_{n,\sigma}|x|^{\sigma-n},\qquad
 c_{n,\sigma}=\frac{\Gamma((n-\sigma)/2)}
 {2^\sigma\pi^{n/2}\Gamma(\sigma/2)},\qquad I_\sigma H=K_\sigma*H.
\]
The Fourier symbol of $I_\sigma$ is $|\xi|^{-\sigma}$.
For the normalization and standard potential estimates, see
\cite[Chapter~V]{Stein1970} and \cite{AdamsHedberg1996}.
All unqualified integrals are over $\mathbb R^n$. The balls $B_R$
are centered at zero. We state a common potential theorem before its PDE consequences.

Fix $e_0\in\mathbb S^{n-1}$ and put
$\mathcal A=\{x:1/2\le |x|\le2\}$. For positive continuous
coefficients $a,b$ on $\mathbb R^n\setminus\{0\}$ write
\[
 a_R(x)=a(Rx)/a(Re_0),\qquad b_R(x)=b(Rx)/b(Re_0).
\]
We use global Euler lower bounds in the distributional sense:
\begin{equation}
 \operatorname{div}(xa)\ge(n+A_-)a\,dx,\qquad
 \operatorname{div}(xb)\ge(n+B_-)b\,dx
 \quad\hbox{on }\mathbb R^n\setminus\{0\}.
 \label{v5:euler-lower}
\end{equation}
For continuous coefficients these are equivalent to
\begin{equation}
 a(tx)\ge t^{A_-}a(x),\qquad b(tx)\ge t^{B_-}b(x)
 \quad(x\ne0,\ t\ge1).
 \label{r6:weak-euler-front}
\end{equation}
The equivalence follows by polar disintegration and continuity;
it is also explained in Section~\ref{vi:section}.
For $C^1$ coefficients, \eqref{v5:euler-lower} is precisely
$\vartheta_a:=x\cdot\nabla\log a\ge A_-$ and
$\vartheta_b:=x\cdot\nabla\log b\ge B_-$.
A modulus means a finite nondecreasing function
$\Omega:[0,4]\to[0,\infty)$ with $\Omega(0)=0$ and
$\Omega(s)\to0$ as $s\downarrow0$.

\begin{theorem}
\label{thm:dini-front}\label{r7:master-front}
\label{r6:master-front}\label{r6:nonlinear-front}
\label{r6:radial-kernel-nonexistence}
Let $n\ge1$, $0<\sigma<n$, and $\lambda=n-\sigma$.
Let $K(z)=k(|z|)$, where $k\in C^1((0,\infty))$ is positive and
\begin{equation}
 c_Kr^{-\lambda}\le k(r)\le C_Kr^{-\lambda},\qquad
 0\le\theta_K(r):=-\frac{rk'(r)}{k(r)}\le\Lambda
 \quad(r>0)
 \label{r6:kernel-class}
\end{equation}
for finite positive constants $c_K,C_K,\Lambda$.
Let $a,b>0$ be continuous off zero and satisfy the weak Euler
bounds \eqref{v5:euler-lower}, where $A_-,B_-$ are arbitrary real
numbers. Suppose that for some $R_0>0$, finite $H_a,H_b$ and a
modulus $\Omega$,
\begin{equation}
 \begin{split}
 \int_0^{1/2}\Omega(s)\,\frac{ds}{s}&<\infty,\\
 |\log a_R(x)-\log a_R(y)|&\le H_a\Omega(|x-y|),\\
 |\log b_R(x)-\log b_R(y)|&\le H_b\Omega(|x-y|)
 \quad(R\ge R_0,\ x,y\in\mathcal A).
 \end{split}
 \label{eq:front-dini-orbit}
\end{equation}
When $n=1$, compare points only within the same sign branch.
Let $F,G$ be continuous on $[0,\infty)$, locally Lipschitz on
$(0,\infty)$, positive there, and zero at zero. Set
$P_F(s)=\int_0^sF(t)\,dt$, $P_G(s)=\int_0^sG(t)\,dt$.
Assume that constants $p_0,q_0,\mu_F,\mu_G>0$ satisfy
\begin{equation}
 \begin{split}
 \frac{sF'(s)}{F(s)}&\ge q_0,\qquad
 \frac{sG'(s)}{G(s)}\ge p_0\quad\hbox{for almost every }s>0,
 \qquad p_0q_0>1,\\
 P_F(s)&\ge\mu_FsF(s),\qquad
 P_G(s)\ge\mu_GsG(s)\quad(s>0),
 \end{split}
 \label{r7:response-assumptions}
\end{equation}
and
\begin{equation}
 \Gamma_P:=(n+B_-)\mu_F+(n+A_-)\mu_G-\Lambda>0.
 \label{r7:primitive-gap}
\end{equation}
There is no finite strictly positive pair continuous off zero
satisfying
\begin{equation}
 u=K*(aG(v)),\qquad v=K*(bF(u))
 \quad\hbox{at every }x\ne0,
 \label{v5:master-system}
\end{equation}
and
\begin{equation}
 \int_{B_1}[bF(u)u+aG(v)v]<\infty.
 \label{eq:main-local-energy}
\end{equation}
\end{theorem}

The response assumptions concern their full positive range.
They require a lower logarithmic derivative and a positive lower
ratio between each primitive and its endpoint product, but no upper
logarithmic derivative. Both conditions are preserved by amplitude
normalization. The constants $\mu_F,\mu_G$ may be any verified
positive lower bounds; their optimal values are the infima of the
displayed primitive ratios. Necessarily $\Lambda\ge\lambda$, and
for the Riesz kernel $K_\sigma$ one takes $\Lambda=\lambda$.
For $F(s)=s^q$, $G(s)=s^p$, the choices
$\mu_F=1/(q+1)$, $\mu_G=1/(p+1)$ recover the weighted power
condition. In particular, no separate bound
$A_-,B_->-\sigma$ is imposed. The local energy condition concerns
each putative pair, rather than a uniform bound over a class.
The Dini condition is sufficient; no necessity for nonexistence is
asserted.

Our second theorem concerns the actual Green kernel of
$L_V=-\Delta+V$ on $\mathbb R^n$. Its assumptions control the
positive part of the dilation defect, rather than requiring a small
Newton norm of $V$. For $n\ge3$, put $\lambda=n-2$,
$N(x)=c_{n,2}|x|^{-\lambda}$ and
\begin{equation}
 C_\lambda=2^{\max\{\lambda-1,0\}},\qquad
 \mathcal B(h)=2C_\lambda\|I_2|h|\|_\infty,\qquad
 Q=2V+x\cdot\nabla V.
 \label{r7:newton-budget}
\end{equation}
Let $G_V$ be the minimal positive Green kernel, obtained by increasing
Dirichlet balls. With $Q_+=\max\{Q,0\}$, define
\begin{equation}
 \mathcal N_V=\sup_{x\ne y}
 \frac{\int G_V(x,z)Q_+(z)G_V(z,y)\,dz}{G_V(x,y)}.
 \label{r7:green-defect-budget}
\end{equation}

\begin{theorem}
\label{r7:green-local-energy}\label{r8:longrange-local-energy}
\label{r7:green-primitive-local-energy}
Let $n\ge3$ and $0\le V\in C^1(\mathbb R^n)$ satisfy
\begin{equation}
 A_V:=\sup_x|x|^2V(x)<\infty,\qquad
 \mathcal B(V)<\infty,\qquad \mathcal B(|Q|)<\infty.
 \label{r8:annular-potential-class}
\end{equation}
Let $F,G$ be continuous on $[0,\infty)$, locally Lipschitz
on $(0,\infty)$, positive there, and zero at zero. Assume the
full-range response bounds \eqref{r7:response-assumptions}, with
$p_0q_0>1$ and $\mu_F,\mu_G>0$. Let $a,b>0$ be continuous off zero,
satisfy the global weak Euler bounds \eqref{v5:euler-lower}
with $A_-=A$, $B_-=B$ for real $A,B$, and obey the normalized Dini orbit
condition \eqref{eq:front-dini-orbit}. The Euler inequalities hold
on all of $\mathbb R^n\setminus\{0\}$ with arbitrary real $A,B$;
the Dini condition is required only along sufficiently large dilations. If
\begin{equation}
 \Gamma_G:=(n+B)\mu_F+(n+A)\mu_G-(n-2)>\mathcal N_V,
 \label{r7:green-primitive-gap}
\end{equation}
there is no finite strictly positive pair $u,v$, continuous off zero,
for which
\[
 u(x)=\int G_V(x,y)a(y)G(v(y))\,dy,\qquad
 v(x)=\int G_V(x,y)b(y)F(u(y))\,dy
 \quad(x\ne0),
\]
and
\[
 \int_{B_1}[bF(u)u+aG(v)v]<\infty.
\]
% The representations are required pointwise at every nonzero point.
% Only local nonlinear energy is assumed; no upper response derivative,
% global energy, growth bound or additional source-tail hypothesis is
% imposed on the pair.
\end{theorem}

The proof is given in Section~\ref{r7:green-section}. The budget
with $|Q|$ ensures absolute integrability, whereas only $Q_+$ enters
the strict nonexistence gap. The hypotheses do not require compact
support or small $\mathcal B(V)$.

For the pure-power Riesz specialization we continue to write
\begin{equation}
 \gamma_-:=\frac{n+A_-}{p+1}+\frac{n+B_-}{q+1}-(n-\sigma).
 \label{v5:master-gap}
\end{equation}
A pair extending continuously and finitely to zero satisfies the local
energy condition once its positive potential identities are known.
The consequences below specify separately the representation inputs
needed for differential equations.

For $s\ge0$ we use the tail class
\begin{equation}
 \mathcal L_s=\left\{w\in L^1_{\rm loc}(\mathbb R^n):
 \int\frac{|w(x)|}{1+|x|^{n+s}}\,dx<\infty\right\}.
 \label{ho:tail}
\end{equation}
For $w\in\mathcal L_\sigma$, the distribution $L_\sigma w$ is
defined by $\langle L_\sigma w,\phi\rangle=\int wL_\sigma\phi$.
The test-function estimates proved below justify this pairing.
At even integer orders it also has the ordinary compact-test meaning,
which requires no tail assumption.

\begin{corollary}
\label{cor:pde-front}
Specialize Theorem~\ref{r7:master-front} to $F(s)=s^q$, $G(s)=s^p$ and
$K=K_\sigma$, so that $\Lambda=n-\sigma$ and
$\gamma_->0$. Under its continuous-coefficient, weak Euler and
Dini hypotheses, there is no nonnegative pair, strictly
positive off zero, solving
\[
 L_\sigma u=av^p,\qquad L_\sigma v=bu^q
 \quad\text{off zero}
\]
in any of the following classes:
\begin{enumerate}[label=(\roman*)]
\item $p,q\ge1$, $u,v\in C(\mathbb R^n)\cap\mathcal L_\sigma$,
with the equations interpreted distributionally;
\item $0<\sigma<\min(2,n)$, $p,q>0$,
$u,v\in C(\mathbb R^n)\cap\mathcal L_\sigma$, with distributional equations;
\item $\sigma=2m<n$, $p,q\ge1$,
$u,v\in C(\mathbb R^n)\cap C^{2m}_{\rm loc}(\mathbb R^n\setminus\{0\})$,
with pointwise equations off zero and no initial tail.
\end{enumerate}
% In all cases $pq>1$. The indicated continuity, tails and classical
% regularity are retained.
\end{corollary}

\begin{corollary}
\label{thm:integer-front}
Let $m\ge1$ be an integer, $2m<n$, $A,B\in\mathbb R$, and
$p,q\ge1$ with $pq>1$. If
\begin{equation}
 \frac{n+A}{p+1}+\frac{n+B}{q+1}>n-2m,
 \label{eq:front-hyperbola}
\end{equation}
there are no functions
\[
 u,v\in C(\mathbb R^n)\cap
 C^{2m}_{\rm loc}(\mathbb R^n\setminus\{0\})
\]
that are nonnegative on $\mathbb R^n$, strictly positive off zero,
and solve
\[
 (-\Delta)^m u=|x|^A v^p,\qquad
 (-\Delta)^m v=|x|^B u^q
 \quad\text{on }\mathbb R^n\setminus\{0\}.
\]
% No tail, global boundedness, growth condition, global nonlinear
% energy bound, or sign assumption on intermediate iterates is imposed.
\end{corollary}

Continuity at the origin and the restrictions $p,q\ge1$ are part of
this corollary. In particular, its statement does not cover singular
origins or a sublinear component. The proof uses compact tests at the
integer order to obtain source estimates before a potential
representation is invoked. It therefore does not rely on a maximum
principle for a higher-order operator.

% The hypotheses retained by the different conclusions are summarized
% below. Exact potential identities in both main theorems hold at every
% nonzero point. The PDE representation results provide those identities
% only in their stated differential solution classes.
% \begin{center}
% \small
% \setlength{\tabcolsep}{4pt}
% \renewcommand{\arraystretch}{1.15}
% \begin{tabular}{@{}p{.25\linewidth}p{.69\linewidth}@{}}
% \toprule
% Conclusion & Retained hypotheses and scope\\
% \midrule
% Theorem~\ref{r7:master-front}
% & Radial kernel class, global weak Euler bounds, a Dini orbit at large
% scales, exact potential identities and finite local nonlinear energy;
% no upper response derivative bound.\\
% Theorem~\ref{r7:green-local-energy}
% & The same coefficient and response conditions, actual Green
% representations and local energy, with the potential budgets and
% $\Gamma_G>\mathcal N_V$.\\
% Corollary~\ref{cor:pde-front}
% & The three stated power-source PDE classes retain their respective
% origin continuity, tail, order and exponent restrictions.\\
% Theorem~\ref{r7:nonlinear-pde-representation}
% & General-response representation allows unequal orders and requires
% $A>-\sigma_1$, $B>-\sigma_2$ and one of its three PDE classes.
% Corollary~\ref{r7:nonlinear-pde-nonexistence} uses equal orders.\\
% Theorem~\ref{r7:universal-interior-bound}
% & Even integer order and fixed two-sided elasticity windows, with
% lower endpoints at least one and a strict window gap; the upper
% bounds supply compactness of the normalized response family.\\
% \bottomrule
% \end{tabular}
% \end{center}

\subsection{Previous results and the present assumptions}
The classical Lane--Emden problem has a long history. Rellich-type
identities and early nonexistence results appear in
\cite{Mitidieri1993,SerrinZou1996}; the connection with a priori
estimates was developed in \cite{PolacikQuittnerSouplet2007}, and
Souplet~\cite{Souplet2009} proved the four-dimensional conjecture.
For power weights, the radial problem was studied by
Bidaut-V\'eron and Giacomini~\cite{BidautVeronGiacomini2010}.
Phan and Souplet~\cite{PhanSouplet2012} treated
the scalar Hardy--H\'enon equation, whereas
Phan~\cite{Phan2012Systems} established system results with explicit
dimension, growth and weight restrictions.
Fazly and Ghoussoub~\cite{FazlyGhoussoub2014} proved
three-dimensional bounded-solution results and studied stable
solutions of the weighted system. Further three-dimensional
H\'enon--Lane--Emden results were obtained by Li and
Zhang~\cite{LiZhang2019}. For solutions stable outside a compact set,
Huang and Zou~\cite{HuangZou2026} established further
Liouville-type theorems under dimensional or parameter restrictions.

At integer higher orders, radial and componentwise subcritical
results were obtained by Liu, Guo and Zhang~\cite{LiuGuoZhang2006}.
Arthur, Yan and Zhao~\cite{ArthurYanZhao2014} established the full
unweighted strict-subcritical range for $p,q\ge1$ in dimensions
$n=2m+1$ and $n=2m+2$, with an additional scaling condition in
higher dimensions. Weighted results include \cite{Fazly2014,Phan2015,ArthurYan2016}.
These earlier results use different exponent, weight and solution-class
conditions. Several already dispense with prescribed decay or
boundedness. The power applications below are consequences of the
common potential theorem. The principal extensions here concern the
Euler--Dini coefficient and primitive-response conditions, and the
localized pressure argument for the actual Green kernel.

For fractional systems, extension and moving-plane arguments give
nonexistence in exponent subregions \cite{QuaasXia2016}, while
nonlocal average estimates and scaling spheres give further results
for fractional and higher-order systems \cite{CaoQin2021}.
The latter results include individual scaling-index criteria and
separate upper bounds on the two exponents; they do not give the
entire strict-subcritical hyperbola.
The scalar H\'enon--Hardy theory, including classes of variable
nonlinearities, was developed through scaling spheres in
\cite{DaiQin2023}; see also \cite{Yang2021,Yang2026} for higher-order
fractional scalar equations. Xu and Luo~\cite{XuLuo2026} study mixed-power higher-order systems;
their integer-order Liouville result for all positive powers assumes
radial symmetry and $u,v=O(|x|^2)$ at infinity.

The passage from differential to integral equations also has
substantial precedents. In particular, \cite[Theorems~2 and~3]{ChenLi2013}
already treat unweighted higher-order system representation without
an initial decay assumption, including unequal integer orders under
their exponent and super-polyharmonic hypotheses. Super-polyharmonic and representation
arguments are developed in
\cite{CaristiDAmbrosioMitidieri2008,ChenLi2013,CaoDaiQin2019,Le2021}.
Li and Zhang~\cite{LiZhang2021} study higher-order Hardy--H\'enon
systems through origin regularity and PDE--integral equivalence.
For fractional Hardy--H\'enon systems,
\cite[Lemmas~2.1, 2.2 and~3.1]{LiYuZhang2024} establish
distributional extension, Green minorants and exact representation.
Their solution nonexistence theorem
\cite[Theorem~1.7]{LiYuZhang2024} uses separate upper bounds on the
two powers, whereas the equal-order criterion here uses the weighted
hyperbola. The representation statements below display the inputs
retained in each of the three PDE classes.

Villavert~\cite[Theorem~2 and Corollary~1]{Villavert2015} gives sharp
existence criteria for some double-bounded coefficients, together with
PDE--integral equivalence under suitable hypotheses.  The integral
system in that work places the variable coefficients outside the
potentials, while its differential system has coefficients multiplying
the nonlinear sources.  Its existence quantifier is important: those
criteria do not give the Euler--Dini nonexistence statement for a fixed
coefficient pair considered here.  Related Pohozaev identities obtained
from deformations of fractional Dirichlet energies appear in
\cite{DjitteFallWeth2024}.  Our compact transport calculation is made
for positive potential pairs and uses radial coefficient variation;
its one-dimensional continuous BV chain rules are instances of the
standard theory in \cite[Chapter~3]{AFP2000}.

Non-power nonlinearities and the passage from Liouville theorems to
universal estimates have substantial precedents.
Souplet~\cite[Theorems~2 and~4 in arXiv:2202.02955v2]{Souplet2023}
develops generalized doubling and rescaling for scalar problems;
Theorem~2 includes nonlinearities with controlled variation beyond
regular variation. Quittner and
Souplet~\cite[Theorems~4.1 and~4.2]{QuittnerSouplet2025} extend this
approach to Lane--Emden type systems under strengthened regular
variation assumptions; their Theorem~4.3 treats systems whose
normalized nonlinearities have homogeneous limits.  Their
Liouville results in Section~2 also include nonhomogeneous gradient
systems and special noncooperative systems.
The role of the primitive ratios in the present potential-system
argument is therefore a structural sufficient condition, rather than
a claim that nonhomogeneous nonlinearities or nonlinear rescaling are
new subjects.

For anisotropic nonlocal operators, Pohozaev identities with local
Dirichlet boundary terms were established by Ros-Oton, Serra and
Valdinoci~\cite[Theorem~1.1]{RosOtonSerraValdinoci2017}.
Those bounded-domain identities do not supply the local interval
pressure estimate used here for a potential pair of possibly infinite
total interaction energy.  For Green kernels, the three-kernel
perturbation budget is classical: Hirata~\cite[Theorems~3.4 and~4.1]{Hirata2006}
relates it to Green-operator bounds and generalized
Cranston--McConnell inequalities in uniform cones.  Such bounds
control positive potential compositions, but do not by themselves
control the spatial derivative of a kernel under a compact dilation
or prove the local pressure cancellation needed below.

\subsection{Proof Strategy and Outline}
The localized interaction identity and the interval estimates for
source superlevel sets are inherited from \cite{OpenAI2026}, which treats the
Newtonian pure-power weighted problem, and from \cite{MeiWei2026},
which implements this framework for the unweighted Riesz kernel of
order $0<\sigma<2$. In particular, the case $m=1$ with power
weights is already covered by the former work. We prove the
coefficient and representation statements needed for the Dini
orbits and the orders $0<\sigma<n$ used here, while retaining the
inherited interval mechanism.

The coefficient modulus enters through a uniform bound on the error
from interval endpoints. Positive smoothing and an inverse-function
construction produce radial cutoffs satisfying this bound for every
Dini modulus. On the PDE side, compact tests at
integer orders and a finite source-moment iteration at general
orders provide the representation inputs independently of pressure.
The higher-order power-tail cutoff estimate originates in
\cite{AoGonzalezHyderWei2022}; we use the formulation in
\cite[Lemma~2.3]{Yang2021}, with its rescaling and limiting argument
made explicit below. The strict gap is retained: the unweighted
diagonal critical equation has the familiar integral bubbles studied
in \cite{Li2004,ChenLiOu2006}, so a general nonexistence assertion on
the critical hyperbola is not intended.

% \paragraph{Proof architecture and the two kernel settings.}
Theorem~\ref{r7:master-front} is the common radial-potential
theorem. The transport identity in Lemma~\ref{r6:transport-virial}
requires neither differentiable solutions nor differentiable
coefficients. Lemma~\ref{r7:geometric-budgets} isolates the geometric
error bounds; Proposition~\ref{r6:general-pressure} supplies the
response primitive and the exact kernel direction term.
Section~\ref{gd:section} realizes the endpoint budget for every
Dini modulus. Lemma~\ref{r7:source-parameters} first obtains the
parameter signs from finite positive potentials, and
Lemma~\ref{r7:normalization} gives the amplitude and source
normalization. One energy amplification then completes the proof.
Power responses, mixtures, nonhomogeneous radial kernels and the
coefficient examples follow from this single theorem.

The second main result, Theorem~\ref{r7:green-local-energy},
treats the actual Green kernel of $-\Delta+V$ under
\eqref{r8:annular-potential-class}. It is proved in
Section~\ref{r7:green-section}, before the PDE representation
arguments. The passage to this non-translation-invariant kernel
requires more than comparison with the Newton kernel: a relative
short-interval endpoint difference restores the true primitive
pressure, while the exact dilation correction cancels the virtual
Newton-energy excess. The remaining positive defect is measured by
the relative Green composition budget for
$Q_+=(2V+x\cdot\nabla V)_+$. Its power, small-potential and
entire PDE consequences are given with their respective assumptions.

Section~\ref{sec:paper1-pde-representation} collects the independent
PDE reduction in one place: linear source inversion, the three
power-source representation classes, and the general-response
representation theorem. The power routes include wider exterior-weight
and unequal-order ranges; the nonlinear theorem has its own global
Euler and Dini hypotheses.
Theorem~\ref{r7:nonlinear-pde-representation} allows unequal
orders at the representation stage;
Corollary~\ref{r7:nonlinear-pde-nonexistence} uses equal orders
for nonexistence. Finally,
Theorem~\ref{r7:universal-interior-bound} gives interior bounds
for a compact normalization family without regular variation.
Its upper response derivative bound is needed for compactness
and is absent from the potential and Green nonexistence theorems.

\section{Source mass estimates and local regularity away from zero}
\label{sec:regularity}
\begin{lemma}
\label{so:automatic-weights}\label{so:unit-mass}
Let $0<\sigma_1,\sigma_2<n$, $p,q>0$, $g=av^p$ and $f=bu^q$,
with positive coefficients and finite positive functions on
$\mathbb R^n\setminus\{0\}$ satisfying
$u\geq I_{\sigma_1}g$, $v\geq I_{\sigma_2}f$ there.
Then $f,g\in L^1_{\mathrm{loc}}(\mathbb R^n)$ and the infima of $u,v$
on every $B_R\setminus\{0\}$ are positive.
For pure powers $a=|x|^A$, $b=|x|^B$, necessarily $A,B>-n$.
If the pair extends continuously and finitely to zero, then its local
nonlinear energy is finite and, for pure powers,
$u(0),v(0)>0$, $A>-\sigma_1$, $B>-\sigma_2$.

Suppose additionally that $pq>1$ and a fixed patch $D\subset B_1$
has $\int_Da\geq m_a>0$, $\int_Db\geq m_b>0$.
Then $\int_{B_1}(f+g)\leq C$, with $C$ depending only on these
lower bounds, the orders, dimension and powers. 
% No upper coefficient
% bound or nonlinear energy bound is required for this mass estimate.
\end{lemma}
\begin{proof}
At a fixed $x_*\ne0$, each kernel is bounded below on $B_R$ by a
positive constant; finite minorants therefore give finite positive
source masses on every ball. The lower kernel bound on
$B_R\times B_R$ then gives strictly positive infima on $B_R\setminus\{0\}$.
For pure powers this bounds $g,f$ below by constant multiples of
$|x|^A,|x|^B$ on a ball, forcing $A,B>-n$.
If the pair is bounded near zero, multiply the locally integrable sources
by these bounded factors to obtain local energy.
Fatou at zero gives $I_{\sigma_1}g(0)\leq u(0)<\infty$ and the
analogous bound for $v$; annular source mass makes both potentials
strictly positive. The positive continuous values of $v,u$ near zero
then imply finiteness of
$\int_0^1r^{\sigma_1+A-1}dr$ and
$\int_0^1r^{\sigma_2+B-1}dr$, respectively.

For the source mass bound put $F=\int_{B_1}f$, $G=\int_{B_1}g$.
Let $k_i>0$ be the lower bound for $K_{\sigma_i}$ on distances at most
two. The minorants give $u\geq k_1G$, $v\geq k_2F$ on $B_1$.
Thus $F\geq m_bk_1^qG^q$ and $G\geq m_ak_2^pF^p$.
Substitution in both orders bounds $F^{pq-1}$ and $G^{pq-1}$ above,
which proves the assertion. The dependence is uniform on compact
parameter sets with positive lower bounds and $pq-1$ separated from zero.
\end{proof}

The representation arguments will only need continuity of positive
potentials with locally bounded sources. Standard potential estimates
are treated in \cite[Chapter~V]{Stein1970} and \cite{AdamsHedberg1996}.

\begin{lemma}
\label{pd:local-gain}
Let $F\ge0$ be locally integrable, $0<\sigma<n$, and suppose
$P=I_\sigma F$ is finite at one point. If $F$ is locally bounded
on an open set $U$, then $P$ is finite and continuous on $U$.
\end{lemma}
\begin{proof}
Localize the source with a smooth compact cutoff in $U$.
For the bounded near part, local integrability of the Riesz kernel
and continuity of translations in $L^1$ give finiteness and continuity.
On a compact set separated from the remaining source, the kernel is
smooth. Its tails, uniformly on that compact set, are bounded by a
constant times the kernel based at the point of finite potential.
Dominated convergence gives continuity of the remaining part.
\end{proof}

The nonexistence proof below uses transport increments rather than
derivatives of the solution. 
% No $C^1$ bootstrap is needed.

\subsection{Normalized source families}
\label{v5:master-section}\label{sec:closure}\label{sec:sources}

The source bounds in this subsection use potential minorants, but not an energy upper bound or a pressure estimate.

\begin{lemma}
\label{gp:semiorbit-source}
Let $n\geq1$, $0<\sigma<n$, $p,q>0$, $pq>1$, and let
$a,b>0$ be continuous away from zero with the global weak Euler lower bounds
\eqref{v5:euler-lower}.  Let $u,v$ be finite, strictly positive and
continuous off zero, set $g=av^p$, $f=bu^q$, and suppose
$u\geq I_\sigma g$, $v\geq I_\sigma f$ at every nonzero point.
Fix $e_0\in\mathbb S^{n-1}$, $R_*\geq1$ and
$D=B_{1/8}(3e_0/4)$.  Suppose
\begin{equation}
 \inf_{R\geq R_*}\int_D\frac{a(Rx)}{a(Re_0)}\,dx>0,
 \qquad
 \inf_{R\geq R_*}\int_D\frac{b(Rx)}{b(Re_0)}\,dx>0.
 \label{gp:patch-input}
\end{equation}
Then $f,g\in L^1_{\rm loc}$, the normalized sources below have
uniform unit-ball masses, and their large-ball bounds are
\eqref{v5:large-source-masses}.  Every putative pair forces
$\alpha_-,\beta_-\leq\lambda=n-\sigma$, where
\begin{equation}
 \alpha_-:=\frac{\sigma+A_-+p(\sigma+B_-)}{pq-1},\qquad
 \beta_-:=\frac{\sigma+B_-+q(\sigma+A_-)}{pq-1}.
 \label{gp:floor-scales}
\end{equation}
If the strict gap \eqref{v5:master-gap} holds, both exponents are positive.
In original variables, for $R\geq R_*$,
\begin{equation}
 \int_{B_R}g\leq CR^{\lambda-\alpha_-},\qquad
 \int_{B_R}f\leq CR^{\lambda-\beta_-}.
 \label{gp:original-source-bounds}
\end{equation}
% The constant here may depend on the original coefficients and starting
% normalization.  No energy, exact-potential equality or orbit modulus
% is an assumption of this lemma.
\end{lemma}
\begin{proof}
Write $d_0=pq-1>0$.  Lemma~\ref{so:automatic-weights} gives source
local integrability from the finite potential minorants.
\medskip
\noindent\emph{The normalized dilation family.}
Use the fixed $e_0,R_*$ of the lemma.
For every $R\geq R_*$ define
\begin{equation}
 a_R(x)=\frac{a(Rx)}{a(Re_0)},\qquad
 b_R(x)=\frac{b(Rx)}{b(Re_0)},
 \label{v5:normalized-coefficients}
\end{equation}
\begin{equation}
 C_R^{d_0}=R^{\sigma(p+1)}a(Re_0)b(Re_0)^p,
 \qquad
 D_R^{d_0}=R^{\sigma(q+1)}b(Re_0)a(Re_0)^q,
 \label{v5:amplitudes}
\end{equation}
and $U_R=C_Ru(R\cdot)$, $V_R=D_Rv(R\cdot)$.
Then $a_R(e_0)=b_R(e_0)=1$, and the finite-dilation Euler inequalities
\eqref{r6:weak-euler-front} are unchanged.
The balances $D_R^p=C_RR^\sigma a(Re_0)$ and
$C_R^q=D_RR^\sigma b(Re_0)$ give exact source scaling
\begin{equation}
 g_R:=a_RV_R^p=C_RR^\sigma g(R\cdot),\qquad
 f_R:=b_RU_R^q=D_RR^\sigma f(R\cdot).
 \label{v5:source-scaling}
\end{equation}
Thus $(U_R,V_R)$ retains both potential minorants with $(a_R,b_R)$.
This scaling statement uses no local-energy input.

Only further scales $t\geq1$ are used.  They preserve this family since
$Rt\geq R_*$, and
\[
 (a_R)_t=a_{Rt},\qquad (b_R)_t=b_{Rt}.
\]
The relative amplitudes obtained from $(a_R,b_R)$ are exactly
\begin{equation}
 c_R(t):=C_t[a_R,b_R]=\frac{C_{Rt}}{C_R},\qquad
 d_R(t):=D_t[a_R,b_R]=\frac{D_{Rt}}{D_R},
 \qquad c_R(1)=d_R(1)=1.
 \label{v5:relative-amplitudes}
\end{equation}
The original amplitudes $C_{R_*},D_{R_*}$ need not be one.  It is the
relative amplitudes at $t=1$, and the normalized coefficient values at
$e_0$, that are one.  
% This distinction avoids an unauthorized normalization
% of the original coefficient values.

\medskip
\noindent\emph{Source mass bounds and scaling exponents.}
The assumed patch integrals and Lemma~\ref{so:unit-mass} give a
uniform bound for both normalized unit-ball source masses.
The lower kernel bound controls the entire ball even if $n=1$ and
the patch is in only one branch. 
% No comparison of branch values
% or local energy is needed.

Use the scaling exponents \eqref{gp:floor-scales}, without assuming their signs.
For $t\geq1$, the finite-dilation inequalities
\eqref{r6:weak-euler-front} and the explicit amplitudes give
\[
 c_R(t)\geq t^{\alpha_-},\qquad
 d_R(t)\geq t^{\beta_-}\quad(R\ge R_*).
\]
These comparisons require no coefficient derivative.
Apply the unit-ball mass bound to the member at scale $Rt$ and use the
relative source identities.  It follows that
\begin{equation}
 \int_{B_t}g_R\leq Ct^{\lambda-\alpha_-},\qquad
 \int_{B_t}f_R\leq Ct^{\lambda-\beta_-}
 \quad(R\geq R_*,\ t\geq1).
 \label{v5:large-source-masses}
\end{equation}
% No small-scale mass estimate or upper bound on the Euler derivatives has been used.

Positive, nondecreasing source masses in the preceding bounds force
$\alpha_-,\beta_-\leq\lambda$: a larger exponent makes the associated
upper bound tend to zero.  The exact identity
\[
 \alpha_-+\beta_--\lambda
 =\frac{(p+1)(q+1)}{pq-1}\gamma_-
\]
therefore makes both exponents positive when $\gamma_->0$.
Changing variables in the member at $R_*$ gives
\eqref{gp:original-source-bounds}.  This completes the lemma.
% without summing any potential tail or invoking pressure.
\end{proof}

\section{Interaction measure and a transport virial identity}
\label{sec:virial}\label{vi:section}

The proof uses radial flows and one-dimensional bounded variation;
see \cite{AFP2000} for the underlying BV facts.
Deformation-based fractional Pohozaev identities in a Dirichlet setting
are developed in \cite{DjitteFallWeth2024}. We prove the finite
interaction variation needed here, including when the total energy
is infinite.

Throughout this section, $n\ge1$, $0<\sigma<n$, and
$\lambda=n-\sigma$. We first allow a radial kernel $K(z)=k(|z|)$,
where $k\in C^1((0,\infty))$ is positive and
\begin{equation}
 c_0r^{-\lambda}\le k(r)\le C_0r^{-\lambda},\qquad
 0\le\theta_K(r):=-\frac{rk'(r)}{k(r)}\le\Lambda
 \quad(r>0).
 \label{r6:kernel-transport}
\end{equation}
The Riesz kernel has $\theta_K=\lambda$. Let $F,G$ be positive,
strictly increasing locally Lipschitz functions on $(0,\infty)$, and put
\[
 P_F(t)=\int_0^tF(s)\,ds,\qquad P_G(t)=\int_0^tG(s)\,ds.
\]
These primitives are finite, and $0\le P_F(t)\le tF(t)$,
$0\le P_G(t)\le tG(t)$. 
% No pointwise differentiability of the responses is required.

Let $a,b$ be positive continuous functions off zero. Their Euler
lower bounds will be understood in the distributional sense:
\begin{equation}
 \operatorname{div}(xa)\ge(n+A)a\,dx,\qquad
 \operatorname{div}(xb)\ge(n+B)b\,dx
 \quad\hbox{on }\mathbb R^n\setminus\{0\}.
 \label{r6:weak-euler}
\end{equation}
Equivalently, $r^{-A}a(r\omega)$ and $r^{-B}b(r\omega)$ are
nondecreasing along every ray. Indeed, polar disintegration gives this
monotonicity for almost every direction, and continuity extends it to
all directions. In dimension one the two rays are treated separately.
In particular,
\[
 D_Ea:=\operatorname{div}(xa)-na\,dx,\qquad
 D_Eb:=\operatorname{div}(xb)-nb\,dx
\]
are locally finite signed Radon measures off zero, with
$D_Ea\ge Aa\,dx$ and $D_Eb\ge Bb\,dx$. 
% No angular bounded variation is assumed.

Suppose that $u,v$ are finite, strictly positive and continuous off
zero, and satisfy
\begin{equation}
 u=K*g,\qquad v=K*f,\qquad f=bF(u),\qquad g=aG(v)
 \quad\hbox{at every nonzero point}.
 \label{vi:potentials}
\end{equation}
Assume $\int_{B_1}(fu+gv)<\infty$. Finiteness of a positive
potential implies local integrability of its source; hence $f,g$ are
locally integrable on the whole space. The energy densities are also
locally integrable, by the hypothesis and continuity away from zero.

\subsection{The interaction measure and the cutoff}

Define
\begin{equation}
 d\pi(x,y)=f(x)g(y)K(x-y)\,dx\,dy.
 \label{vi:interaction}
\end{equation}
Tonelli's theorem gives
\begin{equation}
 \pi(dx,\mathbb R^n)=f(x)u(x)\,dx,\qquad
 \pi(\mathbb R^n,dy)=g(y)v(y)\,dy.
 \label{vi:marginals}
\end{equation}
Thus both marginals have finite mass on compact sets; the total
interaction mass need not be finite.

Let $\eta\in C_c^2(\mathbb R^n)$ be nonnegative and radial, equal
to one on $B_1$ and supported in $B_2$, and set
$w=-x\cdot\nabla\eta$, $X=x\eta$. Write
\begin{equation}
 \begin{split}
 E_f&=\int\eta fu,\qquad E_g=\int\eta gv,\qquad E=E_f+E_g,\\
 W_f&=\int wfu,\qquad W_g=\int wgv.
 \end{split}
 \label{vi:energies}
\end{equation}
For $x\ne0$ set $\omega_x=x/|x|$, and define
\begin{equation}
 D(x,y)=\frac{((x-y)\cdot\omega_x)^2}{|x-y|^2},\qquad
 \Phi_X(x,y)=\frac{(X(x)-X(y))\cdot(x-y)}{|x-y|^2}.
 \label{vi:directional-factors}
\end{equation}
We also write $\Phi=\Phi_X$. Values on the null exceptional sets
are immaterial. Since $\Phi_X$ is bounded by the Lipschitz constant
of $X$ and vanishes when both endpoints are outside its support,
$\Phi_X\in L^1(\pi)$.

\subsection{Increment estimates along a radial flow}

\begin{lemma}
\label{r6:transport-increments}
Let $Y=x\zeta(x)$, with $\zeta\in C_c^2(\mathbb R^n\setminus\{0\})$
radial, and let $T_t$ be its flow. Choose a compact annulus $Q$
containing the support of $Y$ in its interior, so that $T_t(Q)=Q$.
For $\delta_th=h\circ T_t-h$, one has, as $t\to0$,
\begin{equation}
 \|\delta_tu\|_{L^2(Q)}+\|\delta_tv\|_{L^2(Q)}
 \le C_Y\bigl(\|\delta_ta\|_{L^2(Q)}
              +\|\delta_tb\|_{L^2(Q)}+|t|\bigr).
 \label{r6:transport-absorption}
\end{equation}
Consequently, the squares of the $L^2(Q)$ increments of $a,b,u,v,f,g$
are $o(|t|)$. The same conclusion holds for the density increments
$\Delta_tf=(T_t)_\#f-f$ and $\Delta_tg=(T_t)_\#g-g$ of their
mass pushforwards.
\end{lemma}

\begin{proof}
All constants in this proof may depend on the pair and on $Q$.
Unsubscripted $L^2$ norms below are taken on this invariant annulus;
source increments supported there are extended by zero outside $Q$.
The coefficients have finite variation along radial lines, with finite
integrated radial variation on compact annuli. The one-dimensional BV
translation estimate, applied to the radial flow, therefore gives
\[
 \|\delta_ta\|_{L^1(Q)}+\|\delta_tb\|_{L^1(Q)}\le C_Y|t|.
\]
One can obtain the same estimate by integrating the variation measures
along the flow; its radial Jacobians are uniformly bounded on a fixed
slightly larger annulus. Uniform continuity there gives
$\|\delta_ta\|_\infty+\|\delta_tb\|_\infty=o(1)$, so
\begin{equation}
 \|\delta_ta\|_2^2+\|\delta_tb\|_2^2=o(|t|).
 \label{r6:coefficient-square-increments}
\end{equation}

Write $J_t=\det DT_t$. The flow satisfies
$J_t=1+O(|t|)$ and
$|T_tx-T_ty|/|x-y|=1+O(|t|)$ uniformly in $x,y$.
The logarithmic kernel bound in \eqref{r6:kernel-transport} yields
\[
 |K(T_tx-T_ty)J_t(y)-K(x-y)|\le C_Y|t|K(x-y).
\]
Changing variables in the potential identity and comparing these
positive kernels gives, uniformly for $x\in Q$,
\begin{equation}
 \delta_tu=K*(\delta_tg)+O_Y(|t|),\qquad
 \delta_tv=K*(\delta_tf)+O_Y(|t|).
 \label{r6:flow-commutator}
\end{equation}
For example, the first error is bounded by
$C|t|K*(g\circ T_t)(x)\le C'|t|u(T_tx)$, which is bounded on $Q$.

Fix a small $r>0$ and split $K=K_r+S_r$, with
$K_r(z)=K(z)\chi(|z|/r)$, where $\chi$ is smooth, equals one near
zero and vanishes outside $B_2$. Then
$\|K_r\|_1\le Cr^\sigma$, while $S_r\in C^1(\mathbb R^n)$.
The increments of the sources vanish outside $Q$. For $x\in Q$,
changing variables once more gives
\[
 S_r*(\delta_tg)(x)=\int g(y)
 \bigl[S_r(x-T_{-t}y)J_{-t}(y)-S_r(x-y)\bigr]dy.
\]
The bracket vanishes outside $Q$, and on $Q$ it is $O_{Y,r}(|t|)$,
uniformly in $x$. Thus this expression is $O_{Y,r}(|t|)$.
No derivative of $g$ is used. The same holds for $f$.
The positive functions $a,b,u,v$ have positive lower bounds and finite
upper bounds on $Q$. Local Lipschitz continuity of $F,G$ consequently
gives
\[
 \|\delta_tg\|_2\le C(\|\delta_tv\|_2+\|\delta_ta\|_2),\qquad
 \|\delta_tf\|_2\le C(\|\delta_tu\|_2+\|\delta_tb\|_2).
\]
Apply Young's inequality to the near kernels in
\eqref{r6:flow-commutator}, and choose $r$ so that the two terms
$Cr^\sigma\|\delta_tu\|_2$, $Cr^\sigma\|\delta_tv\|_2$ can be
absorbed. This proves \eqref{r6:transport-absorption};
\eqref{r6:coefficient-square-increments} proves the assertions for
composition increments. Finally,
$(T_t)_\#f=(f\circ T_{-t})J_{-t}$, and similarly for $g$;
$J_{-t}=1+O(|t|)$ proves the pushforward assertions.
\end{proof}

\subsection{The primitive virial inequality}

\begin{lemma}
\label{r6:transport-virial}
Under the standing assumptions,
\begin{equation}
 \begin{split}
 &\int\bigl[(n+B)\eta-w\bigr]bP_F(u)
  +\int\bigl[(n+A)\eta-w\bigr]aP_G(v)\\
 &\hspace{20mm}\le
 \iint\theta_K(|x-y|)\Phi_X(x,y)\,d\pi(x,y).
 \end{split}
 \label{r6:primitive-virial}
\end{equation}
If $Y=x\zeta$ is supported away from zero, there is the exact
measure identity
\begin{equation}
 \int P_F(u)\,d\operatorname{div}(bY)
 +\int P_G(v)\,d\operatorname{div}(aY)
 =\iint\theta_K(|x-y|)\Phi_Y(x,y)\,d\pi.
 \label{r6:measure-virial}
\end{equation}
% In particular, no differentiability of $u,v$ is required.
\end{lemma}

\begin{proof}
First fix such a field $Y$ and its flow $T_t$. Use the finite
interaction increment
\[
 \mathcal D_t=\iint f(x)g(y)
       \bigl[K(T_tx-T_ty)-K(x-y)\bigr]dxdy.
\]
Its integrand is bounded in absolute value by
$C|t|K(x-y)$ times the indicator that at least one endpoint is in
$Q$. The marginals therefore make $\mathcal D_t$ absolutely
convergent, even if the total interaction mass is infinite.
Dominated convergence yields
\begin{equation}
 \lim_{t\to0}\frac{\mathcal D_t}{t}
 =-\iint\theta_K(|x-y|)\Phi_Y(x,y)\,d\pi.
 \label{r6:interaction-variation}
\end{equation}
Let $\Delta_tf=(T_t)_\#f-f$, $\Delta_tg=(T_t)_\#g-g$.
Changing variables gives the finite expansion
\begin{equation}
 \mathcal D_t=\int u\Delta_tf+\int v\Delta_tg
                 +\int\Delta_tf\,(K*\Delta_tg).
 \label{r6:finite-interaction-expansion}
\end{equation}
For rigor, one may first perform the changes of variables in a large
ball containing $Q$, which is invariant under the flow, and then use
the compact-endpoint marginals to pass to the whole space.
All three terms on the right are absolutely convergent. Since the
source increments are supported in $Q$, the kernel can be restricted to
$Q-Q$, where $\|K\mathbf1_{Q-Q}\|_1<\infty$. Young's inequality and
Lemma~\ref{r6:transport-increments} bound the last term by
\[
 C_Q\|\Delta_tf\|_2\|\Delta_tg\|_2=o(|t|).
\]

For the first term, the definition of a pushforward gives
$\int u\Delta_tf=\int f(u\circ T_t-u)$. The values of $u$ and
$u\circ T_t$ on $Q$ lie in a fixed positive compact interval.
The quadratic Taylor estimate for the locally $C^{1,1}$ primitive $P_F$ gives
\[
 \int u\Delta_tf
 =\int b\bigl[P_F(u\circ T_t)-P_F(u)\bigr]
       +O(\|\delta_tu\|_2^2).
\]
The error is $o(|t|)$. After changing variables, the remaining
integral is
\[
 \int P_F(u(y))
       \bigl[b(T_{-t}y)J_{-t}(y)-b(y)\bigr]dy.
\]
The measure difference quotients in brackets converge weakly* to
$-\operatorname{div}(bY)$. To justify continuous, rather than smooth,
test functions here, the radial BV estimate provides uniform total
variation bounds for the quotients; distributional convergence then
extends by uniform approximation on $Q$. Thus
\[
 \lim_{t\to0}\frac1t\int u\Delta_tf
       =-\int P_F(u)\,d\operatorname{div}(bY).
\]
The second term in \eqref{r6:finite-interaction-expansion} is treated
identically. Combining these limits with
\eqref{r6:interaction-variation} proves
\eqref{r6:measure-virial}. This argument uses finite interaction
increments throughout, not a subtraction of two infinite energies.

To remove the origin cutoff, choose smooth $\chi:[0,\infty)\to[0,1]$
equal to zero on $[0,1]$ and one on $[2,\infty)$, and set
$\chi_\rho(x)=\chi(|x|/\rho)$, $X_\rho=x\eta\chi_\rho$.
Since $X=O(|x|)$ near zero,
\begin{equation}
 \|\nabla X_\rho\|_\infty+
 \|\nabla(X_\rho-X)\|_\infty\le C.
 \label{vi:uniform-lipschitz}
\end{equation}
Apply \eqref{r6:measure-virial} to $X_\rho$. The product rule for
radial derivative measures gives
\begin{equation}
 \operatorname{div}(bX_\rho)
 =\eta\chi_\rho D_Eb
 +b\bigl[(n\eta-w)\chi_\rho
            +\eta x\cdot\nabla\chi_\rho\bigr]dx,
 \label{vi:weighted-ibp}
\end{equation}
with the analogous formula for $a$. Substitute the measure lower
bounds before taking $\rho\downarrow0$. This is legitimate because
$\eta\chi_\rho P_F(u)$ and $\eta\chi_\rho P_G(v)$ are nonnegative.
The resulting origin errors obey
\begin{equation}
 \left|\int\eta\bigl[bP_F(u)+aP_G(v)\bigr]
                         x\cdot\nabla\chi_\rho\right|
 \le C\int_{B_{2\rho}}(fu+gv)\longrightarrow0.
 \label{vi:origin-ibp-error}
\end{equation}
The other marginal terms converge by $P_F(u)\le uF(u)$ and
$P_G(v)\le vG(v)$. Finally,
\begin{equation}
 \iint\theta_K(|x-y|)|\Phi_{X_\rho}-\Phi_X|\,d\pi
 \le C\Lambda\int_{B_{2\rho}}(fu+gv)\longrightarrow0,
 \label{vi:origin-pair-error}
\end{equation}
because $X_\rho-X$ is supported in $B_{2\rho}$ and has a uniformly
bounded Lipschitz norm. These limits prove
\eqref{r6:primitive-virial}. They do not require the Euler measures
themselves to have finite mass at the origin.
\end{proof}

\begin{lemma}
\label{vi:virial-lemma}
Let $K=K_\sigma=c_{n,\sigma}|z|^{\sigma-n}$,
$F(t)=t^q$, $G(t)=t^p$, with $p,q>0$. Under the above positive
continuity, distributional Euler lower bounds and local energy
assumptions, $\Phi\in L^1(\pi)$ and
\begin{equation}
 \frac{n+B}{q+1}E_f+\frac{n+A}{p+1}E_g
 -\frac{W_f}{q+1}-\frac{W_g}{p+1}
 \le(n-\sigma)\iint\Phi(x,y)\,d\pi(x,y).
 \label{vi:virial-identity}
\end{equation}
For $a=|x|^A$, $b=|x|^B$, this is an equality. The same conclusion
holds for dilations of the stated cutoff. In particular, the lemma
applies to the original $C^1$ coefficients and pointwise Euler bounds.
\end{lemma}
\begin{proof}
Use $P_F(t)=t^{q+1}/(q+1)$,
$P_G(t)=t^{p+1}/(p+1)$ and $\theta_K=n-\sigma$ in
Lemma~\ref{r6:transport-virial}. For the coefficient powers
$a=|x|^A$ and $b=|x|^B$, the measure lower bounds are equalities,
and every limiting estimate above removes an
error tending to zero, so the resulting formula is exact.
\end{proof}

\section{Interval estimates and source superlevel sets}
\label{sec:pr-pressure}\label{sec:pressure}

We adapt the interval pressure argument of
\cite{OpenAI2026,MeiWei2026} to normalized coefficient orbits.
The interval comparison, directional averaging and layer aggregation
are proved here, with the coefficient error isolated in
\eqref{pr:endpoint-error}. Section~\ref{gp:section} applies these
estimates to an adapted radial cutoff.
Throughout, $n\geq1$, $0<\sigma<n$, $\lambda=n-\sigma$, and
$K_\sigma(z)=c_{n,\sigma}|z|^{-\lambda}$, with $c_{n,\sigma}>0$.
The nonlinear and source assumptions are
\begin{equation}
 f=bu^q,\qquad g=av^p,\qquad u=K_\sigma*g,\qquad v=K_\sigma*f
 \quad\hbox{at every nonzero point},
 \label{eq:pr-potentials}
\end{equation}
with positive coefficients continuous away from zero, $p,q>0$, finite
positive $u,v$ continuous away from zero, and $f,g,fu,gv\in L^1_{\rm loc}$.
When required, a common $M\geq1$ supplies
\begin{equation}
 \int_{B_4}(f+g+u+v)\leq M,
 \label{eq:pr-source-input}
\end{equation}
\begin{equation}
 \int_{|x-y|\geq s}K_\sigma(x-y)(f(y)+g(y))\,dy
 \leq Ms^{-\lambda}\quad(x\in B_2,\ 0<s\leq1).
 \label{eq:pr-tail-input}
\end{equation}
\subsection{Interval comparison, including coincident lines}

For $\rho\geq0$ define, off its possible diagonal singularity,
\begin{equation}
 \mathcal K_\rho(s,s')
 =c_{n,\sigma}(\rho^2+(s-s')^2)^{-\lambda/2},\qquad
 \mathcal D_\rho(s,s')
 =\frac{(s-s')^2}{\rho^2+(s-s')^2}.
 \label{eq:pr-line-kernel}
\end{equation}
For $\rho=0$, put $\mathcal D_0=1$ off the diagonal. For bounded
$I=(s_-,s_+)$ let
$\mathcal K_{\rho,I}(s')=
[\mathcal K_\rho(s_-,s')+\mathcal K_\rho(s_+,s')]/2$, and define
$\mathcal K_{\rho,J}$ analogously.

\begin{lemma}
\label{lem:pr-interval}
\label{lem:pr-interval-one}
Assume $\lambda>0$, $\rho\geq0$, with the additional restriction
$\lambda<1$ if $\rho=0$. Let $I,J$ be nonempty open intervals, possibly
unbounded, and $a_I,a_J\geq0$, with each interval bounded whenever its
coefficient is positive. Then
\begin{equation}
 \begin{split}
 \int_{I\times J}(a_I+a_J)\mathcal K_\rho
 \leq\int_{I\times J}\bigl[&a_I\mathcal K_{\rho,I}
             +a_J\mathcal K_{\rho,J}\\
 &+(\lambda\min\{a_I,a_J\}\mathcal D_\rho
                         +|a_I-a_J|)\mathcal K_\rho\bigr].
 \end{split}
 \label{eq:pr-interval}
\end{equation}
% All terms are nonnegative extended integrals. A zero-coefficient endpoint
% term is defined as zero without evaluating its endpoints.
\end{lemma}
\begin{proof}
If at most one coefficient is positive, the mismatch term already equals
the left integrand. Suppose both are positive, so $I,J$ are bounded.
First take $\rho>0$, for which all the following integrals are finite.
Let $m_I,m_J$ be the interval midpoints, $d_0=m_I-m_J$, and $h=s-s'$.
Integration by parts gives
$\int_I(\mathcal K_\rho-\mathcal K_{\rho,I})\,ds
=-\int_I(s-m_I)\partial_s\mathcal K_\rho\,ds$.
Adding the counterpart for $J$, using the two kernel derivatives, yields
\begin{equation}
 \int_{I\times J}(2\mathcal K_\rho-
               \mathcal K_{\rho,I}-\mathcal K_{\rho,J})
 =\lambda\int_{I\times J}
       \frac{h^2-d_0h}{\rho^2+h^2}\mathcal K_\rho.
 \label{eq:pr-interval-identity}
\end{equation}
The density of $h$ is $\varrho(h)=|I\cap(J+h)|$. It is a function
of $|h-d_0|$, nonincreasing in that variable, because it is the translated
convolution of two centered interval indicators. If $d_0\geq0$ and
$h>0$, then $\varrho(h)\geq\varrho(-h)$; pairing the two signs
against the odd function $h\mathcal K_\rho(h)/(\rho^2+h^2)$ gives
\[
 d_0\int_{I\times J}\frac{h}{\rho^2+h^2}\mathcal K_\rho\geq0.
\]
The case $d_0\leq0$ follows by reversing the variables. Hence the
right side of \eqref{eq:pr-interval-identity} is at most
$\lambda\int\mathcal D_\rho\mathcal K_\rho$.
Multiply by $\min\{a_I,a_J\}$. Each remaining defect is at most
$\int\mathcal K_\rho$, since its endpoint term is nonnegative;
the remaining coefficients sum to $|a_I-a_J|$. This proves the claim
for $\rho>0$.

For $\rho=0$ and $0<\lambda<1$, apply the proved inequality with
$\rho=r>0$ and let $r\downarrow0$. The kernel, both endpoint
averages, and $\mathcal D_r\mathcal K_r$ increase to their stated
limits almost everywhere. Monotone convergence proves the claim. The
condition $\lambda<1$ makes diagonal and endpoint singularities locally
integrable.
% no singular kernel is differentiated and no integrability
% of the line kernel at infinity is used.
\end{proof}

In the actual slicing below, $\rho>0$ occurs on distinct parallel lines
when $n\geq2$. In dimension one there is a single line and
$\rho=0$, with $\lambda=1-\sigma\in(0,1)$, so the same lemma applies.

\subsection{The exact directional normalizer}

Use normalized surface probability on $\mathbb S^{n-1}$; when $n=1$
this is the measure assigning probability $1/2$ to each of $e=\pm1$.
Write $\mathbb E_e,\mathbb P_e$ for expectation and probability.
For $n\geq2$, choose a fixed smooth $0\leq\psi\leq1$ that equals
one near zero and vanishes outside $[0,1]$, and put
\begin{equation}
 Z(r)=\mathbb E_e\psi(|e-\vartheta|/r),\qquad
 cr^{n-1}\leq Z(r)\leq Cr^{n-1},\qquad |Z'(r)|\leq Cr^{n-2}.
 \label{eq:pr-normalizer}
\end{equation}
Rotations make $Z$ independent of $\vartheta\in\mathbb S^{n-1}$.
The bounds follow in chordal polar coordinates: for $0<r\leq1$ the
spherical density is comparable to $s^{n-2}\,ds$ on $0<s<r$;
differentiation of $\psi(s/r)$ costs at most $Cs/r^2$.

\subsection{Measurability and integration over superlevel sets}

Let $0<\ell\leq1/4$ on $B_2$ be continuous, extended by zero
outside, and let $h_e\geq0$ be a jointly Borel measurable family of
continuous bounded weights vanishing outside $1<|x|<2$.
For a component $I(e,x,t)$ of the open superlevel set
$\{z\ne0:f(z)>t\}$ on $x+\mathbb Re$, put
\begin{equation}
 \chi_I=
 \begin{cases}
 \inf_Ih_e,&I\text{ bounded and }|I|\leq2\inf_I\ell,\\
 0,&\text{otherwise}.
 \end{cases}
 \label{pr:retention}
\end{equation}
Write $\chi_f(e,x,t)=\chi_{I(e,x,t)}$, and similarly $\chi_g$.
Each forward exit distance exceeds $a\geq0$ exactly when the compact
segment of length $a$ lies in the open superlevel set.  This is an open condition
in $(e,x,t)$, since that segment avoids zero and has a positive minimum
of $f-t$.  Both exit distances, their sum and finite endpoints are
therefore measurable.  Component infima are countable infima over
rational offsets between the exits.  This proves joint measurability
without enumerating components.  Zero-weight endpoint products are
zero without selecting an endpoint.

If $\chi_I>0$, continuity makes $h_e\geq\chi_I$ on its closure,
so the entire closed component lies in $1<|x|<2$.
Its endpoint source values are $t$, by maximality and continuity.
For a finite nondecreasing modulus $\Omega$ assume
\begin{equation}
 \begin{split}
 |\log a(x)-\log a(z)|&\leq H_a\Omega(|x-z|),\\
 |\log b(x)-\log b(z)|&\leq H_b\Omega(|x-z|)
 \end{split}
 \label{pr:annular-modulus}
\end{equation}
\label{ld:modulus}
within each annular component.  With
$u_I=[u(x_-)+u(x_+)]/2$ and $\tau_f=(t/b(x))^{1/q}$,
retention gives $|x_\pm-x|\leq2\ell(x)$ and hence
\begin{equation}
 |u_I-\tau_f|\leq C\Omega(2\ell(x))u(x),\qquad u_I\leq Cu(x).
 \label{pr:endpoint-error}
\end{equation}
Indeed the endpoint value is $(t/b(x_\pm))^{1/q}$ and
$e^s-1\leq se^s$; the exponential constant is bounded using
$H_b\Omega(1/2)/q$.  The analogous estimate holds for $g,p,a,v_J$.
In dimension one each retained interval stays within one sign branch.

For a fixed direction write
\begin{equation}
 \begin{split}
 H_f(e)&=\int h_efu,\qquad H_g(e)=\int h_egv,\\
 L_f(e)&=\int u(x)\int_0^{f(x)}(h_e(x)-\chi_f(e,x,t))\,dt\,dx,\\
 L_g(e)&=\int v(y)\int_0^{g(y)}(h_e(y)-\chi_g(e,y,t'))\,dt'\,dy.
 \end{split}
 \label{eq:pr-losses}
\end{equation}
They are finite, since the weight is bounded and the local energies
are integrable.  Define $D_e=((x-y)\cdot e)^2/|x-y|^2$ off null sets.

\begin{lemma}
\label{pr:layer-aggregation}
Under the preceding assumptions and \eqref{eq:pr-potentials},
\begin{equation}
 \begin{split}
 \frac{H_f(e)}{q+1}+\frac{H_g(e)}{p+1}
 \leq&\ \lambda\iint h_e(x)D_e(x,y)\,d\pi
      +\iint|h_e(x)-h_e(y)|\,d\pi\\
 &+2L_f(e)+2L_g(e)
      +C\int\Omega(2\ell)h_e(fu+gv).
 \end{split}
 \label{pr:aggregation-bound}
\end{equation}
Here $d\pi=f(x)g(y)K_\sigma(x-y)\,dx\,dy$; its marginals are
$fu$ and $gv$.  The constant depends on the coefficient modulus
constants and positive powers, not on absolute coefficient values.
\end{lemma}
\begin{proof}
Fix $e,t,t'$.  In orthogonal coordinates $x=z+se$, $y=z'+s'e$,
apply Lemma~\ref{lem:pr-interval} to every component pair with
weights $\chi_I,\chi_J$.  For $n\geq2$ ignore the null coincident-line
set $z=z'$ and take $\rho=|z-z'|>0$.  For $n=1$, the transverse
space has one atom of mass one: retain it, take $\rho=0$ and use
$\lambda=1-\sigma<1$.  The absolute Jacobian is one in both cases.
Bound $\min(\chi_I,\chi_J)$ by $h_e(x)$.
Each superlevel set on a line has countably many components.  Their nonnegative
sums reassemble into jointly measurable retained fields and weighted
endpoint kernels, so Tonelli applies without measurable enumeration
or a global choice of a basis for varying $e^\perp$.

For a retained component $I$ with endpoints $x_-,x_+$, integrating
out the opposite source gives the nonnegative identity
\begin{equation}
 \begin{split}
 \int_{\mathbb R^n}g(y)\,
 \frac{K_\sigma(x_--y)+K_\sigma(x_+-y)}2\,dy
 &=\frac{u(x_-)+u(x_+)}2=u_I,\\
 \int\!\int_0^{f(x)}\chi_f(e,x,t)u_I\,dt\,dx
 &\le C\int h_e(x)f(x)u(x)\,dx=CH_f(e).
 \end{split}
 \label{r5:endpoint-reassembly}
\end{equation}
The first equality uses the pointwise potential identities at the
nonzero endpoints. The second line follows from
\eqref{pr:endpoint-error} and $0\le\chi_f\le h_e$.
Products with a zero retention weight are interpreted as zero, as above.
Tonelli applies before either expression is subtracted; the analogous
identities hold for $g,v_J,H_g$.
For $n\geq2$ an omitted affine line is null for full-dimensional
Lebesgue integration; for $n=1$ no positive-measure set is omitted.
The two left aggregates are bounded by $H_f,H_g$, and
\eqref{r5:endpoint-reassembly} proves finiteness of the endpoint
aggregates.
Pressure is bounded by $\lambda H_f$ and mismatch by $H_f+H_g$.
All are finite before subtraction, even if an unweighted line kernel
is not integrable at infinity.  Thus the signed quantity
\[
 S_e=\int\!\int_0^f\chi_f(u-u_I)\,dt\,dx
       +\int\!\int_0^g\chi_g(v-v_J)\,dt'\,dy
\]
obeys
\begin{equation}
 S_e\leq\lambda\iint h_e(x)D_e\,d\pi
       +\iint|h_e(x)-h_e(y)|\,d\pi+L_f+L_g.
 \label{eq:pr-signed-pressure}
\end{equation}
Here one uses
$|\chi_f-\chi_g|\leq|h_e(x)-h_e(y)|+
(h_e(x)-\chi_f)+(h_e(y)-\chi_g)$.
% There is no assertion that $S_e\geq0$ or $u_I\leq u$.

The layer-cake identities for the nonlinear terms are
\begin{equation}
 \int_0^{f(x)}(u-\tau_f)\,dt=\frac{fu}{q+1},\qquad
 \int_0^{g(y)}(v-\tau_g)\,dt'=\frac{gv}{p+1}.
 \label{eq:pr-layer-identities}
\end{equation}
They hold for every $p,q>0$.
The pointwise recovery expansion
$h_e(u-\tau_f)-\chi_f(u-u_I)
=(h_e-\chi_f)(u-\tau_f)+\chi_f(u_I-\tau_f)$,
$0\leq\tau_f\leq u$, and \eqref{pr:endpoint-error} give
\[
 H_f/(q+1)+H_g/(p+1)
 \leq S_e+L_f+L_g+C\int\Omega(2\ell)h_e(fu+gv).
\]
Combining this with \eqref{eq:pr-signed-pressure} proves the result,
including its two copies of the layer loss.
If $\mathbb E_eh_e=w$ with bounded compactly supported $w$, all
finite-aggregate and signed-average assertions are also absolutely
integrable in $e$, by the marginal bounds above.
\end{proof}

\begin{lemma}
\label{pr:polar-event}
If $x\in B_2$, $0<\ell(x)\leq1/4$, and
$\int_{B_4}f\leq M$, then
\begin{equation}
 \mathbb P_e\{|I(e,x,t)|>\ell(x)\}
       \leq CM/[t\ell(x)^n].
 \label{eq:pr-long-event}
\end{equation}
\end{lemma}
\begin{proof}
A long component contains either the forward or backward segment
of length $\ell(x)/2$ from $x$.  Let $\mathcal F$ be the event that
the forward segment is in $f>t$.  Polar integration about $x$ gives
\[
 \int_{B_{\ell(x)/2}(x)}f
 \geq\frac{|\mathbb S^{n-1}|}{n}t(\ell(x)/2)^n
       \mathbb P_e(\mathcal F).
\]
Antipodal symmetry and the union of the two events prove the bound.
The formula is also valid on $\mathbb S^0$, with its two orientations
of probability $1/2$; it is just the necessary interval mass bound
$t\ell(x)/2\leq M$ in dimension one.
\end{proof}

\section{Cutoff estimates and pressure}
\label{gp:section}

We first prove localization and pressure estimates under explicit
conditions on a radial cutoff.  The interval and measurable
layer tools are proved in Section~\ref{sec:pr-pressure};
none of those tools invokes this theorem or a nonexistence result.
Throughout, $n\geq1$, $0<\sigma<n$, and $\lambda=n-\sigma$.

\subsection{Conditions on the radial cutoff}

Let $\eta\in C_c^2(\mathbb R^n)$ be radial and nonincreasing in
radius, equal to one on $B_1$, strictly positive on $B_2$, and zero
outside $B_2$.  Set
\begin{equation}
 d(x)=\min\{1,(2-|x|)_+\},\qquad
 w(x)=-x\cdot\nabla\eta(x),\qquad X(x)=x\eta(x).
 \label{gp:cutoff}
\end{equation}
Let $H\geq1$ belong to $C^1(B_2)$ and, with fixed constants, satisfy
\begin{equation}
 |\nabla H|\leq CH^2,\qquad |\nabla\eta|\leq C\eta H,
 \qquad |\nabla^2\eta|\leq C\eta H^2.
 \label{gp:derivatives}
\end{equation}
For some $c>0$, require $j=\eta^c$ to be globally Lipschitz, and
require the following uniform bounds and boundary limit:
\begin{equation}
 \begin{split}
 \sup_{B_2}jH^{n+1}&<\infty,\qquad
 \sup_{B_2}\eta^{1-c\lambda}H^n<\infty,\\
 \sup_{B_2}\eta^{1-cn}H^{n+1}&<\infty,\qquad
 \sup_{B_2}\eta^{1-c(\lambda+1)}<\infty,\\
 \eta(x)H(x)^{n+1}&\longrightarrow0\quad(|x|\uparrow2).
 \end{split}
 \label{gp:finite-budgets}
\end{equation}
Only the displayed bounds are required.
Several useful bounds are consequences rather than extra hypotheses:
$H^{-1}$ is Lipschitz on the convex ball $B_2$ by
\eqref{gp:derivatives}, $j\leq Cd$ since $j$ is Lipschitz and
zero outside $B_2$, and
\begin{equation}
 w\geq0,\quad w/H\leq C\eta,\quad
 |\nabla w|+|\nabla^2X|\leq C\eta H^2.
 \label{gp:derived-bounds}
\end{equation}

\begin{definition}
\label{gp:admissible-modulus}
A finite nondecreasing $\Omega:[0,4]\to[0,\infty)$, with
$\Omega(0)=0$ and $\Omega(s)\to0$ as $s\downarrow0$, is admissible
for $(n,\sigma)$ if some cutoff triple $(\eta,H,c)$ satisfying
\eqref{gp:cutoff}--\eqref{gp:finite-budgets} also satisfies
\begin{equation}
 R_\delta:=\sup_{1<|x|<2}
       \frac{w(x)}{\eta(x)}\Omega(2\delta j(x))<\infty,
 \qquad R_\delta\longrightarrow0\quad(\delta\downarrow0).
 \label{gp:endpoint-budget}
\end{equation}
\end{definition}

\begin{lemma}
\label{gp:geometry}
For a fixed sufficiently small $\delta_*>0$ and every
$0<\delta\leq\delta_*$, $\ell=\delta j$ obeys
$\ell\leq d/4$ and $\operatorname{Lip}(\ell)\leq1/4$.
If $|x-y|\leq\ell(z)$ for either endpoint $z\in B_2$, the joining
segment stays in $B_2$ and $d,\eta,H$ there are comparable to their
values at $z$, uniformly in $\delta\leq\delta_*$.
\end{lemma}
\begin{proof}
The bounds $j\leq Cd$ and $\operatorname{Lip}(j)<\infty$ give the
first two assertions and the $d$ comparison.  Along a close segment,
\[
 |j(\xi)-j(z)|\leq\operatorname{Lip}(j)\delta j(z).
\]
Thus $j(\xi)/j(z)=1+O(\delta)$ and $\eta=j^{1/c}$ is comparable.
Further,
\[
 \frac{|H(\xi)^{-1}-H(z)^{-1}|}{H(z)^{-1}}
 \leq C\delta j(z)H(z)
 \leq C\delta\sup_{B_2}jH^{n+1}.
\]
After reducing $\delta_*$ this proves the $H$ comparison.
This proof works for polynomial, exponential and adaptive cutoffs;
it never infers exponential comparability from the distance $d$ alone.
\end{proof}

\subsection{Geometric localization and discarded layer bounds}

The next lemma contains the geometry shared by all response and kernel
classes. It involves no inverse response and no Euler inequality.

\begin{lemma}
\label{r7:geometric-budgets}
Let $K$ be a positive even kernel with
$K(z)\le C_K|z|^{-\lambda}$, $0<\lambda<n$. Suppose
$u=K*g$, $v=K*f$, where $f,g\ge0$ are locally integrable and
continuous off zero, and the energy densities $fu,gv$ are locally
integrable. Let the cutoff obey
\eqref{gp:cutoff}--\eqref{gp:finite-budgets}. If
$\int_{B_4}(f+g)\le M$ and the truncated-potential bound in
\eqref{r6:general-pressure-input} holds, then for every $\xi>0$,
\begin{equation}
 |E_f-E_g|+\iint|\Phi-\eta(x)+w(x)D|\,d\pi
       \le\xi E+C_\xi,
 \label{gp:localization-bound}
\end{equation}
where $d\pi=f(x)g(y)K(x-y)dxdy$ and the marginals and cutoff
energies have their previous meaning.
Write $T=\iint w(x)D(x,y)\,d\pi$.
There are normalized cap weights $h_e$ satisfying
\eqref{gp:cap-bounds}, with the oscillation bound
\eqref{gp:oscillation}. If also
$\int_{B_4}(u+v)\le M$, their discarded layer losses obey
\eqref{gp:layer-loss}, for the retention rule
\eqref{pr:retention} and level threshold \eqref{gp:threshold}.
All constants are independent of the pair.
\end{lemma}
\begin{proof}
\emph{Localization.}
Call a pair close when $|x-y|\leq\ell(z)$ for either endpoint
in $B_2$, and far otherwise.  These sets are measurable.
The close geometry, Taylor's formula and
$\nabla X=\eta\operatorname{Id}-w\omega_x\otimes\omega_x$ give
\[
 |\Phi-\eta(x)+w(x)D|+|\eta(x)-\eta(y)|
 \leq C\delta j(z)\eta(z)H(z)^2\leq C\delta\eta(z),
\]
since $H\geq1$ and $\sup jH^{n+1}<\infty$.
The marginals integrate this to $C\delta E$.
On far pairs the pointwise loss is bounded by
$C\sum_z[(\eta(z)+w(z))+\eta(z)/\ell(z)]$, summed over
the endpoints in $B_2$.  After truncated-potential integration the
weights are
\begin{equation}
 \begin{split}
 (\eta+w)\ell^{-\lambda}
 &\leq C\delta^{-\lambda}\eta^{1-c\lambda}H^n,\\
 \eta\ell^{-(\lambda+1)}
 &=\delta^{-(\lambda+1)}\eta^{1-c(\lambda+1)}.
 \end{split}
 \label{gp:far-localization}
\end{equation}
The cutoff bounds and source masses give a finite uniform
$C_\delta$.  Both localized marginals are finite, so
$E_f-E_g=\iint(\eta(x)-\eta(y))d\pi$.
This proves \eqref{gp:localization-bound} from
$C\delta E+C_\delta$.  This part uses neither an annular modulus
nor the potential $L^1$ input.

\medskip
\noindent\emph{Spherical cap weights and their oscillation.}
Fix $0<\kappa<1$.  For $n\geq2$, use the exact normalizer
\eqref{eq:pr-normalizer} and set
\begin{equation}
 \rho_\kappa=\kappa/H,\qquad
 h_e=w\frac{\psi(|e-\omega_x|/\rho_\kappa)}{Z(\rho_\kappa)}
       \quad(1<|x|<2),
 \label{gp:cap}
\end{equation}
extended by zero.  For $n=1$, set $h_e=w$.
Put $(\Lambda_0,\Lambda_1,\Theta)=(1,1,0)$ for $n=1$ and
$(\kappa^{-(n-1)},\kappa^{-n},\kappa)$ otherwise.
The normalizer and derivative assumptions give
\begin{equation}
 \mathbb E_eh_e=w,\quad
 h_e\leq C\Lambda_0\eta H^n,\quad
 |\nabla h_e|\leq C\Lambda_1\eta H^{n+1}.
 \label{gp:cap-bounds}
\end{equation}
For the last bound, profile differentiation costs
$C[H/\kappa+|\nabla H|/H]\leq CH/\kappa$ on its support,
and the normalizer's relative derivative costs at most $CH$.
The boundary limit in \eqref{gp:finite-budgets} gives a $C^1$
zero extension at $|x|=2$.  At $|x|=1$, the $C^2$ constant
inner cutoff gives $w=\nabla w=0$, while $H$ is locally bounded.
The profile is constant near its center, avoiding a chordal norm
singularity.  All fields are jointly measurable.

Close geometry and $\sup jH^{n+1}<\infty$ give cap oscillation
at most $C\Lambda_1\delta\eta(z)$.
On far pairs the endpoint factor is bounded by
$C\Lambda_0\delta^{-\lambda}\eta^{1-c\lambda}H^n$.
Thus
\begin{equation}
 \mathbb E_e\iint|h_e(x)-h_e(y)|d\pi
 \leq C\Lambda_1\delta E+CM^2\Lambda_0\delta^{-\lambda}.
 \label{gp:oscillation}
\end{equation}
On cap support $|D_e-D|\leq2\kappa/H$; exact normalization
and $w/H\leq C\eta$ imply
\begin{equation}
 \mathbb E_e\iint h_e(x)D_e(x,y)d\pi\leq T+C\Theta E.
 \label{gp:angular}
\end{equation}
For $n=1$ the factors coincide and this error vanishes.

\medskip
\noindent\emph{Errors in the superlevel-set decomposition.}\par
\noindent Use the retention \eqref{pr:retention} with $\ell=\delta j$.
If $|I|\leq\ell(x)$, then
$\inf_I\ell\geq3\ell(x)/4$ by its Lipschitz bound, so the
length criterion holds.  The gradient estimate and close geometry
bound the short-component loss by $C\Lambda_1\delta\eta(x)$,
including when its infimum weight is zero.
Long components obey Lemma~\ref{pr:polar-event}.
Split levels at
\begin{equation}
 t_0(x)=Pj(x)^{-n}H(x)^n=P\eta(x)^{-cn}H(x)^n,
       \qquad P\geq1.
 \label{gp:threshold}
\end{equation}
At a level $t\ge t_0(x)$, the cap bound and
Lemma~\ref{pr:polar-event} give
\begin{equation}
 \begin{split}
 \mathbb E_e\bigl[h_e(x)\mathbf1_{\{|I(e,x,t)|>\ell(x)\}}\bigr]
 &\le \bigl(\sup_eh_e(x)\bigr)
       \mathbb P_e\{|I(e,x,t)|>\ell(x)\}\\
 &\le C\Lambda_0\eta(x)H(x)^n
          \frac{M}{t\delta^n j(x)^n}\\
 &\le CM\Lambda_0\delta^{-n}P^{-1}\eta(x).
 \end{split}
 \label{r5:high-level-cancellation}
\end{equation}
The choice of $t_0$ cancels both $H^n$ and $j^{-n}$.
Integrating this bound over $t_0<t<f(x)$ and multiplying by $u(x)$
gives a contribution at most
$CM\Lambda_0\delta^{-n}P^{-1}\eta(x)f(x)u(x)$;
the other source is identical. 
% No independence of the cap and the
% long-component event is used.
At lower levels use normalization rather than the cap supremum:
\[
 t_0w\leq CP\eta^{1-cn}H^{n+1}\leq CP.
\]
The potential $L^1$ input integrates this to $CMP$.
Together these estimates yield
\begin{equation}
 \mathbb E_e(L_f+L_g)
 \leq C[\Lambda_1\delta+M\Lambda_0\delta^{-n}P^{-1}]E+CMP.
 \label{gp:layer-loss}
\end{equation}

\end{proof}

\subsection{A common pressure estimate for radial kernels and increasing nonlinearities}

The preceding argument has a form in which the kernel homogeneity and
power nonlinearities can both be separated from the geometric estimates.
For $0<\lambda<n$, let
\begin{equation}
 K(z)=k(|z|),\qquad
 c_0r^{-\lambda}\le k(r)\le C_0r^{-\lambda},\qquad
 0\le\theta_K(r):=-\frac{rk'(r)}{k(r)}\le\Lambda
 \quad(r>0),
 \label{r6:radial-kernel-class}
\end{equation}
where $k>0$ belongs to $C^1((0,\infty))$ and
$0<c_0\le C_0<\infty$, $\Lambda<\infty$.
The same kernel is used in the two potential identities.
Let $F,G\in C([0,\infty))$ be locally Lipschitz on $(0,\infty)$, vanish at zero,
be positive on $(0,\infty)$, and satisfy
\begin{equation}
 \frac{sF'(s)}{F(s)}\ge q_0>0,\qquad
 \frac{sG'(s)}{G(s)}\ge p_0>0\quad\hbox{for almost every }s>0.
 \label{r6:pressure-elasticity}
\end{equation}
Set $P_F(s)=\int_0^sF(t)\,dt$ and
$P_G(s)=\int_0^sG(t)\,dt$.
These assumptions make $F,G$ strictly increasing onto $[0,\infty)$.

\begin{proposition}
\label{r6:general-pressure}
Suppose \eqref{r6:radial-kernel-class} and
\eqref{r6:pressure-elasticity} hold. Let $a,b,u,v>0$ be continuous
on $\mathbb R^n\setminus\{0\}$, put
\[
 f=bF(u),\qquad g=aG(v),\qquad u=K*g,\qquad v=K*f,
\]
where both identities hold finitely at every nonzero point. Assume
$f,g,fu,gv\in L^1_{\rm loc}(\mathbb R^n)$ and, for a common $M\ge1$,
\begin{equation}
 \begin{split}
 \int_{B_4}(f+g+u+v)&\le M,\\
 \int_{|x-y|\ge s}K(x-y)(f(y)+g(y))\,dy
   &\le Ms^{-\lambda}\quad(x\in B_2,\ 0<s\le1).
 \end{split}
 \label{r6:general-pressure-input}
\end{equation}
Let $(\eta,H,c)$ satisfy
\eqref{gp:cutoff}--\eqref{gp:finite-budgets}, with the present
$\lambda$, and suppose \eqref{pr:annular-modulus} and
\eqref{gp:endpoint-budget} hold for the coefficient modulus.
Define $w,X,D,\Phi$ as before and set
\[
 d\pi=f(x)g(y)K(x-y)\,dx\,dy,\quad
 E_f=\int\eta fu,\quad E_g=\int\eta gv,\quad E=E_f+E_g,
\]
\[
 T_K=\iint w(x)\theta_K(|x-y|)D(x,y)\,d\pi(x,y),\qquad
 W_P=\int w\{bP_F(u)+aP_G(v)\}.
\]
For every $\xi>0$,
\begin{equation}
 \begin{split}
 |E_f-E_g|+
 \iint|\Phi-\eta(x)+w(x)D|\,d\pi&\le\xi E+C_\xi,\\
 W_P&\le T_K+\xi E+C_\xi.
 \end{split}
 \label{r6:general-pressure-estimates}
\end{equation}
% Constants depend only on the displayed kernel and elasticity bounds,
% the cutoff and modulus bounds, $n,\lambda,M,\xi$; they do not depend
% on absolute coefficient values or on the particular pair.
% The first estimate does not require the modulus condition or the
% $L^1$ bounds for $u,v$.
\end{proposition}

\begin{proof}
The kernel is even. Its interaction measure therefore has the same
marginals $fu\,dx$ and $gv\,dy$ as in \eqref{vi:marginals}.
All localized energies and $T_K$ are finite, since $w$ is bounded,
compactly supported, and $0\le\theta_KD\le\Lambda$.
Also $0\le P_F(u)\le uF(u)$ and
$0\le P_G(v)\le vG(v)$, so $W_P$ is finite.

\emph{Localization.}
The first estimate is Lemma~\ref{r7:geometric-budgets}, applied to
the present common kernel and its marginals. Its proof requires no
coefficient modulus or response derivative.

\emph{The radial interval identity.}
On two parallel lines write
\[
 k_\rho(h)=k\bigl((\rho^2+h^2)^{1/2}\bigr),\qquad
 Q_\rho(h)=\theta_K\bigl((\rho^2+h^2)^{1/2}\bigr)
                  \frac{h^2}{\rho^2+h^2}.
\]
For bounded intervals $I,J$, let $d_0$ be their midpoint difference,
and let $k_{\rho,I},k_{\rho,J}$ denote the endpoint averages as
in Section~\ref{sec:pr-pressure}. For $\rho>0$, integration by
parts gives the exact identity
\begin{equation}
 \int_{I\times J}(2k_\rho-k_{\rho,I}-k_{\rho,J})
 =\int_{I\times J}Q_\rho k_\rho
       +d_0\int_{I\times J}k_\rho'(h).
 \label{r6:radial-interval-identity}
\end{equation}
The line kernel is even and nonincreasing in $|h|$.
Pairing $h$ and $-h$ against
$\varrho(h)=|I\cap(J+h)|$, which is nonincreasing in
$|h-d_0|$, gives $d_0\int k_\rho'(h)\le0$.
Multiplication by the smaller interval weight, followed by the
nonnegative endpoint bound for each remaining defect, proves
\eqref{eq:pr-interval} with $\lambda\mathcal D_\rho$ replaced
by $Q_\rho$. If only one interval weight is positive, the mismatch
term alone bounds the left side, as before.
For coincident lines in dimension one, $0<\lambda<1$.
Letting $\rho\downarrow0$ is justified on bounded interval pairs
by $k_\rho(h)\le C_0|h|^{-\lambda}$ and
$0\le Q_\rho k_\rho\le\Lambda C_0|h|^{-\lambda}$;
the endpoint averages have the same integrable one-dimensional
majorants. Thus dominated convergence gives the required inequality.
No monotonicity of $Q_\rho$ in $\rho$ is asserted.

\emph{Inverse endpoint values and the nonlinear layer identity.}
Use the retained components and weights $\chi_f,\chi_g$ in
\eqref{pr:retention}. Put
$\tau_f=F^{-1}(t/b(x))$ for $0<t<f(x)$.
The first inequality in \eqref{r6:pressure-elasticity} implies
\[
 |\log F^{-1}(r_1)-\log F^{-1}(r_2)|
       \le q_0^{-1}|\log r_1-\log r_2|\quad(r_1,r_2>0).
\]
At a retained endpoint $x_\pm$, the source value is $t$, so its
potential value is $F^{-1}(t/b(x_\pm))$.
Consequently \eqref{pr:annular-modulus} and
$|x_\pm-x|\le2\ell(x)$ give
\begin{equation}
 |u_I-\tau_f|\le C\Omega(2\ell(x))u(x),\qquad u_I\le Cu(x).
 \label{r6:inverse-endpoint}
\end{equation}
Here one uses $e^s-1\le se^s$ and
$\tau_f\le u(x)$; the exponential constant is bounded by
$H_b\Omega(1/2)/q_0$. The other source uses $G,p_0,H_a$.
Changing variables $t=b(x)F(s)$, or using the elementary area identity
for a strictly increasing function and its inverse, gives
\begin{equation}
 \int_0^{f(x)}(u(x)-\tau_f)\,dt=b(x)P_F(u(x)),\qquad
 \int_0^{g(y)}(v(y)-\tau_g)\,dt'=a(y)P_G(v(y)).
 \label{r6:primitive-layer}
\end{equation}

\emph{Finite layer aggregation.}
Apply \eqref{r6:radial-interval-identity} to every retained component
pair. The measurability argument, unit-Jacobian slicing and
nonnegative Tonelli reassembly in Lemma~\ref{pr:layer-aggregation}
are unchanged: they use continuity of the sources, not their power
form. In particular, integrating an endpoint kernel against the
opposite source gives exactly $u_I$ or $v_J$.
Estimate \eqref{r6:inverse-endpoint} bounds these endpoint aggregates
by $CH_f(e)$ and $CH_g(e)$ before subtraction. The new pressure
aggregate is bounded by $\Lambda H_f(e)$, and mismatch by
$H_f(e)+H_g(e)$. Hence all signed operations involve finite
quantities. With the losses from \eqref{eq:pr-losses}, the signed
pressure bound is
\[
 S_e\le\iint h_e(x)\theta_K(|x-y|)D_e\,d\pi
       +\iint|h_e(x)-h_e(y)|\,d\pi+L_f(e)+L_g(e).
\]
The exact recovery expansion used after
\eqref{eq:pr-layer-identities}, now combined with
\eqref{r6:inverse-endpoint}--\eqref{r6:primitive-layer}, yields
\begin{equation}
 \begin{split}
 \int h_e\{bP_F(u)+aP_G(v)\}
 \le{}&\iint h_e(x)\theta_K(|x-y|)D_e\,d\pi
       +\iint|h_e(x)-h_e(y)|\,d\pi\\
 &+2L_f(e)+2L_g(e)
       +C\int\Omega(2\ell)h_e(fu+gv).
 \end{split}
 \label{r6:primitive-aggregation}
\end{equation}
Products with zero retention weight are defined as zero, so no
endpoint value is required there.

\emph{Directional averaging and the error budget.}
Use exactly the cap weights \eqref{gp:cap}, with the same
$(\Lambda_0,\Lambda_1,\Theta)$ and exact average
$\mathbb E_eh_e=w$. Their derivative bounds, close/far oscillation
bound \eqref{gp:oscillation}, polar long-component event and
high/low level split \eqref{gp:threshold} are independent of both
the kernel derivative and the nonlinearities. The tail input and
marginals therefore prove \eqref{gp:layer-loss} unchanged.
Since $\theta_K(|x-y|)$ does not depend on $e$, cap support gives
\[
 \mathbb E_e\iint h_e(x)\theta_K(|x-y|)D_e\,d\pi
 \le T_K+C\Lambda\Theta E.
\]
The endpoint average in \eqref{r6:primitive-aggregation} is at most
$CR_\delta E$. Thus its averaged right side is bounded by
\[
 T_K+C[\Lambda\Theta+\Lambda_1\delta+R_\delta
       +M\Lambda_0\delta^{-n}/P]E
       +C[M^2\Lambda_0\delta^{-\lambda}+MP].
\]
First choose the cap parameter, then $\delta$, and finally a finite
$P$ to make each displayed energy coefficient small. This makes the energy
coefficient at most $\xi$ and proves the second line of
\eqref{r6:general-pressure-estimates}.
\end{proof}

\begin{corollary}
\label{gp:interface}\label{gp:localization}\label{gp:pressure}
For the Riesz kernel and power responses $F(s)=s^q$, $G(s)=s^p$,
the geometric input of Lemma~\ref{r7:geometric-budgets} gives
\eqref{gp:localization-bound}. Under all the additional hypotheses of
Proposition~\ref{r6:general-pressure}, including the coefficient
modulus \eqref{pr:annular-modulus}, endpoint budget
\eqref{gp:endpoint-budget} and full source/potential inputs,
\begin{equation}
 \frac{W_f}{q+1}+\frac{W_g}{p+1}
       \le\lambda T+\xi E+C_\xi.
 \label{gp:pressure-bound}
\end{equation}
Here $p,q>0$; their product need not exceed one.
\end{corollary}
\begin{proof}
The primitives are $P_F(u)=u^{q+1}/(q+1)$ and
$P_G(v)=v^{p+1}/(p+1)$, while $\theta_K=\lambda$.
Substitute them in Proposition~\ref{r6:general-pressure}.
\end{proof}

\section{Radial cutoffs adapted to Dini moduli}
\label{gr18:section}\label{gd:section}

\subsection{A general cutoff construction}

The following inverse-function construction supplies the geometric
properties and uniform bounds needed for cutoffs adapted to a Dini modulus.

\begin{lemma}
\label{gr18:speed-realization}
Let $n\geq1$, $0<\sigma<n$, $\lambda=n-\sigma$, and
\[
 0<c<\min\{1/n,1/(\lambda+1)\}.
\]
Fix $S_0\geq1$.  Suppose $F>0$ is smooth on $[S_0,\infty)$,
and, with fixed constants, satisfies
\begin{equation}
 \begin{gathered}
 L:=\int_{S_0}^{\infty}\frac{dt}{F(t)}<\infty,\qquad
 |F'(t)|\leq C F(t),\\
 |F^{(j)}(t)|\leq C_j(1+t)^{M_j}
 \quad(t\geq S_0,\ j=0,1,\ldots),
 \end{gathered}
 \label{gr18:speed-hypotheses}
\end{equation}
where all $M_j$ are finite and nonnegative.
Define $S_o$ on $1\leq r<2$ by
\begin{equation}
 2-r=L^{-1}\int_{S_o(r)}^\infty\frac{dt}{F(t)},
 \qquad S_o(1)=S_0,\qquad S_o'=L F(S_o).
 \label{gr18:outer-coordinate}
\end{equation}
Let $\zeta$ be a fixed smooth nondecreasing function, zero for
$r\leq1$ and one for $r\geq3/2$.  Set $S=\zeta S_o$ on
$1<r<2$, extend $S=0$ to $r\leq1$, and take
\[
 \eta=e^{-S}\ \hbox{on }B_2,\qquad \eta=0\ \hbox{outside }B_2,
 \qquad H=1+S'(|x|)\ \hbox{on }B_2.
\]
Then $\eta$ and every positive power of $\eta$ are smooth with
compact support, $\eta=1$ on $B_1$, and $w=-x\cdot\nabla\eta\geq0$.
The derivative conditions \eqref{gp:derivatives}, the bounds
\eqref{gp:finite-budgets}, and the globally Lipschitz property of
$j=\eta^c$ all hold.  More strongly, for all $t>0$ and finite
$m,k\geq0$,
\begin{equation}
 \sup_{B_2}\eta^tH^m d^{-k}<\infty,\qquad
 \eta^tH^m\longrightarrow0\quad(|x|\uparrow2).
 \label{gr18:rapid-domination}
\end{equation}
The close-pair conclusions of Lemma~\ref{gp:geometry} consequently
hold with this cutoff triple.

If a finite nondecreasing modulus $\Omega:[0,4]\to[0,\infty)$,
with $\Omega(0)=0$ and $\Omega(s)\to0$, also satisfies
\begin{equation}
 F(t)\Omega(e^{-ct})\longrightarrow0\quad(t\to\infty),
 \label{gr18:speed-endpoint}
\end{equation}
then this cutoff proves its admissibility in
Definition~\ref{gp:admissible-modulus}.  
% Constants need not be uniform over different choices of $F$ or $\Omega$.
\end{lemma}

\begin{proof}
The positive finite tail integral in \eqref{gr18:outer-coordinate}
is strictly decreasing, so it defines a smooth increasing $S_o$,
with $S_o\to\infty$ as $r\uparrow2$.  Positivity of $S_o$ and
monotonicity of $\zeta$ give $S'\geq0$.  On the compact transition
annulus all derivatives are bounded, and flatness of $\zeta$ at
$r=1$ supplies smooth matching with the inner constant.

Choose $M\geq2$ so that $F(t)\leq C(1+t)^M$.  Near the outer
boundary, where $S=S_o$, the tail coordinate gives
\begin{equation}
 d=2-r\geq C^{-1}(1+S)^{1-M}.
 \label{gr18:boundary-distance}
\end{equation}
Repeated differentiation there applies the operator
$LF(S)\,d/dS$.  The polynomial bounds for every derivative of
$F$ therefore imply polynomial growth in $S$ of every radial
derivative of $S$.  All radial derivatives of $e^{-tS}$ are
bounded by $e^{-tS}$ times a polynomial in $S$, and tend to zero
at $r=2$.  Thus every positive power of $\eta$ has a smooth
zero extension.  The coordinate is constant near the origin,
so radial smoothness there causes no additional condition.
The same polynomial bounds and \eqref{gr18:boundary-distance}
prove \eqref{gr18:rapid-domination}.

On the outer annulus $H=1+LF(S)$ and
$H'=L^2F'(S)F(S)$, so $|H'|\leq C H^2$.
The compact transition satisfies the same inequality after increasing
$C$, since $H\geq1$.  Differentiating $\eta=e^{-S}$ gives
$|\nabla\eta|\leq\eta H$ and $|\nabla^2\eta|\leq C\eta H^2$.
The radial Hessian also contains $S'/r$; here $r\geq1$ whenever
$S'\ne0$.  All exponents of $\eta$ required in
\eqref{gp:finite-budgets} are positive by the chosen $c$,
so \eqref{gr18:rapid-domination} verifies them and the boundary
limit.  The smooth compact extension of $j$ is globally Lipschitz.
The close-pair result is exactly Lemma~\ref{gp:geometry}.

Finally, on the outer annulus $w/\eta=rLF(S)$.
For $0<\delta\leq1/2$, monotonicity of $\Omega$ gives
\[
 \frac{w}{\eta}\Omega(2\delta\eta^c)
 \leq 2L F(S)\Omega(e^{-cS}).
\]
Given $\varepsilon>0$, choose a fixed large $M_*$ so that this
bound is below $\varepsilon$ when $S\geq M_*$.
On the remaining compact region $w/\eta$ is bounded and
$\Omega(2\delta\eta^c)\leq\Omega(2\delta)\to0$.
The same split proves finiteness for every $0<\delta\leq1/2$
and the required uniform supremum convergence.
% No monotonicity of $F$ or differentiability of $\Omega$ was used.
\end{proof}

Multiplying the speed by a fixed positive constant does not change
these sufficient conditions.  An already prescribed outer coordinate
with $S'=\kappa F(S)$ therefore satisfies the same conclusions
whenever its finite tail coordinate ends at $r=2$; the proof uses
$\kappa$ in place of $L$.  This permits explicit cutoffs to keep
their own normalization and endpoint rates.  The lemma is not a
characterization of all cutoff triples satisfying the stated bounds.

\subsection{Realization and characterization of Dini moduli}

The cutoff conditions permit a profile adapted to the full modulus,
rather than to a prescribed power of a logarithm. The constants may depend on the full modulus.

\begin{theorem}
\label{gd:realization}\label{gd:admissible-characterization}
Let $n\geq1$, $0<\sigma<n$, and let $\Omega:[0,4]\to[0,\infty)$
be finite and nondecreasing, with $\Omega(0)=0$ and
$\Omega(s)\to0$ as $s\downarrow0$. Then $\Omega$ is admissible
in Definition~\ref{gp:admissible-modulus} if and only if
\begin{equation}
 \int_0^{1/2}\Omega(s)\,\frac{ds}{s}<\infty.
 \label{gd:dini}
\end{equation}
For every Dini modulus and any
$0<c<\min\{1/n,1/(n-\sigma+1)\}$, one can choose a smooth radial
cutoff $\eta=e^{-S}$ and $H=1+S'$ on $B_2$, with $S=0$ on $B_1$,
such that all cutoff bounds of Corollary~\ref{gp:interface} hold. Conversely, finiteness of
$R_\delta$ for just one $0<\delta\leq1/2$ in the radial cutoff class
already forces \eqref{gd:dini}.
\end{theorem}
\begin{proof}
\emph{An integrable slowly weighted majorant.}
Set $m(t)=\Omega(e^{-ct})$ for $t\geq0$ and extend it by the
constant $\Omega(1)$ for $t<0$. Thus $m$ is nonincreasing and
integrable on $[0,\infty)$. Choose $T_k\uparrow\infty$ with
$T_1\geq2$, $T_{k+1}\geq T_k+2$, and
$\int_{T_k}^{\infty}m\leq4^{-k}$.
Let $\psi$ be a smooth nondecreasing step, zero on $(-\infty,0]$
and one on $[1,\infty)$, and put
\[
 \rho(t)=1+\sum_{k\geq1}\psi(t-T_k).
\]
This sum is locally finite; $\rho\geq1$, $\rho\to\infty$, and
$0\leq\rho'\leq C$. In particular, $|\log\rho|'\leq C$.
Tonelli gives $\int_0^{\infty}\rho m<\infty$.
With $t_+=\max(t,0)$, define
\[
 b(t)=(1+t_+)^{-2},\qquad \varphi(t)=\rho(t)m(t)+b(t).
\]
The function $\varphi$ is positive, integrable on the positive half-line,
and constant on the negative half-line. 
% No continuity of $m$ is needed.

\medskip
\noindent\emph{Positive smoothing with relative derivative bounds.}
Consider the one-sided convolution
\[
 r_1(t)=\int_0^{\infty}e^{-u}\varphi(t-u)\,du.
\]
For $0\leq u\leq1$, monotonicity of $m,b$ and log-Lipschitz
control of $\rho$ imply $\varphi(t-u)\geq a\varphi(t)$ with
a fixed $a>0$. Hence $r_1\geq a(1-e^{-1})\varphi$.
The weak identity $r_1'=\varphi-r_1$ therefore gives
$|r_1'|\leq C r_1$ almost everywhere. The function $r_1$ is
locally absolutely continuous and log-Lipschitz. It is integrable on
$[0,\infty)$: the negative-half-line contribution decays exponentially,
and Tonelli bounds the remaining contribution by $\int_0^{\infty}\varphi$.

Choose a nonnegative smooth mollifier $\kappa$ of integral one,
supported in $[-1,1]$, and set $r_2=\kappa*r_1$. Shifted comparability
of $r_1$ shows
\begin{equation}
 C^{-1}r_1\leq r_2\leq Cr_1,
 \qquad |r_2^{(k)}|\leq C_k r_2\quad(k\geq1).
 \label{gd:smoothing}
\end{equation}
For the derivative bound differentiate $\kappa$, not $m$.
Consequently $r_2>0$ is smooth and integrable on the positive half-line,
and
\[
 r_2(t)\geq C^{-1}\rho(t)m(t),\qquad
 r_2(t)\geq C^{-1}(1+t)^{-2}\quad(t\geq0).
\]
An integrable positive log-Lipschitz function tends to zero at infinity:
otherwise its values bounded below on an infinite separated sequence
would give disjoint intervals with a fixed positive mass. Thus, with
$F=1/r_2$,
\begin{equation}
 \begin{split}
 &F(t)\to\infty,\qquad F(t)\leq C(1+t)^2,
 \qquad |F^{(k)}(t)|\leq C_kF(t),\\
 &\int_1^{\infty}\frac{dt}{F(t)}<\infty,
 \qquad F(t)\Omega(e^{-ct})\leq C/\rho(t)\to0.
 \end{split}
 \label{gd:speed}
\end{equation}
Reciprocal differentiation and \eqref{gd:smoothing} prove the derivative
bounds.
% monotonicity of $F$ is neither asserted nor required.

\medskip
\noindent\emph{Construction of the cutoff.}
The speed in \eqref{gd:speed} satisfies
\eqref{gr18:speed-hypotheses} and \eqref{gr18:speed-endpoint}.
Lemma~\ref{gr18:speed-realization}, with any fixed $S_0\geq1$,
therefore supplies the required smooth cutoff, all uniform bounds and
the uniform endpoint convergence. 
% No second cutoff construction or
% boundary-distance argument is needed.

\medskip
\noindent\emph{Necessity within the radial cutoff class.}
For necessity take an admissible radial cutoff and set
$S(r)=-\log\eta(r)$ for $1\leq r<2$.
Then $S$ is $C^1$ and nondecreasing, $S(1)=0$,
$S(r)\to\infty$, and $w/\eta=rS'$.
The endpoint bound gives
\[
 S'(r)\Omega(2\delta e^{-cS(r)})\leq R_\delta/r.
\]
Integrating and changing variables along the nondecreasing $S$ yields
\[
 \frac1c\int_0^{2\delta}\Omega(s)\,\frac{ds}{s}
 =\int_1^2 S'(r)\Omega(2\delta e^{-cS(r)})\,dr
 \leq R_\delta\log2<\infty.
\]
The monotone substitution permits flat intervals of $S$ and jumps
of $\Omega$; it follows equally by composing the locally absolutely
continuous primitive of $t\mapsto\Omega(2\delta e^{-ct})$ with $S$
on $[1,b]$, $b<2$, and then letting $b\uparrow2$.
The remaining compact part of the Dini integral is finite because
$\Omega$ is finite. This proves necessity within the stated cutoff class.
\end{proof}

This equivalence characterizes the radial cutoff conditions used in
this proof. It does not assert necessity of the Dini condition for
Lane--Emden nonexistence. A non-Dini modulus may require a different
pressure estimate. Quantitative control uniform over a
family of Dini moduli additionally requires control of their tails and
of the choices $T_k$ in the realization.

\section{Response normalization and proof of the master theorem}
\label{r6:nonlinear-section}\label{r7:master-proof-section}

Write $A=A_-$, $B=B_-$, and $\lambda=n-\sigma$ throughout
this section. The response primitives are denoted by $P_F,P_G$.
Only their lower ratios and the lower logarithmic response
derivatives are used. In particular, no upper response derivative
bound will enter the uniform estimates.

\subsection{Response bounds and necessary source exponents}

Since $F$ is positive and locally Lipschitz, $\log F$ is absolutely
continuous on every compact positive interval. Integrating the
almost everywhere lower bound in \eqref{r7:response-assumptions}
gives $F(sz)/F(s)\ge z^{q_0}$ for $z\ge1$. In particular, $F$
is strictly increasing onto $[0,\infty)$. The same holds for $G$.
For $0<z\le1$ the reversed estimate gives
$F(sz)/F(s)\le z^{q_0}$, so
\begin{equation}
 0<\mu_F\le\frac{P_F(s)}{sF(s)}\le\frac1{q_0+1},\qquad
 0<\mu_G\le\frac{P_G(s)}{sG(s)}\le\frac1{p_0+1}.
 \label{r7:primitive-ratio-bounds}
\end{equation}
Define the finite effective response exponents
\begin{equation}
 q_*:=\mu_F^{-1}-1\ge q_0,\qquad
 p_*:=\mu_G^{-1}-1\ge p_0.
 \label{r7:effective-exponents}
\end{equation}
The primitive inequality implies
$sP_F'(s)/P_F(s)\le1/\mu_F$. Integrating it from $sz$ to $s$
for $0<z\le1$, and then using monotonicity of $F$, gives
\[
 F(sz)\ge\frac{P_F(sz)}{sz}
 \ge z^{1/\mu_F-1}\frac{P_F(s)}s
 \ge\mu_F z^{q_*}F(s).
\]
Integrating the same primitive inequality from $s$ to $sz$ for
$z\ge1$ gives the corresponding upper estimate. Thus
\begin{equation}
 \begin{array}{ll}
 \mu_F z^{q_*}\le F(sz)/F(s)\le z^{q_0},&0<z\le1,\\[2pt]
 z^{q_0}\le F(sz)/F(s)\le\mu_F^{-1}z^{q_*},&z\ge1.
 \end{array}
 \label{r6:response-bounds}
\end{equation}
There are identical bounds for $G,p_0,p_*,\mu_G$.
All these constants are independent of $s>0$.
For the normalized responses
\begin{equation}
 F_s(z)=F(sz)/F(s),\qquad G_t(z)=G(tz)/G(t),
 \label{r6:normalized-responses}
\end{equation}
the lower logarithmic derivative bounds are unchanged. Their
primitive ratios are also preserved, since
\begin{equation}
 \frac{P_{F_s}(z)}{zF_s(z)}
 =\frac{P_F(sz)}{szF(sz)},\qquad
 \frac{P_{G_t}(z)}{zG_t(z)}
 =\frac{P_G(tz)}{tzG(tz)}.
 \label{r7:primitive-ratio-invariance}
\end{equation}

\begin{lemma}
\label{r7:source-parameters}
Suppose a positive pair satisfies the finite potential identities
in Theorem~\ref{r7:master-front}. Put $f=bF(u)$, $g=aG(v)$,
and use \eqref{r7:effective-exponents}. If
$\Gamma_P>0$, then
\begin{equation}
 \begin{split}
 \alpha_*&:=\frac{\sigma(p_*+1)+A+p_*B}{p_*q_*-1},\\
 \beta_*&:=\frac{\sigma(q_*+1)+B+q_*A}{p_*q_*-1}
 \end{split}
 \label{r7:source-exponents}
\end{equation}
satisfy
\begin{equation}
 0<\alpha_*,\beta_*\le\lambda,\qquad
 \alpha_*+\beta_*>\lambda,\qquad n+A>0,\quad n+B>0.
 \label{r7:forced-positivity}
\end{equation}
Moreover, if $x_0=\sigma+A$, $y_0=\sigma+B$, then
\begin{equation}
 x_0+Py_0>0,\qquad y_0+Qx_0>0
 \quad(P\in[p_0,p_*],\ Q\in[q_0,q_*]),
 \label{r7:rectangle-positivity}
\end{equation}
with uniform positive lower bounds on these intervals.
% No nonlinear energy hypothesis or coefficient modulus is needed
% for these conclusions.
\end{lemma}

\begin{proof}
Finite positive potentials make $f,g$ locally integrable on the
whole space. For example, a nonzero observation point outside a
fixed ball gives a positive kernel lower bound on that ball.
At infinity, finiteness at one observation point and the kernel
comparison give
\[
 \int_{|y|>1}|y|^{-\lambda}[f(y)+g(y)]\,dy<\infty.
\]
Let $M_f(R)=\int_{B_R}f$, $M_g(R)=\int_{B_R}g$.
For $1<L<R$,
\[
 R^{-\lambda}M_f(R)
 \le R^{-\lambda}M_f(L)
       +\int_{|y|>L}|y|^{-\lambda}f(y)\,dy.
\]
Taking $R\to\infty$ and then $L\to\infty$, and arguing likewise
for $g$, proves
\begin{equation}
 R^{-\lambda}M_f(R)\longrightarrow0,\qquad
 R^{-\lambda}M_g(R)\longrightarrow0.
 \label{r7:finite-potential-source-smallness}
\end{equation}

Use the fixed patch $D=B_{1/8}(3e_0/4)\subset B_1$.
For $R\ge1$ and $z\in D$, the weak Euler bounds imply
$a(Rz)\ge R^Aa(z)$ and $b(Rz)\ge R^Bb(z)$.
The coefficients have positive minima on $\overline D$.
Also, for $x\in RD$ and $y\in B_R$,
$K(x-y)\ge cR^{-\lambda}$ almost everywhere. Consequently,
\[
 u(x)\ge cR^{-\lambda}M_g(R),\qquad
 v(x)\ge cR^{-\lambda}M_f(R)\quad(x\in RD).
\]
Both quantities on the right tend to zero by
\eqref{r7:finite-potential-source-smallness}. Apply monotonicity of
$F,G$ followed by their small-argument lower bounds in
\eqref{r6:response-bounds}, and integrate over $RD$. For all
sufficiently large $R$ this gives
\begin{equation}
 \begin{split}
 M_f(R)&\ge cR^{n+B-\lambda q_*}M_g(R)^{q_*},\\
 M_g(R)&\ge cR^{n+A-\lambda p_*}M_f(R)^{p_*}.
 \end{split}
 \label{r7:unscaled-source-feedback}
\end{equation}
The constants may depend on the original coefficients and responses;
no uniform family estimate is being claimed in this preliminary
step. Substituting the first inequality into the second and
rearranging, using $p_*q_*>1$, yields
\[
 M_g(R)\le CR^{\lambda-\alpha_*},\qquad
 M_f(R)\le CR^{\lambda-\beta_*}.
\]
Each source has a strictly positive mass on a fixed ball. Therefore
neither right side can tend to zero, and
$\alpha_*,\beta_*\le\lambda$.

The kernel assumptions imply $\Lambda\ge\lambda$: otherwise
integrating $-d\log k(r)/d\log r\le\Lambda$ from one to $r\ge1$
would give $k(r)\ge k(1)r^{-\Lambda}$, contradicting
$k(r)\le C_Kr^{-\lambda}$ as $r\to\infty$.
By \eqref{r7:effective-exponents},
\begin{equation}
 \alpha_*+\beta_*-\lambda
 =\frac{(p_*+1)(q_*+1)}{p_*q_*-1}
                         [\Gamma_P+\Lambda-\lambda]>0.
 \label{r7:source-gap-identity}
\end{equation}
Together with the individual upper bounds this forces
$\alpha_*,\beta_*>0$. The defining equations also give
\[
 n+A=\lambda-\alpha_*+p_*\beta_*>0,\qquad
 n+B=\lambda-\beta_*+q_*\alpha_*>0.
\]
Finally, $x_0=p_*\beta_*-\alpha_*$ and
$y_0=q_*\alpha_*-\beta_*$. Thus
\begin{equation}
 \begin{split}
 x_0+Py_0&=(p_*-P)\beta_*+(Pq_*-1)\alpha_*>0,\\
 y_0+Qx_0&=(q_*-Q)\alpha_*+(Qp_*-1)\beta_*>0.
 \end{split}
 \label{r7:rectangle-decomposition}
\end{equation}
Here $Pq_*\ge p_0q_0>1$ and $Qp_*\ge p_0q_0>1$.
Compactness of the exponent intervals supplies uniform positive
lower bounds. This proves every assertion.
\end{proof}

\subsection{A closed normalization family}

\begin{lemma}
\label{r6:nonlinear-normalization}\label{r7:normalization}
Suppose a pair as in Theorem~\ref{r7:master-front} exists.
For every $R>0$ there are unique $s_R,t_R>0$ such that
\begin{equation}
 s_R=R^\sigma a(Re_0)G(t_R),\qquad
 t_R=R^\sigma b(Re_0)F(s_R).
 \label{r6:implicit-amplitudes}
\end{equation}
Define
\begin{equation}
 \begin{aligned}
 K_R(z)&=R^\lambda K(Rz),\\
 u_R(x)&=u(Rx)/s_R,& v_R(x)&=v(Rx)/t_R,\\
 f_R(x)&=b_R(x)F_{s_R}(u_R(x)),&
 g_R(x)&=a_R(x)G_{t_R}(v_R(x)).
 \end{aligned}
 \label{r6:complete-family}
\end{equation}
Then
\begin{equation}
 u_R=K_R*g_R,\qquad v_R=K_R*f_R,\qquad
 g_R(x)=\frac{R^\sigma}{s_R}g(Rx),\quad
 f_R(x)=\frac{R^\sigma}{t_R}f(Rx).
 \label{r6:exact-source-transform}
\end{equation}
There is a constant $C$ independent of $R>0$ and $\rho\ge1$
such that
\begin{equation}
 \frac{s_{R\rho}}{s_R}\le C\rho^{-\alpha_*},\qquad
 \frac{t_{R\rho}}{t_R}\le C\rho^{-\beta_*}.
 \label{r6:amplitude-ratio}
\end{equation}
For every $R\ge R_0$, there is a common finite $M$ such that
\begin{equation}
 \begin{split}
 \int_{B_4}(f_R+g_R+u_R+v_R)&\le M,\\
 \int_{|x-y|\ge h}K_R(x-y)(f_R(y)+g_R(y))\,dy
 &\le Mh^{-\lambda}\quad(x\in B_2,\ 0<h\le1).
 \end{split}
 \label{r6:uniform-source-input}
\end{equation}
All kernel, response and coefficient assumptions used in these
bounds have the same constants over the normalized family.
\end{lemma}

\begin{proof}
Put $A_R=R^\sigma a(Re_0)$ and $B_R=R^\sigma b(Re_0)$.
Eliminating $t_R$ reduces \eqref{r6:implicit-amplitudes} to
\[
 h_R(s):=\frac{s}{G(B_RF(s))}=A_R.
\]
For $s_2>s_1>0$, the integrated lower response derivatives give
\[
 \frac{h_R(s_2)}{h_R(s_1)}
 \le\left(\frac{s_2}{s_1}\right)^{1-p_0q_0}.
\]
Thus $h_R$ is continuous and strictly decreasing, tends to infinity
at zero and to zero at infinity. It has a unique positive inverse
value; then $t_R=B_RF(s_R)$ is unique. This argument uses finite
response increments and is valid for locally Lipschitz responses.
Changing variables in the potentials proves
\eqref{r6:exact-source-transform}. The definition of $K_R$ preserves
its comparison constants and gives
$\theta_{K_R}(r)=\theta_K(Rr)$.
The weak Euler bounds and normalized coefficient modulus are also
preserved. Equations \eqref{r6:normalized-responses} and
\eqref{r7:primitive-ratio-invariance} show the same for all response
constants.

\emph{Finite-scale amplitude bounds.}
Fix $\rho>1$ and write
\[
 S=\frac{s_{R\rho}}{s_R},\quad
 T=\frac{t_{R\rho}}{t_R},\quad L=\log\rho,\quad
 U=\log S,\quad V=\log T.
\]
The implicit equations and the weak Euler inequalities imply
\begin{equation}
 S\ge\rho^{x_0}G_{t_R}(T),\qquad
 T\ge\rho^{y_0}F_{s_R}(S).
 \label{r7:amplitude-inequalities}
\end{equation}
By \eqref{r6:response-bounds},
\[
 S\ge\mu_G\rho^{x_0}\min\{T^{p_0},T^{p_*}\},\qquad
 T\ge\mu_F\rho^{y_0}\min\{S^{q_0},S^{q_*}\}.
\]
Select $P=p_0$ when $V\ge0$ and $P=p_*$ when $V<0$;
select $Q=q_0$ when $U\ge0$ and $Q=q_*$ when $U<0$.
At equality either selection is permitted.
With $C_F=-\log\mu_F\ge0$, $C_G=-\log\mu_G\ge0$, this gives
\[
 U\ge x_0L+PV-C_G,\qquad
 V\ge y_0L+QU-C_F.
\]
Substitute each inequality into the other, using $P,Q>0$, to obtain
\begin{equation}
 \begin{split}
 (PQ-1)U&\le-(x_0+Py_0)L+PC_F+C_G,\\
 (PQ-1)V&\le-(y_0+Qx_0)L+QC_G+C_F.
 \end{split}
 \label{r7:log-amplitude-bounds}
\end{equation}
Lemma~\ref{r7:source-parameters} provides positive constants
\begin{equation}
 \delta_s=\frac{\min_{P\in[p_0,p_*]}(x_0+Py_0)}{p_*q_*-1}>0,
 \qquad
 \delta_t=\frac{\min_{Q\in[q_0,q_*]}(y_0+Qx_0)}{p_*q_*-1}>0.
 \label{r6:positive-amplitude-growth}
\end{equation}
Since $PQ-1\ge p_0q_0-1>0$, the additive constants in
\eqref{r7:log-amplitude-bounds} are uniformly bounded after division.
It follows that $U\le-\delta_sL+C$ and
$V\le-\delta_tL+C$. Hence, above one fixed threshold
$\rho_1>1$, both are negative. For $\rho\ge\rho_1$, the selected
exponents are therefore $P=p_*$, $Q=q_*$, and
\eqref{r7:log-amplitude-bounds} gives
$U\le-\alpha_*L+C$, $V\le-\beta_*L+C$.
For $1\le\rho\le\rho_1$, enlarge the constant in these estimates
using the preliminary $\delta_s,\delta_t$ bounds. This proves
\eqref{r6:amplitude-ratio} for every $\rho\ge1$.
% Neither a pointwise upper response derivative nor radial
% coefficient differentiability has been used.

\emph{Uniform unit-ball source masses.}
Let $m_f(R)=\int_{B_1}f_R$, $m_g(R)=\int_{B_1}g_R$.
On $D=B_{1/8}(3e_0/4)$, the coefficient modulus and normalization
$a_R(e_0)=b_R(e_0)=1$ give uniform positive lower coefficient
bounds for $R\ge R_0$. In dimension one, this patch lies in the
positive branch. The kernel lower bound implies
\[
 u_R(x)\ge c\,m_g(R),\qquad
 v_R(x)\ge c\,m_f(R)\quad(x\in D).
\]
The response window gives
\begin{equation}
 \begin{split}
 m_f(R)&\ge c\min\{m_g(R)^{q_0},m_g(R)^{q_*}\},\\
 m_g(R)&\ge c\min\{m_f(R)^{p_0},m_f(R)^{p_*}\}.
 \end{split}
 \label{r6:mass-feedback}
\end{equation}
If either mass were unbounded, these inequalities would force both
to be large. Their large-mass branches would then imply
$m_f(R)\ge c\,m_f(R)^{p_0q_0}$, which is impossible for
unbounded $m_f(R)$ because $p_0q_0>1$. The first inequality then
also bounds $m_g(R)$. Thus their sum is uniformly bounded.

\emph{Larger balls and potential inputs.}
The exact source transforms at scales $R$ and $R\rho$, and the
unit-ball bound at the latter scale, give
\begin{equation}
 \begin{split}
 \int_{B_\rho}g_R
 &=\rho^\lambda\frac{s_{R\rho}}{s_R}\int_{B_1}g_{R\rho}
 \le C\rho^{\lambda-\alpha_*},\\
 \int_{B_\rho}f_R
 &=\rho^\lambda\frac{t_{R\rho}}{t_R}\int_{B_1}f_{R\rho}
 \le C\rho^{\lambda-\beta_*}\qquad(\rho\ge1).
 \end{split}
 \label{r6:large-source-balls}
\end{equation}
For $x\in B_4$, the source outside $B_8$ is at a comparable
radial distance. Dyadic summation therefore yields
\begin{equation}
 \int_{|y|>8}K_R(x-y)(f_R(y)+g_R(y))\,dy
 \le C\sum_{j\ge3}(2^{-j\alpha_*}+2^{-j\beta_*})\le C.
 \label{r6:uniform-source-tails}
\end{equation}
For $x\in B_2$, the portion in $B_8$ with $|x-y|\ge h$ is at
most $Ch^{-\lambda}\int_{B_8}(f_R+g_R)$. This proves the second
line of \eqref{r6:uniform-source-input}. Finally,
$\sup_{y\in B_8}\int_{B_4}K_R(x-y)\,dx\le C$, since
$\lambda<n$. Tonelli and \eqref{r6:uniform-source-tails} give
uniform $L^1(B_4)$ bounds for $u_R,v_R$; the source bounds there
follow from \eqref{r6:large-source-balls}. This proves the lemma.
\end{proof}

\subsection{Uniform energy and the final contradiction}

\begin{proof}[Proof of Theorem~\ref{r7:master-front}]
Suppose a pair exists. Lemma~\ref{r7:source-parameters} gives
\eqref{r7:forced-positivity}, and Lemma~\ref{r7:normalization}
provides the complete normalized family with a common source input
$M$. This family is closed under further scaling:
$(K_R)_\rho=K_{R\rho}$, $(F_s)_t=F_{st}$, and similarly for
$G$ and the coefficients. Uniqueness in
\eqref{r6:implicit-amplitudes} identifies the subsequent amplitude
normalizations with $s_{R\rho}/s_R$ and $t_{R\rho}/t_R$.
The local energy of every normalized member is finite, since
\begin{equation}
 f_R(x)u_R(x)+g_R(x)v_R(x)
 =\frac{R^\sigma}{s_Rt_R}[fu+gv](Rx).
 \label{r6:exact-energy-density}
\end{equation}
% No uniform energy bound has been assumed or used to obtain the
% source estimates.

Fix the cutoff from Theorem~\ref{gd:realization} for the common
Dini modulus. For a normalized member suppress the subscript $R$,
and use the definitions $E_f,E_g,E,\pi,\Phi,D,w$ from
Proposition~\ref{r6:general-pressure}. Set
\[
 W_P=\int w\{bP_F(u)+aP_G(v)\},\qquad
 T_K=\iint w(x)\theta_K(|x-y|)D(x,y)\,d\pi.
\]
Here the symbols $F,G,K,a,b$ temporarily denote the normalized
data. Their lower response derivatives and all coefficient, kernel
and source bounds are uniform. Proposition~\ref{r6:general-pressure}
therefore gives, for every $\xi>0$,
\begin{equation}
 \begin{split}
 |E_f-E_g|&\le\xi E+C_\xi,\\
 \iint|\Phi-\eta(x)+w(x)D|\,d\pi&\le\xi E+C_\xi,\\
 W_P&\le T_K+\xi E+C_\xi.
 \end{split}
 \label{r6:master-pressure-inputs}
\end{equation}
The bound $0\le\theta_K\le\Lambda$ implies
\begin{equation}
 \begin{split}
 J_K:=\iint\theta_K(|x-y|)\Phi\,d\pi
 &\le\Lambda E_f-T_K+\Lambda\xi E+C_\xi\\
 &\le\frac\Lambda2E-T_K+
                  \frac{3\Lambda}{2}\xi E+C_\xi.
 \end{split}
 \label{r6:weighted-localization}
\end{equation}
All signed quantities are finite by the compact marginals and
bounded logarithmic kernel slope.

The weak transport virial inequality in
Lemma~\ref{r6:transport-virial} gives
\begin{equation}
 \int\eta\{(n+B)bP_F(u)+(n+A)aP_G(v)\}-W_P\le J_K.
 \label{r6:nonlinear-virial-input}
\end{equation}
The lemma applies to the present locally Lipschitz responses:
on the positive compact solution range used in its finite
transport increments, each primitive is $C^{1,1}$, and its Taylor
remainder is bounded by the square of the solution increment.
% Thus no upper response derivative uniform over the normalization
% family is needed. The weak Euler lower bounds are inserted before
% the origin cutoff is removed, as in that lemma.

The signs needed for primitive coercivity were proved in
Lemma~\ref{r7:source-parameters}: $n+A,n+B>0$.
Consequently,
\[
 \int\eta\{(n+B)bP_F(u)+(n+A)aP_G(v)\}
 \ge\kappa_fE_f+\kappa_gE_g,
 \quad\kappa_f=(n+B)\mu_F,\quad\kappa_g=(n+A)\mu_G.
\]
Combining \eqref{r6:master-pressure-inputs}--
\eqref{r6:nonlinear-virial-input} cancels $T_K$ and gives
\[
 \kappa_fE_f+\kappa_gE_g
 \le\frac\Lambda2E+(1+3\Lambda/2)\xi E+C_\xi.
\]
Since $\kappa_f+\kappa_g=\Lambda+\Gamma_P$,
\[
 \kappa_fE_f+\kappa_gE_g
 \ge\frac{\Lambda+\Gamma_P}{2}E
       -\frac{|\kappa_f-\kappa_g|}{2}|E_f-E_g|.
\]
Choose $\xi>0$ so that
\[
 [1+3\Lambda/2+|\kappa_f-\kappa_g|/2]\xi\le\Gamma_P/4.
\]
The first line of \eqref{r6:master-pressure-inputs} then yields
\begin{equation}
 E_R\le C_*\quad(R\ge R_0),
 \label{r6:uniform-nonlinear-energy}
\end{equation}
with the same constant for the whole normalized family.

Define
\begin{equation}
 \mathcal H_R=\frac{R^{-\lambda}}{s_Rt_R}.
 \label{r6:energy-amplification-factor}
\end{equation}
The amplitude estimates \eqref{r6:amplitude-ratio} imply
\begin{equation}
 \frac{\mathcal H_{R\rho}}{\mathcal H_R}
 =\frac{\rho^{-\lambda}}
 {(s_{R\rho}/s_R)(t_{R\rho}/t_R)}
 \ge c\rho^{\alpha_*+\beta_*-\lambda}
 \quad(\rho\ge1).
 \label{r6:energy-amplification-growth}
\end{equation}
The exponent is positive by \eqref{r7:source-gap-identity}.
Fixing $R=R_0$ therefore gives $\mathcal H_R\to\infty$.
On the other hand, \eqref{r6:exact-energy-density} and
$\eta=1$ on $B_1$ imply
\begin{equation}
 C_*\ge E_R\ge\mathcal H_R
       \int_{B_R}[bF(u)u+aG(v)v].
 \label{r6:final-nonlinear-amplification}
\end{equation}
The original energy on any fixed positive-radius ball is finite
and strictly positive. For large $R$ it is included in the last
integral, contradicting $\mathcal H_R\to\infty$. This proves
the theorem.
\end{proof}

\subsection{Power windows, mixtures and response perturbations}

\begin{corollary}
\label{r7:two-sided-elasticity-corollary}
Let $K,a,b$ satisfy the kernel, weak Euler and coefficient modulus
assumptions of Theorem~\ref{r7:master-front}, with arbitrary real
$A=A_-$ and $B=B_-$. Let $F,G$ be positive continuous responses
vanishing at zero and locally Lipschitz on $(0,\infty)$. Suppose
\begin{equation}
 \begin{split}
 0<q_-&\le sF'(s)/F(s)\le q_+<\infty,\\
 0<p_-&\le sG'(s)/G(s)\le p_+<\infty
 \quad\hbox{for almost every }s>0,\qquad p_-q_->1,
 \end{split}
 \label{r7:two-sided-response-bounds}
\end{equation}
and
\begin{equation}
 \frac{n+A}{p_++1}+\frac{n+B}{q_++1}>\Lambda.
 \label{r6:nonlinear-gap}
\end{equation}
Then no positive pair satisfies \eqref{v5:master-system} and
\eqref{eq:main-local-energy}.
\end{corollary}

\begin{proof}
Integrating the upper logarithmic derivative bound between $sz$
and $s$, for $0<z<1$, gives
$F(sz)\ge z^{q_+}F(s)$. Integrating in $z$ yields
$P_F(s)\ge sF(s)/(q_++1)$, with the analogous bound for $G$.
Take $q_0=q_-$, $p_0=p_-$,
$\mu_F=1/(q_++1)$ and $\mu_G=1/(p_++1)$ in the master theorem.
No restriction on the signs of $n+A$ and $n+B$ is inserted here:
the theorem obtains their necessary positivity from the putative
finite potentials before applying primitive coercivity.
\end{proof}

For the power responses $F(s)=s^q$, $G(s)=s^p$, this gives the
condition
\begin{equation}
 \Gamma_\Lambda:=\frac{n+A_-}{p+1}
                 +\frac{n+B_-}{q+1}-\Lambda>0,
 \qquad p,q>0,\quad pq>1,
 \label{r6:kernel-gap}
\end{equation}
with arbitrary real Euler bounds. Thus the power case and the
bounded logarithmic derivative case follow from one proof.

\begin{corollary}
\label{r6:mixed-powers-corollary}
Under the coefficient and kernel hypotheses of
Corollary~\ref{r7:two-sided-elasticity-corollary}, let
\[
 F(s)=\sum_{i=1}^N c_i s^{q_i},\qquad
 G(s)=\sum_{j=1}^M d_j s^{p_j},\qquad c_i,d_j>0,
\]
where all exponents are positive. Put
$q_- =\min_iq_i$, $q_+=\max_iq_i$,
$p_-=\min_jp_j$, $p_+=\max_jp_j$.
If $p_-q_->1$ and \eqref{r6:nonlinear-gap} holds, there is no
pair in the stated positive potential and local energy class.
The same assertion holds for
\[
 F(s)=\int_{[q_-,q_+]}s^q\,d\nu_F(q),\qquad
 G(s)=\int_{[p_-,p_+]}s^p\,d\nu_G(p),
\]
where $\nu_F,\nu_G$ are finite nonzero positive measures.
\end{corollary}

\begin{proof}
Each logarithmic derivative is a positive weighted average of the
response exponents. For the integral mixtures, differentiation
under the integral is justified locally uniformly on the positive
axis by the compact exponent supports, and gives the same formula.
Apply Corollary~\ref{r7:two-sided-elasticity-corollary}.
\end{proof}

\begin{corollary}
\label{r6:response-perturbation-corollary}
Let $K,a,b$ satisfy the same coefficient and kernel assumptions,
with arbitrary real $A,B$. Let
\[
 F(s)=s^qL_F(s),\qquad G(s)=s^pL_G(s),
\]
where $L_F,L_G$ are positive and locally Lipschitz on
$(0,\infty)$, and almost everywhere
\[
 \left|\frac{sL_F'(s)}{L_F(s)}\right|\le\varepsilon_F<q,
 \qquad
 \left|\frac{sL_G'(s)}{L_G(s)}\right|\le\varepsilon_G<p.
\]
If
\[
 (p-\varepsilon_G)(q-\varepsilon_F)>1,\qquad
 \frac{n+A}{p+\varepsilon_G+1}
 +\frac{n+B}{q+\varepsilon_F+1}>\Lambda,
\]
there is no pair satisfying \eqref{v5:master-system} and
\eqref{eq:main-local-energy}.
This includes the responses
\[
 F(s)=s^q[\log(e+s)]^{\beta},\qquad
 G(s)=s^p[\log(e+s)]^{\alpha}
\]
whenever the displayed inequalities hold with
$\varepsilon_F=|\beta|$, $\varepsilon_G=|\alpha|$.
\end{corollary}

\begin{proof}
The response logarithmic derivatives lie in
$[q-\varepsilon_F,q+\varepsilon_F]$ and
$[p-\varepsilon_G,p+\varepsilon_G]$ almost everywhere.
Their positive lower bounds also show, by integration, that both
responses tend to zero at zero. Apply
Corollary~\ref{r7:two-sided-elasticity-corollary}.
For the logarithmic factors use
\[
 \frac{d\log\log(e+s)}{d\log s}
 =\frac{s}{(e+s)\log(e+s)}\in[0,1].
\]
\end{proof}

\subsection{Responses with no upper logarithmic derivative bound}

The primitive hypothesis is stronger in scope than a finite
logarithmic derivative window. The next example supplies a class
which is covered by the master theorem and has no such upper
window.

\begin{proposition}
\label{r7:unbounded-response-elasticity}
In dimension $n=5$, with $\sigma=2$ and $a=b=1$, there are smooth
positive responses $F=G$ such that
\[
 \frac{sF'(s)}{F(s)}\ge2,\qquad
 \sup_{s>0}\frac{sF'(s)}{F(s)}=\infty,
\]
for which the system $u=I_2(F(v))$, $v=I_2(F(u))$ has no finite
positive pair continuous off zero with
$\int_{B_1}[uF(u)+vF(v)]<\infty$.
\end{proposition}

\begin{proof}
Let $\chi\in C^\infty(\mathbb R)$ be nondecreasing, equal to
zero on $(-\infty,0]$ and to one on $[1,\infty)$.
Choose $\tau>0$ so that
\[
 M:=\tau\sum_{k\ge2}e^{-k}<\log(10/9),\qquad
 \varepsilon_k=e^{-k^2},
\]
and define
\[
 h(t)=\tau\sum_{k\ge2}e^{-k}
                   \chi((t-3k)/\varepsilon_k),\qquad
 F(s)=G(s)=s^2e^{h(\log s)}.
\]
The sum is locally finite in $t$, $h$ is smooth and nondecreasing,
and $0\le h\le M$. Moreover, $h=0$ on $(-\infty,6]$, so
$F(s)=s^2$ near zero and $F(0)=0$ gives a smooth extension there.
Choose $\xi\in(0,1)$ with $\chi'(\xi)>0$.
The transition intervals are disjoint, and
\[
 h'(3k+\xi\varepsilon_k)
   =\tau e^{k^2-k}\chi'(\xi)\longrightarrow\infty.
\]
Thus the logarithmic response derivative equals $2+h'(\log s)$,
has lower bound two, and has no finite upper bound.
For every $s>0$,
\[
 P_F(s)\ge\int_0^s t^2\,dt=s^3/3,\qquad
 sF(s)\le e^M s^3.
\]
Hence $\mu_F=\mu_G=e^{-M}/3$ are valid primitive bounds.
The Riesz kernel has $\lambda=\Lambda=3$ and $A=B=0$, giving
\[
 \Gamma_P=10e^{-M}/3-3>0.
\]
Theorem~\ref{r7:master-front} applies with $p_0=q_0=2$.
\end{proof}

\subsection{Why asymptotic response exponents do not suffice}

The full-range primitive hypothesis cannot be replaced merely by
strict subcritical regular variation at zero and infinity.
Here regular variation of index $q$ at an endpoint means that
$F(cs)/F(s)\to c^q$ for every fixed $c>0$ as $s$ tends to that
endpoint.

\begin{example}
\label{r6:regular-variation-counterexample}
There is a positive strictly increasing
$F\in C^\infty((0,\infty))\cap C([0,\infty))$, with $F(0)=0$,
which is regularly varying of index $4$ both at zero and infinity,
yet the system
\[
 u=I_2(F(v)),\qquad v=I_2(F(u))\quad\hbox{in }\mathbb R^3
\]
has a smooth positive entire pair with finite total nonlinear
energy. The index $4$ is strictly below the scalar critical index
$5$ in this dimension.
\end{example}

\begin{proof}
Choose $\rho\in C^\infty([0,\infty))$, with $0\le\rho\le1$,
$\rho=1$ on $[0,1]$ and $\rho=0$ on $[2,\infty)$. Put
\[
 h(s)=\int_0^s\rho(t)\,dt,\qquad
 F(s)=s^4\left[\frac29+\frac{16}{9}h(s)^3\right].
\]
The function $h$ is nondecreasing, equals $s$ for $s\le1$, and
is a positive constant for $s\ge2$. Thus $F$ is positive and
strictly increasing on $(0,\infty)$, is smooth, and satisfies
\[
 F(s)\sim\frac29s^4\quad(s\downarrow0),\qquad
 F(s)\sim C s^4\quad(s\to\infty)
\]
for a positive constant $C$.
Take
\[
 u(x)=v(x)=(1+|x|^2)^{-1/3}.
\]
Direct radial differentiation gives
\begin{equation}
 -\Delta u
 =\frac29(1+|x|^2)^{-4/3}
   +\frac{16}{9}(1+|x|^2)^{-7/3}
 =\frac29u^4+\frac{16}{9}u^7=F(u),
 \label{r6:explicit-nonpower-source}
\end{equation}
because $0<u\le1$ and hence $h(u)=u$.
For completeness, write $f(r)=F(u(r))$. The radial equation gives
\[
 -r^2u'(r)=\int_0^r s^2f(s)\,ds.
\]
Using $u(\infty)=0$ and Tonelli's theorem, integration from $r$
to infinity yields
\[
 u(r)=\frac1r\int_0^r s^2f(s)\,ds+
                      \int_r^\infty s f(s)\,ds.
\]
Both integrals are finite: $f$ is bounded at zero and is
$O(r^{-8/3})$ at infinity. This is the Newton potential formula
with $K_2(z)=(4\pi|z|)^{-1}$, so the asserted exact identities
hold everywhere, including at zero. Finally,
\[
 \int_{\mathbb R^3}uF(u)
 =4\pi\int_0^\infty r^2
       \left[\frac29u(r)^5+\frac{16}{9}u(r)^8\right]dr<\infty:
\]
the integrand is $O(r^2)$ near zero and $O(r^{-4/3})$ at infinity.
The primitive gap fails on the full response range. Indeed,
\[
 \frac{P_F(1)}{F(1)}
 =\frac{2/45+2/9}{2}=\frac2{15},\qquad
 \frac{F'(1)}{F(1)}=\frac{20}{3}>5.
\]
Thus any admissible primitive constants satisfy
$\mu_F,\mu_G\le2/15$, and
$3\mu_F+3\mu_G\le4/5<1=\Lambda$.
The intermediate response amplitudes prevent the positive primitive
gap, although both asymptotic indices are subcritical.
\end{proof}

The master theorem concerns exact positive potential equations.
It does not assert nonexistence for arbitrary monotone responses,
for asymptotic response exponents alone, or for unequal kernels.
Its PDE consequences require the independent representation theorem
proved in Subsection~\ref{r7:nonlinear-pde-section}.
A solution-dependent effective coefficient would not justify an
application: its required dilation modulus would depend on unknown
solution oscillation. Response normalization instead preserves one
class with uniform source, pressure and primitive bounds.

\section{Coefficient and kernel applications}
\label{sec:coefficient-proof}

The following consequences use the single primitive theorem.

\begin{corollary}
\label{cor:power}
Let $0<\sigma<n$, $p,q>0$, $pq>1$ and $A,B\in\mathbb R$.
If $(n+A)/(p+1)+(n+B)/(q+1)>n-\sigma$, there is no finite
strictly positive pair, continuous away from zero, satisfying the
Riesz integral system $u=I_\sigma(av^p)$, $v=I_\sigma(bu^q)$ with
$a(x)=|x|^A$, $b(x)=|x|^B$ and finite local nonlinear energy.
\end{corollary}
\begin{proof}
The Euler derivatives equal $A,B$. On the annulus the normalized coefficient
logarithms are Lipschitz, so $\Omega(s)=s$ satisfies
Theorem~\ref{thm:dini-front}.
\end{proof}

\subsection{A Dini coefficient family without a uniform H\"older modulus}

The following construction exhibits a coefficient class beyond every
uniform H\"older bound on normalized dilations. Each coefficient is
nonetheless smooth on every compact set away from zero.

\begin{proposition}
\label{r5:nonholder-coefficients}
Let $n\ge1$, $A\in\mathbb R$, and let $h$ be a positive smooth
function on $\mathbb S^{n-1}$. Choose a smooth nondecreasing function
$\chi:\mathbb R\to[0,1]$ that vanishes on $(-\infty,0]$ and equals
one on $[1,\infty)$. Set
\[
 \varepsilon_k=e^{-k^2},\qquad
 \psi(t)=\sum_{k=2}^\infty e^{-k}
       \chi\bigl((t-3k)/\varepsilon_k\bigr),\qquad
 a(x)=|x|^A h(x/|x|)e^{\psi(\log|x|)}.
\]
Then $a$ is positive and smooth away from zero, and
$\vartheta_a\ge A$, while $\sup_{x\ne0}\vartheta_a(x)=\infty$.
Its normalized logarithmic dilations satisfy the coefficient estimate in
\eqref{eq:front-dini-orbit} uniformly for all $R>0$, with a Dini
modulus of the form
\[
 \Omega(0)=0,\qquad
 \Omega(s)=\exp\!\left[-\frac14\sqrt{\log(e/s)}\right]
       \quad(0<s\le1/4),
\]
extended constantly to $[1/4,4]$. However, for every $\alpha>0$ and
every $R_0>0$ there is no finite constant $H$ such that
\[
 |\log a_R(x)-\log a_R(y)|\le H|x-y|^\alpha
 \quad(R\ge R_0,\ x,y\in\mathcal A).
\]
For $n=1$, the comparisons are within the same sign branch.
\end{proposition}

\begin{proof}
The defining sum is locally finite on $\mathbb R$, and
$0\le\psi\le\sum_{k\ge2}e^{-k}$. Thus $a$ is smooth away from zero
and $\vartheta_a=A+\psi'(\log|x|)\ge A$.
Choose $\xi\in(0,1)$ with $\chi'(\xi)>0$. The identity
$\psi'(3k+\xi\varepsilon_k)=e^{k^2-k}\chi'(\xi)$
shows that its supremum is infinite.
Write $M=\max(1,\|\chi'\|_\infty)$. For $0<d\le1/4$ and every
$t\in\mathbb R$,
\begin{equation}
 |\psi(t+d)-\psi(t)|
 \le M\sum_{k=2}^\infty e^{-k}\min(1,de^{k^2}).
 \label{r5:step-modulus}
\end{equation}
For small $d$, put $L=\log(e/d)$ and
$K=\lfloor\sqrt{L/2}\rfloor$. Splitting the sum at $K$ gives
\[
 \sum_{k=2}^\infty e^{-k}\min(1,de^{k^2})
 \le dK e^{K^2}+\frac{e^{-K-1}}{1-e^{-1}}
 \le C\bigl(\sqrt L e^{-L/2}+e^{-\sqrt{L/2}}\bigr)
 \le C e^{-\sqrt L/2}.
\]
Changing $C$ covers the remaining $d\le1/4$. The same estimate
holds for negative increments by translating $t$.

For $x,y\in\mathcal A$,
$|\log|x|-\log|y||\le2|x-y|$. The normalization term
$\log a(Re_0)$ cancels in differences. The radial power and the
smooth angular factor contribute at most $C|x-y|$; for $n=1$ the
angular factor is constant on each branch. Applying
\eqref{r5:step-modulus} with the logarithmic radial increment and
enlarging the constant therefore gives
\[
 |\log a_R(x)-\log a_R(y)|\le C\Omega(|x-y|)
 \quad(R>0).
\]
For distances bounded away from zero this follows instead from the
boundedness of $\psi$ and of the smooth terms on $\mathcal A$.
Finally, the substitution $z=\sqrt{\log(e/s)}$ shows that
\[
 \int_0^{1/4}\Omega(s)\,\frac{ds}{s}
 =2\int_{\sqrt{\log(4e)}}^\infty ze^{-z/4}\,dz<\infty.
\]

To disprove a uniform H\"older estimate, take
$R_k=e^{3k}$, $y_k=e_0$, and $x_k=e^{\varepsilon_k}e_0$.
The $k$th step changes from zero to one between the two logarithmic
radii, while every other step is constant there. Hence
\[
 \log a_{R_k}(x_k)-\log a_{R_k}(y_k)
       =A\varepsilon_k+e^{-k}.
\]
For large $k$ its absolute value is at least $e^{-k}/2$, whereas
$|x_k-y_k|\le2e^{-k^2}$. The ratio to $|x_k-y_k|^\alpha$ tends
to infinity for every $\alpha>0$, with $R_k\ge R_0$ eventually.
\end{proof}

\begin{corollary}
\label{r5:nonholder-nonexistence}
Let $0<\sigma<n$, $p,q>0$, $pq>1$, and
\[
 \frac{n+A}{p+1}+\frac{n+B}{q+1}>n-\sigma.
\]
Take $a$ as in Proposition~\ref{r5:nonholder-coefficients} and
$b(x)=|x|^B g(x/|x|)$, where $g>0$ is smooth on the sphere.
There is no finite strictly positive pair, continuous away from
zero, satisfying
\[
 u=I_\sigma(av^p),\qquad v=I_\sigma(bu^q),\qquad
 \int_{B_1}(bu^{q+1}+av^{p+1})<\infty.
\]
The corresponding nonexistence conclusion also holds in each
of the three solution classes of Corollary~\ref{cor:pde-front}.
\end{corollary}

\begin{proof}
Proposition~\ref{r5:nonholder-coefficients} supplies the uniform
Dini modulus and the Euler lower bound $A$ for $a$. The normalized
logarithmic dilations of $b$ are Lipschitz on the annulus, and
$\vartheta_b=B$. Thus both coefficients satisfy the hypotheses of
Theorem~\ref{thm:dini-front} with the same modulus $\Omega$.
Apply that theorem and Corollary~\ref{cor:pde-front}, respectively.
\end{proof}

\subsection{The local energy condition at the origin}

\begin{example}
\label{r5:energy-boundary}
The local energy hypothesis in Theorem~\ref{thm:dini-front} cannot
be omitted from the class of punctured integral pairs. Indeed, take
$n=3$, $\sigma=2$, $p=q=4$, $a=b=1$, and
\[
 u(x)=v(x)=c|x|^{-2/3},\qquad c^3=\frac29.
\]
These functions satisfy $u=I_2(v^4)$ and $v=I_2(u^4)$ at every
$x\ne0$, and the subcriticality gap is $\gamma_-=1/5>0$.
Their sources are locally integrable and their potentials are
finite off zero, but
\[
 \int_{B_1}(u^5+v^5)
 =8\pi c^5\int_0^1 r^{-4/3}\,dr=\infty.
\]
\end{example}
\begin{proof}
For the normalization $K_2(x)=(4\pi|x|)^{-1}$, the radial
Newton potential formula gives, for $r>0$,
\[
 I_2(c^4|\cdot|^{-8/3})(r)
 =c^4\left(\frac1r\int_0^r s^{-2/3}\,ds
                 +\int_r^\infty s^{-5/3}\,ds\right)
 =\frac92c^4r^{-2/3}=cr^{-2/3}.
\]
The two convergent integrals establish both the local source
integrability and the finite potential tail. The displayed energy
divergence follows directly in polar coordinates. This pair is
singular at zero, so it also lies outside the origin-continuous
PDE classes in Corollary~\ref{cor:pde-front}.
\end{proof}

\begin{example}
\label{r6:logperiodic-kernel}
For $0<\lambda<n$ and $|\varepsilon|<\lambda$, consider
\[
 K_\varepsilon(z)=|z|^{-\lambda}
       \exp\{\varepsilon\sin(\log|z|)\}.
\]
This positive radial kernel is not homogeneous when $\varepsilon\ne0$.
It is comparable to $|z|^{-\lambda}$, and its logarithmic slope is
$\lambda-\varepsilon\cos(\log r)$.
Theorem~\ref{r6:radial-kernel-nonexistence} applies whenever
\[
 \frac{n+A_-}{p+1}+\frac{n+B_-}{q+1}-\lambda>|\varepsilon|.
\]
More generally one may replace $\varepsilon\sin t$ by any bounded
$C^1$ function $h(t)$ with $\|h'\|_\infty<\lambda$, using
$\Lambda=\lambda+\|h'\|_\infty$.
The argument uses a uniformly controlled kernel orbit rather than
homogeneity of one fixed kernel. 
% No assertion with only an asymptotic
% slope bound is made.
\end{example}

\begin{corollary}
\label{r6:ellipsoidal-kernel}
Let $L\in GL(n)$, $c_1>0$, and let $k$ satisfy
\eqref{r6:radial-kernel-class}. The conclusion of
Theorem~\ref{r6:radial-kernel-nonexistence} remains valid with
$K(z)=c_1k(|Lz|)$, under the same coefficient and gap hypotheses.
For power responses, arbitrary ellipsoidal Riesz kernels have the original
strict gap $\gamma_->0$.
% with no small-anisotropy assumption.
\end{corollary}

\begin{proof}
Under $\zeta=Lx$, put
$U(\zeta)=u(L^{-1}\zeta)$ and $V(\zeta)=v(L^{-1}\zeta)$.
Their radial-kernel coefficients are
\[
 \widetilde a(\zeta)=c_1|\det L|^{-1}a(L^{-1}\zeta),\qquad
 \widetilde b(\zeta)=c_1|\det L|^{-1}b(L^{-1}\zeta).
\]
They satisfy \eqref{r6:weak-euler-front} with the same constants.
Their normalized logarithmic Dini condition follows from the original
one on sufficiently close points of a fixed transformed annulus:
rescale each such pair into a common original annulus, obtaining
$C\Omega(C_L|\zeta-\zeta'|)$. A finite chain of close pairs supplies
a uniform bound for larger distances. The transformed annulus is
connected for $n\ge2$; for $n=1$ work within the same sign branch.
A finite nondecreasing extension of this local modulus remains Dini.
The starting scale $R_0$ may be increased by a constant depending on $L$.
Local energy is unchanged up to a constant Jacobian and replacement
of a ball by an ellipsoid, and is finite on every bounded set by
continuity off zero. Apply the radial theorem to $U,V$.
For $k(r)=r^{-\lambda}$ one has $\Lambda=\lambda$, proving the
last assertion.
\end{proof}

\subsection{Continuous coefficients with singular radial variation}

\begin{example}
\label{r6:cantor-coefficient}
Let $C:\mathbb R\to[0,1]$ be the Cantor staircase on $[0,1]$,
extended constantly outside this interval, and let $h_a,h_b$ be
positive smooth functions on $\mathbb S^{n-1}$. Then
\[
 a(r\omega)=r^{A_-}h_a(\omega)e^{C(\log r)},\qquad
 b(r\omega)=r^{B_-}h_b(\omega)
\]
are positive continuous coefficients satisfying
\eqref{v5:euler-lower}. Their normalized logarithmic dilations have
a common H\"older modulus, hence a Dini modulus, although $a$ is not
$C^1$ off zero. Consequently Theorem~\ref{r7:master-front} applies to power responses whenever
\eqref{r6:kernel-gap} holds, including to the PDE classes of
Corollary~\ref{cor:pde-front} when $K=K_\sigma$.
\end{example}
\begin{proof}
The staircase is nondecreasing and H\"older continuous of exponent
$\log2/\log3$. Thus $r^{-A_-}a(r\omega)$ is nondecreasing and
the finite-dilation Euler bound holds; the bound for $b$ is equality.
On the fixed annulus, $\log r$ is Lipschitz and the angular functions
are smooth. The H\"older constant of
$C(\log R+\log r)$ is independent of $R$, giving the claimed uniform
modulus. In dimension one comparisons stay within a sign branch.
The nonconstant singular staircase has a nonzero singular derivative
measure, which excludes $C^1$ regularity of $a$ on the corresponding
annulus. Apply the master theorem and its Riesz PDE consequence.
\end{proof}

\section{Finite-energy rigidity and the global Euler hypothesis}
\label{r6:critical-boundary-section}

Finite total interaction permits a global virial limit without a
coefficient orbit condition.  This gives a rigidity statement at the
critical gap.  It also identifies a boundary of the preceding
nonexistence theorem: an Euler condition imposed only outside a compact
set does not suffice.

\begin{theorem}
\label{r6:critical-rigidity}
Let $n\geq1$, $0<\sigma<n$, and $p,q>0$.  Let
$a,b,u,v$ be strictly positive and continuous on
$\mathbb R^n\setminus\{0\}$, and suppose that the finite potential
identities
\[
 u=I_\sigma(av^p),\qquad v=I_\sigma(bu^q)
\]
hold at every nonzero point.  Assume that
\begin{equation}
 J:=\int_{\mathbb R^n}bu^{q+1}
   =\int_{\mathbb R^n}av^{p+1}<\infty.
 \label{r6:total-energy}
\end{equation}
The equality in \eqref{r6:total-energy} is automatic by Tonelli's
 theorem if either integral is finite.  Suppose that, for real $A,B$,
\begin{equation}
 \begin{split}
 \mu_a&:=\operatorname{div}(xa)-(n+A)a\,dx\geq0,\\
 \mu_b&:=\operatorname{div}(xb)-(n+B)b\,dx\geq0
 \end{split}
 \quad\text{in }\mathcal D'(\mathbb R^n\setminus\{0\}),
 \label{r6:critical-euler-measures}
\end{equation}
and set
\[
 \gamma=\frac{n+A}{p+1}+\frac{n+B}{q+1}-(n-\sigma).
\]
If $\gamma>0$, no such pair exists.  If $\gamma=0$, then
$\mu_a=\mu_b=0$.  In the latter case there are strictly positive
continuous functions $h_a,h_b$ on $\mathbb S^{n-1}$ such that
\begin{equation}
 a(r\omega)=r^A h_a(\omega),\qquad
 b(r\omega)=r^B h_b(\omega)
 \quad(r>0,\ \omega\in\mathbb S^{n-1}).
 \label{r6:critical-homogeneity}
\end{equation}
In particular, at the critical gap either nonzero Euler defect measure
excludes finite-total-energy positive pairs. 
% % No Dini orbit assumption
% is imposed here.
\end{theorem}

\begin{proof}
Put $\lambda=n-\sigma$, $f=bu^q$, $g=av^p$, and
\[
 d\pi(x,y)=f(x)g(y)K_\sigma(x-y)\,dx\,dy.
\]
Tonelli's theorem gives $\pi(\mathbb R^n\times\mathbb R^n)=J>0$.
The positive distributions in
\eqref{r6:critical-euler-measures} are locally finite Radon measures
on the punctured space.

Choose a smooth radial nonincreasing function $\eta$, equal to one on
$B_1$ and zero outside $B_2$, with $0\leq\eta\leq1$.  Set
$\eta_R(x)=\eta(x/R)$, $X_R(x)=x\eta_R(x)$,
$w_R=-x\cdot\nabla\eta_R$, and
\[
 \Phi_R(x,y)=
 \frac{(X_R(x)-X_R(y))\cdot(x-y)}{|x-y|^2}.
\]
Apply the compact radial transport identity from
Lemma~\ref{r6:transport-virial} first with an additional cutoff
$\chi_\rho$ at zero.  Expanding the coefficient divergences separates
the absolutely continuous terms from $\mu_a,\mu_b$.  The terms
containing $x\cdot\nabla\chi_\rho$ tend to zero, since they are bounded
by a constant times
$\int_{B_{2\rho}}(fu+gv)$.  The fields $X_R\chi_\rho$ have uniformly
bounded Lipschitz constants as $\rho\downarrow0$; their kernel factors
therefore converge by dominated convergence against the finite
measure $\pi$.  Retain the nonnegative Euler measure terms and use
Fatou's lemma.  This gives
\begin{equation}
 \begin{split}
 &\frac{n+B}{q+1}\int\eta_Rfu
  +\frac{n+A}{p+1}\int\eta_Rgv
  -\frac1{q+1}\int w_Rfu-\frac1{p+1}\int w_Rgv\\
 &\qquad
  +\frac1{q+1}\int\eta_Ru^{q+1}\,d\mu_b
  +\frac1{p+1}\int\eta_Rv^{p+1}\,d\mu_a
 \leq\lambda\iint\Phi_R\,d\pi.
 \end{split}
 \label{r6:critical-measure-virial}
\end{equation}
The measure integrals are over $\mathbb R^n\setminus\{0\}$; no
coefficient or solution value at zero is needed.

As $R\to\infty$, both cutoff energy integrals tend to $J$.
The functions $w_R$ are uniformly bounded and supported in
$B_{2R}\setminus B_R$, so their energy integrals tend to zero.
Moreover, $\Phi_R\to1$ at every distinct fixed pair and
$|\Phi_R|\leq C$, because $X_R$ has a uniform Lipschitz bound.
Dominated convergence and then Fatou's lemma in
\eqref{r6:critical-measure-virial} yield
\begin{equation}
 \gamma J
 +\frac1{q+1}\int u^{q+1}\,d\mu_b
 +\frac1{p+1}\int v^{p+1}\,d\mu_a\leq0.
 \label{r6:critical-defect-obstruction}
\end{equation}
All three terms are nonnegative when $\gamma\geq0$.  Since $J>0$,
$\gamma>0$ is impossible.  When $\gamma=0$, strict positivity and
continuity of $u,v$ on compact annuli imply
$\mu_a=\mu_b=0$.

Finally, the distributional identity
$\operatorname{div}(xa)=(n+A)a\,dx$ says in polar logarithmic
coordinates that $e^{-At}a(e^t\omega)$ has zero derivative in $t$.
It is independent of $t$ almost everywhere; continuity extends this
identity to every $t$ and every direction.  This proves the first
identity in \eqref{r6:critical-homogeneity}, and the second follows
in the same way.  For $n=1$ the two rays are treated separately.
\end{proof}

\begin{remark}
\label{r6:critical-scope}
The total-energy assumption is essential to this proof: it supplies a
finite global interaction measure.  The theorem does not exclude
critical homogeneous-coefficient solutions, and it does not assert a
critical nonexistence theorem under local energy alone.  Its
subcritical part is the global finite-energy Pohozaev obstruction;
the Dini theorem above treats pairs for which total energy is not
assumed finite.
\end{remark}

\begin{proposition}
\label{r6:compact-euler-defect}
There are strictly positive radial functions
$a,u\in C^\infty(\mathbb R^5)$ such that
\begin{equation}
 -\Delta u=au^2,\qquad u=I_2(au^2),\qquad
 a(x)=1,\quad u(x)=2|x|^{-2}\quad(|x|\geq1),
 \label{r6:compact-defect-pair}
\end{equation}
and $\int_{\mathbb R^5}au^3<\infty$.
The normalized logarithmic coefficient orbits have a uniform
Lipschitz modulus.  Consequently, for $p=q=2$ and $a=b$, even an
exterior Euler bound with strict gap $1/3$ and coefficients constant
outside a compact set does not imply nonexistence.
\end{proposition}

\begin{proof}
Choose $h\in C^\infty([0,\infty))$ with $0<h(r)\leq1$ for $r>0$,
$h'\geq0$, $h(r)=r^4$ near zero, and $h(r)=1$ for $r\geq1$.
For example, if $\chi$ is a smooth nondecreasing cutoff, zero on
$[0,1/4]$ and one on $[3/4,\infty)$, then
$h(r)=(1-\chi(r))r^4+\chi(r)$ has these properties.
Define
\begin{equation}
 M(r)=4rh(r),\qquad
 u(r)=\int_r^\infty\frac{M(s)}{s^4}\,ds,\qquad
 \rho(r)=\frac{M'(r)}{r^4},\qquad
 a(r)=\frac{\rho(r)}{u(r)^2}.
 \label{r6:radial-mass-construction}
\end{equation}
Here $M'(r)=4(h(r)+rh'(r))>0$ for $r>0$.  Near zero,
$M(r)=4r^5$, $u(r)=u(0)-2r^2$, and $\rho(r)=20$.
Thus $u,\rho,a$ extend smoothly and positively across zero.
Radial differentiation in dimension five gives
$-\Delta u=M'(r)/r^4=\rho=au^2$.
For $r\geq1$, the formulas give $M(r)=4r$, $u(r)=2r^{-2}$,
$\rho(r)=4r^{-4}$, and $a(r)=1$.

The radial Newton formula, followed by integration by parts, gives
\[
 I_2\rho(r)
 =\frac13\left(r^{-3}M(r)+\int_r^\infty s\rho(s)\,ds\right)
 =\int_r^\infty\frac{M(s)}{s^4}\,ds=u(r).
\]
Thus the potential identity in \eqref{r6:compact-defect-pair} is
exact.  Smoothness on the unit ball and the exterior bound
$au^3=8r^{-6}$ prove finite total energy.

Since $a$ is smooth, positive, and constant outside the unit ball,
$\sup_{x\ne0}|x|\,|\nabla\log a(x)|<\infty$.
The fixed annulus is uniformly connected by paths of length at most a
constant times the endpoint distance.  Integrating along such paths
shows that $\log a(Rx)-\log a(Ry)$ is bounded by $C|x-y|$ on that
annulus, uniformly for $R>0$.  This is also the estimate for the
normalized logarithms, since the normalizing value cancels.
Finally, $x\cdot\nabla\log a=0$ for $|x|\geq1$, while
$2\cdot5/3-(5-2)=1/3>0$.
\end{proof}

\begin{remark}
\label{r6:compact-defect-scaling}
This obstruction cannot be removed by requiring the Euler negative
part to have a small support or a small unweighted finite $L^r$ norm.
Indeed, set
\[
 a_\varepsilon(x)=a(x/\varepsilon),\qquad
 u_\varepsilon(x)=\varepsilon^{-2}u(x/\varepsilon).
\]
These give the same positive system, with $a_\varepsilon=1$ outside
$B_\varepsilon$.  If
$d(x)=(-x\cdot\nabla\log a(x))_+$, then for every $1\leq r<\infty$,
\[
 \|d(\,\cdot\,/\varepsilon)\|_{L^r(\mathbb R^5)}
 =\varepsilon^{5/r}\|d\|_{L^r(\mathbb R^5)}\longrightarrow0.
\]
Under dilation the defect remains weighted by the same energy
amplification factor as the interior nonlinear energy; compact
support alone supplies no absorbable error.
\end{remark}

\section{Green kernels under annular potential budgets}
\label{r7:green-section}\label{r8:longrange-section}

Throughout this section, $n\ge3$, $\lambda=n-2$ and
$N(x)=c_{n,2}|x|^{-\lambda}$. Let $V\ge0$ belong to
$C^1(\mathbb R^n)$. We use the minimal positive Green
kernel $G_V$ of $-\Delta+V$ and the budgets
\eqref{r7:newton-budget}--\eqref{r8:annular-potential-class}
defined in the introduction. All equations
use that actual Green kernel. We establish the estimates needed for
the proof of the theorem, retaining local nonlinear energy as the
only initial energy assumption.
Green comparison and relative composition bounds have an extensive
perturbation theory; see \cite{Pinchover1992,GangulyPinchover2020,Hirata2006}.
The additional issue here is the compatibility of endpoint differences
with the localized source-layer pressure. We prove those differences
directly, without a global nonlinear energy bound.

The elementary Newton inequality used below is
\[
 \frac{N(x-z)N(z-y)}{N(x-y)}
 \le C_\lambda\{N(x-z)+N(z-y)\}.
\]
It follows from the triangle inequality and
$(a+b)^\lambda\le C_\lambda(a^\lambda+b^\lambda)$.
In particular $N|h|N\le\mathcal B(h)N$.

\begin{lemma}
\label{r8:green-exponential-comparison}
For any nonnegative $V\in C^1(\mathbb R^n)$ with
$\mathcal B(V)<\infty$,
\begin{equation}
 e^{-\mathcal B(V)}N(x-y)\le G_V(x,y)\le N(x-y),\qquad x\ne y.
 \label{r8:green-exponential-bound}
\end{equation}
Consequently, if $\mathcal B(|Q|)<\infty$,
\begin{equation}
 \mathcal N_V\le e^{\mathcal B(V)}\mathcal B(Q_+).
 \label{r7:explicit-green-budget}
\end{equation}
\end{lemma}
\begin{proof}
Fix the pole $y$ and put $\Phi=NVN(\cdot,y)$,
$h=\Phi/N(\cdot-y)$. The Newton three-kernel inequality gives
$0\le h\le\mathcal B(V)$. Also $h(x)\to0$ as $x\to y$.
Indeed the integral restricted to $B_\rho(y)$, divided by
$N(x-y)$, is at most $C\|V\|_{\infty,B_{2\rho}(y)}\rho^2$
for $|x-y|<\rho/2$, by the same three-kernel inequality.
On the complementary set its numerator is uniformly bounded as
$x\to y$; its distant tail is finite since
$\int_{|z-y|>1}V(z)N(z-y)^2dz<\infty$, a consequence of
$\mathcal B(V)<\infty$. First send $x\to y$ and then $\rho\to0$.

The increasing-ball Green kernel has $0<G_V\le N$ and
$N-G_V=NVG_V\le\Phi$; this resolvent identity follows from
Newton representation and does not require a small norm.
In particular $G_V(x,y)/N(x-y)\to1$ at the pole.
Off the pole, define
$\mathcal L^N=-\Delta-2\nabla\log N(\cdot-y)\cdot\nabla$.
Since $-\Delta\Phi=VN$, one has $\mathcal L^Nh=V$.
Therefore $w=N(\cdot-y)e^{-h}$ satisfies
\[
 (-\Delta+V)w=-w|\nabla h|^2\le0
       \quad\hbox{off }y.
\]
For $0<\epsilon<1$, the function $(1-\epsilon)w-G_V(\cdot,y)$
is negative on a sufficiently small sphere about $y$ and has
nonpositive upper limit at infinity, since both terms are bounded
by $N(\cdot-y)\to0$. The maximum principle for $-\Delta+V$ on
punctured expanding balls gives $(1-\epsilon)w\le G_V$.
Letting $\epsilon\downarrow0$ gives $G_V\ge w\ge e^{-\mathcal B(V)}N$.
Finally $G_VQ_+G_V\le NQ_+N\le\mathcal B(Q_+)N$ and the
lower comparison give the asserted relative budget.
\end{proof}

Put $c_V=e^{-\mathcal B(V)}>0$ and, for $R\ge1$, write
\[
\begin{aligned}
 V_R(x)&=R^2V(Rx),&G_R(x,y)&=R^\lambda G_V(Rx,Ry),\\
 H_R(x,y)&=N(x-y)-G_R(x,y),&Q_R(x)&=R^2Q(Rx).
\end{aligned}
\]
All three potential quantities in \eqref{r8:annular-potential-class}
and the relative budget are invariant under this scaling. The exact
correction identity is
\begin{equation}
 (x\cdot\nabla_x+y\cdot\nabla_y)H_R
       =-\lambda H_R+G_RQ_RG_R.
 \label{r7:correction-dilation}
\end{equation}

\begin{lemma}
\label{r8:longrange-kernel}\label{r7:green-correction}
Under \eqref{r8:annular-potential-class}, the comparison and relative
budget are those of Lemma~\ref{r8:green-exponential-comparison}.
The resolvent identities $H_R=NV_RG_R=G_RV_RN$ and the dilation
identity \eqref{r7:correction-dilation} remain valid. Moreover
\begin{equation}
 |x|\,|\nabla_xH_R(x,y)|+|y|\,|\nabla_yH_R(x,y)|
        \le C_VN(x-y),\qquad x\ne y.
 \label{r8:longrange-gradient}
\end{equation}
Define the active-annulus part of the correction kernel by
\[
 H_R^{\rm a}(x,y)=\int_{1/4<|z|<4}N(x-z)V_R(z)G_R(z,y)\,dz.
\]
For $x,a\in B_3$ and $y\ne x,a$,
\begin{equation}
\begin{aligned}
 |H_R^{\rm a}(x,y)-H_R^{\rm a}(a,y)|
 &\le C\nu |x-a|\{N(x-y)+N(a-y)\},\\
 \nu&=\sup_{R\ge1}\|V_R\|_{L^\infty(\{1/4<|z|<4\})}<\infty.
\end{aligned}
 \label{r8:relative-lipschitz}
\end{equation}
\end{lemma}

\begin{proof}
The Newton three-kernel inequality proves all value comparisons
and composition estimates exactly as before. Increasing-ball
construction and the maximum principle give $0<G_V\le N$.
The Newton potential $NVG_V$ is finite off the diagonal and bounded
by $\mathcal B(V)N$. Its distributional negative Laplacian is $VG_V$;
the remaining entire harmonic function tends to zero at infinity
and hence vanishes. Thus $H=NVG_V=G_VVN$.

We record an elementary scaled relative convolution estimate:
\begin{equation}
 \int_{B_{r/2}(x)}|x-z|^{1-n}N(z-y)\,dz
                 \le CrN(x-y),\qquad r>0.
 \label{r8:one-annulus-gradient}
\end{equation}
By translation and scaling it is enough to take $x=0,r=1$.
For distant $y$ this is immediate. For bounded $y$, division by
$N(-y)$ and the triangle inequality reduce to the integrals of
$|z|^{1-n}$ and $|z|^{-1}|z-y|^{-\lambda}$ over a bounded ball.
Splitting according to which of the two distances is smaller bounds
the second integrand by a sum of powers $-(n-1)$, whose integrals
are uniformly finite.

To estimate $|x||\nabla_xH|$, set $r=|x|>0$ and differentiate
only the Newton factor in $H=NVG_V$. On $B_{r/2}(x)$,
$V(z)\le4A_V/r^2$, and \eqref{r8:one-annulus-gradient} gives
an upper bound $CA_VN(x-y)$. On its complement,
$r/|x-z|\le2$, so the integrand is bounded by a constant times
$N(x-z)V(z)N(z-y)$. The Newton three-kernel inequality bounds
this integral by $C\mathcal B(V)N(x-y)$. The case $x=0$ is
trivial, symmetry treats the other endpoint, and scaling gives
\eqref{r8:longrange-gradient}. In particular no dyadic summation
is hidden in that estimate.

For completeness, the resolvent identity between nearby scales is
differentiable pointwise off the diagonal. First the resolvent
identity and the Newton three-kernel bound give
$|G_{Re^t}-G_R|\le C|t|\mathcal B(|Q|)N$, so the kernels
converge locally off the diagonal. Its difference quotient
contains the average of $Q_s(z)=s^2Q(sz)$ over nearby logarithmic
scales. Locally away from the two endpoints use continuity; near
an endpoint use local boundedness of $Q$ and the locally integrable
Newton singularity. At large $z$, domination by
$N(x-z)N(z-y)|Q_s(z)|$ is uniformly integrable: after a change of
variables its tail is bounded by a fixed constant times
$\int_{|Z|>L}|Z|^{-2\lambda}|Q(Z)|dZ$, which tends to zero,
since $\mathcal B(|Q|)<\infty$ implies
$\int_{|Z|>1}|Z|^{-\lambda}|Q(Z)|dZ<\infty$.
The resolvent identity therefore gives
$dG_R/d\log R=-G_RQ_RG_R$, and hence the exact dilation identity.

For the finite-difference estimate use, for all $z\ne x,a$,
\[
 |N(x-z)-N(a-z)|\le C|x-a|
        \{|x-z|^{1-n}+|a-z|^{1-n}\}.
\]
This follows by the mean value theorem when the endpoint distance
is smaller than half the distance to $z$, and by the triangle
inequality otherwise. The elementary estimate
\begin{equation}
 \int_{B_4}|x-z|^{1-n}N(z-y)\,dz\le CN(x-y),\qquad x\in B_3,
 \label{r8:relative-newton-convolution}
\end{equation}
is uniform in $y\ne x$. For $|y|>8$ it follows by comparison
with $|y|^{-\lambda}$. For bounded $y$, divide by $N(x-y)$
and use the triangle inequality to reduce to integrals of
$|x-z|^{1-n}$ and $|x-z|^{-1}|z-y|^{-\lambda}$.
The latter is uniformly integrable by splitting according to which
distance is smaller, since the sum of its exponents is $n-1<n$.
Apply \eqref{r8:relative-newton-convolution} at $x$ and $a$,
use $G_R\le N$ and $V_R\le\nu$ on the active annulus, and
obtain \eqref{r8:relative-lipschitz}. This proof uses finite
differences rather than a line integral of the gradient through a
possible diagonal singularity.
\end{proof}

\subsection{A finite-total-energy consequence}

\begin{proposition}
\label{r7:green-finite-energy}
Let $p,q>0$, and let $a,b,u,v>0$ be continuous off zero, with the
weak Euler bounds \eqref{v5:euler-lower} for real $A_-,B_-$.
Suppose that
\[
 u=G_V(av^p),\qquad v=G_V(bu^q),\qquad
 J:=\int bu^{q+1}=\int av^{p+1}<\infty.
\]
The equality is automatic by symmetry and Tonelli if either integral
is finite. Put
\[
 \gamma=\frac{n+A_-}{p+1}+\frac{n+B_-}{q+1}-(n-2).
\]
If $\gamma>\mathcal N_V$, no such pair exists. No coefficient
Dini hypothesis is needed in this proposition.
\end{proposition}

\begin{proof}
The comparison \eqref{r8:green-exponential-bound} makes both Newton virtual
potentials $u_0=N*(av^p)$, $v_0=N*(bu^q)$ finite, with
$u\le u_0\le{c_V}^{-1}u$, and similarly for $v$.
Since the sources are continuous off zero and the Newton gradient is
locally integrable, the actual pair is $C^1$ locally off zero; use
$u=u_0-N*(Vu)$ and the analogous identity for $v$.
Radial integration by parts with the Euler derivative measures is
therefore justified by the product rule on compact annuli.
Use $X_R=x\eta(x/R)$, first with an additional cutoff at zero.
The proof of Lemma~\ref{r6:transport-virial} removes that cutoff
using local energy; \eqref{r8:longrange-gradient} controls
the additional Green-correction terms by Newton marginals.
The nonnegative Euler-defect terms are retained by Fatou's lemma.
Total energy and Newton comparison allow $R\to\infty$ in the
remaining marginal and kernel terms by dominated convergence. The dilation identity
\eqref{r7:correction-dilation} gives
\[
 \gamma J+\frac1{q+1}\int u^{q+1}\,d\mu_b
          +\frac1{p+1}\int v^{p+1}\,d\mu_a
 \le\int Q(x)u(x)v(x)\,dx,
\]
where $\mu_a,\mu_b\ge0$ are the Euler defect measures.
The right side is absolutely integrable: the absolute-value relative
Green composition budget is finite by $\mathcal B(|Q|)<\infty$.
Positivity and Tonelli give
\[
 \int Q_+uv
 =\iint f(x)g(y)(G_VQ_+G_V)(x,y)\,dx\,dy
 \le\mathcal N_VJ,
\]
with $f=bu^q$, $g=av^p$. Since $J>0$, the stated strict inequality
is impossible.
\end{proof}

\subsection{The localized Green argument}

For a scaled kernel $G_R$, write $u=G_Rg$, $v=G_Rf$, and set
\[
\begin{aligned}
 u_0&=N*g=u+P,&v_0&=N*f=v+S,\\
 P&=N*(V_Ru),&S&=N*(V_Rv).
\end{aligned}
\]
Write $d\pi_0=f(x)g(y)N(x-y)\,dx\,dy$.

The following lemma gives the additional estimates needed to retain
local energy as the only nonlinear energy hypothesis.

Use $E_f=\int\eta fu$, $E_g=\int\eta gv$ and $E=E_f+E_g$.
The following two estimates replace active-annulus harmonicity.

\begin{lemma}
\label{r8:annular-correction}
Assume the source-response and coefficient hypotheses of
Proposition~\ref{r6:general-pressure}, with $f=bF(u)$ and $g=aG(v)$.
Suppose the source input \eqref{r7:green-source-input} holds with
constant $M$, and $u=G_Rg$, $u_0=N*g$, $P=u_0-u=N*(V_Ru)$.
For the retained intervals and weights of
Proposition~\ref{r6:general-pressure}, let $P_I$ be the arithmetic
average of its two endpoint values. Then
\begin{equation}
 \mathbb E_e\int\int_0^{f(x)}\chi_f(e,x,t)|P(x)-P_I|\,dt\,dx
       \le C_V\nu\delta E_f+ C_VM^2,
 \qquad \ell=\delta j.
 \label{r8:annular-endpoint-budget}
\end{equation}
The same assertion holds after interchanging the two sources.
Here $C_V$ is uniform in $R$ and the pair, but may depend on the
fixed potential budgets, $c_V^{-1}$, the coefficient bounds and the
response data. In particular the sum is at most $\xi E+C_\xi$ after
choosing $\delta$ sufficiently small with these data fixed.
\end{lemma}

\begin{proof}
Decompose $V_R$ into its restrictions to $B_{1/4}$,
$\{1/4<|z|<4\}$ and $\{|z|>4\}$, and decompose $P$ accordingly.
Only the truncated core charge is needed:
\begin{equation}
 \int_{B_{1/4}}V_Ru\le C_VM.
 \label{r8:truncated-core-charge}
\end{equation}
Indeed its Newton potential at a source point $y$ is uniformly
bounded by $\|I_2V\|_\infty$ for bounded $y$, whereas for
$|y|>4$ it is at most
$C|y|^{-\lambda}\int_{B_{1/4}}V_R$.
The last mass is uniformly bounded, since $N$ is bounded below
on $B_{1/4}$ and $\|I_2V_R\|_\infty=\|I_2V\|_\infty$.
Tonelli's theorem and the local mass and distant-source input prove
\eqref{r8:truncated-core-charge}. Consequently the core part of
$P$ has value and gradient bounded by $C_VM$ on
$\{1/2<|x|<3\}$.

For the distant part only its gradient is needed. Define
$k_R(z)=\mathbf1_{\{|z|>4\}}|z|^{-\lambda-1}V_R(z)$.
Since $V_R(z)\le A_V|z|^{-2}$,
\[
 0\le k_R(z)\le A_V|z|^{-n-1}\mathbf1_{\{|z|>4\}}.
\]
A direct split into $|z|<|y|/2$, $|z-y|<|y|/2$ and their
complement gives
\[
 I_2k_R(y)\le CA_V\langle y\rangle^{-\lambda}.
\]
In the near-$y$ region, use $k_R\le CA_V|y|^{-n-1}$ and
$\int_{B_{|y|/2}}N\le C|y|^2$; the other regions use $\|k_R\|_1$.
For bounded $y$, boundedness of $k_R$ and its $L^1$ norm suffice.
Since $x\in B_3$ is separated from $|z|>4$, Tonelli's theorem
therefore bounds the gradient of the distant part of $P$ by
$C_V\int\langle y\rangle^{-\lambda}g(y)dy\le C_VM$.
One can also read this directly as an absolute finite-difference
bound by applying the Newton difference estimate before Tonelli;
no source-data bound for the value of the distant part is asserted.

A retained interval has endpoints $a_\pm$ with
$|a_\pm-x|\le2\ell(x)$ and lies in
$\{1/2<|z|<3\}$ whenever its weight is nonzero. The core and
distant contributions to its difference are thus at most
$C_VM\ell(x)$. For the active contribution,
\eqref{r8:relative-lipschitz} gives
\[
 |P^{\rm a}(x)-P^{\rm a}_I|
   \le C\nu\ell(x)\{u_0(x)+u_{0,I}\}.
\]
The lower Green comparison gives
$u_0\le c_V^{-1}u$ at every point, including
both endpoints. The true inverse-source estimate on a retained
component gives $u_I\le Cu(x)$; hence
$u_{0,I}\le Cu(x)$. It follows that
\[
 |P(x)-P_I|\le C\ell(x)\{C_VM+\nu u(x)\}.
\]
This bound holds pointwise before integration in the level parameter.
There is no missing endpoint measure: $\chi_f\le h_e(x)$,
$\mathbb E_eh_e=w$, and $\int_0^{f(x)}dt=f(x)$.
Finally the cutoff budgets imply
\[
 w\ell\le C\delta\eta jH\le C\delta\eta.
\]
These facts give \eqref{r8:annular-endpoint-budget}. The source
term uses only $\int_{B_2}f\le M$; the other term is the true
cutoff energy $E_f$. No bound on $\sup u$, $\sup g$ or total
nonlinear energy has been used.
\end{proof}

\begin{lemma}
\label{r8:longrange-commutator}
Under the same source input, the virial commutator satisfies, for
every $\xi>0$,
\begin{equation}
 \iint f(x)g(y)|\eta(y)-\eta(x)|\,
            |y\cdot\nabla_yH_R(x,y)|\,dx\,dy
              \le\xi E+C_\xi.
 \label{r8:longrange-commutator-budget}
\end{equation}
\end{lemma}

\begin{proof}
By \eqref{r8:longrange-gradient}, the integrand is bounded by
$C_VfgN(x-y)|\eta(y)-\eta(x)|$.
Call the pair close when its separation is at most $\ell(z)$
for either endpoint $z\in B_2$. Close geometry gives
$|\eta(y)-\eta(x)|\le C\delta\eta(z)$.
The virtual Newton marginals integrate this to $C_V\delta E_0$,
and $E_0\le c_V^{-1}E$.
On the complementary pairs,
$|\eta(y)-\eta(x)|\le\eta(x)+\eta(y)$, and each nonzero
endpoint weight is integrated only over separations greater than
$\ell$ at that endpoint. The truncated-potential source input and
$\sup\eta\ell^{-\lambda}<\infty$ bound this part by
$C_{V,\delta}M^2$. Choose $\delta$ small. Notice that an additive
$CM^2$ bound independent of the close-pair choice is unnecessary;
the absorbable energy error is sufficient.
\end{proof}

\begin{lemma}
\label{r7:green-localized-interface}
Fix the cutoff triple and coefficient modulus of
Proposition~\ref{r6:general-pressure}, with $\lambda=n-2$.
Let $f=bF(u)$, $g=aG(v)$, where the responses have positive lower
elasticity bounds as in \eqref{r6:pressure-elasticity}, and suppose
that the local source and energy integrability assumptions there hold.
Assume, uniformly over $R\ge1$,
\begin{equation}
 \begin{split}
 \int_{B_4}(f+g+u_0+v_0)&\le M,\\
 \int_{|x-y|\ge s}N(x-y)(f(y)+g(y))\,dy
     &\le Ms^{-\lambda}\quad(x\in B_2,\ 0<s\le1).
 \end{split}
 \label{r7:green-source-input}
\end{equation}
Use the true energies $E_f=\int\eta fu$, $E_g=\int\eta gv$,
$E=E_f+E_g$, and put
\[
 T_0=\iint w(x)D(x,y)\,d\pi_0,\qquad
 W_P=\int w\{bP_F(u)+aP_G(v)\}.
\]
Then for every $\xi>0$,
\begin{equation}
 \begin{split}
 |E_f-E_g|&\le\xi E+C_\xi,\\
 W_P&\le\lambda T_0+\xi E+C_\xi.
 \end{split}
 \label{r7:green-localized-pressure}
\end{equation}
If the coefficients also have weak Euler lower bounds $A,B$, then
\begin{equation}
 \begin{split}
 &\int[(n+B)\eta-w]bP_F(u)
     +\int[(n+A)\eta-w]aP_G(v)\\
 &\hspace{12mm}\le(\lambda+\mathcal N_V)E_f
                         -\lambda T_0+\xi E+C_\xi.
 \end{split}
 \label{r7:green-localized-virial}
\end{equation}
The constants are independent of $R$ and the particular pair with
these source inputs. They may depend on $V$, the Newton comparison
constant and the displayed fixed bounds.
\end{lemma}

\begin{proof}
The virtual energy
$E_0=\int\eta(fu_0+gv_0)$ is at most
${c_V}^{-1}E$. All geometric estimates for $\pi_0$ therefore
have arbitrarily small true-energy errors after their free parameters
are chosen. The true balance estimate follows directly from the
positive measure $d\pi_G=fgG_R\,dx\,dy$, whose marginals are
$fu$ and $gv$. For close pairs use
$|\eta(x)-\eta(y)|\le C\delta\eta$; for far pairs use
$G_R\le N$, the second input in \eqref{r7:green-source-input},
and the bounded weights $\eta\ell^{-\lambda}$.
This gives $|E_f-E_g|\le C\delta E+C_\delta$.

\emph{The correction at pressure endpoints.}
The conclusion of Proposition~\ref{r6:general-pressure} cannot be
applied directly: the actual response relation is $f=bF(u)$,
whereas its Newton potential is $u_0=u+P$.  We instead apply only
the Newton interval inequality and its finite layer reassembly, and
write the additional recovery terms explicitly.

For a retained component $I=I(e,x,t)$, put
$\tau_f=F^{-1}(t/b(x))$ and let $u_I,P_I,u_{0,I}$ be the
arithmetic means of the respective two endpoint values.  Then
$u_{0,I}=u_I+P_I$. Endpoint products with zero retention weight
are assigned the value zero. The following is a pointwise identity:
\begin{equation}
 \begin{aligned}
 h_e(x)(u(x)-\tau_f)
 ={}&\chi_f(e,x,t)[u_0(x)-u_{0,I}]\\
 &+[h_e(x)-\chi_f(e,x,t)](u(x)-\tau_f)\\
 &+\chi_f(e,x,t)(u_I-\tau_f)
   +\chi_f(e,x,t)[P_I-P(x)].
 \end{aligned}
 \label{paper1:green-layer-recovery}
\end{equation}
There is an identical formula for $g=aG(v)$, with $v_0=v+S$.
The endpoint source values are those of the actual source $f$,
so \eqref{r6:inverse-endpoint} still gives
$|u_I-\tau_f|\le C\Omega(2\ell(x))u(x)$ and $u_I\le Cu(x)$.
The virtual endpoint integral is finite before subtraction, since
\[
 \int g(y)\frac{N(x_--y)+N(x_+-y)}2\,dy
 =u_{0,I}\le c_V^{-1}u_I\le Cu(x).
\]
The analogous statement holds for the other source.  Thus the
signed Newton layer aggregate
\[
 S_e^0=\int\!\int_0^{f(x)}\chi_f(u_0-u_{0,I})\,dt\,dx
       +\int\!\int_0^{g(y)}\chi_g(v_0-v_{0,J})\,dt'\,dy
\]
is a difference of finite quantities.  The Newton interval
inequality and reassembly, without an inverse-response assertion
for $u_0,v_0$, give
\[
 S_e^0\le\lambda\iint h_e(x)D_e(x,y)\,d\pi_0
       +\iint|h_e(x)-h_e(y)|\,d\pi_0+L_f^0(e)+L_g^0(e),
\]
where $L_f^0,L_g^0$ are the discarded-layer losses from
\eqref{eq:pr-losses}, with $u_0,v_0$ in place of $u,v$.
They are finite, and the actual losses are at most these virtual
losses because $u\le u_0$ and $v\le v_0$.

Integrate \eqref{paper1:green-layer-recovery} in $t,x$, use the
primitive layer identity \eqref{r6:primitive-layer}, and add the
other source.  If
\[
 A_e=\int\!\int_0^{f(x)}\chi_f|P(x)-P_I|\,dt\,dx
     +\int\!\int_0^{g(y)}\chi_g|S(y)-S_J|\,dt'\,dy,
\]
the resulting explicit bound is
\begin{equation}
 \begin{aligned}
 \int h_e[bP_F(u)+aP_G(v)]
 \le{}&\lambda\iint h_e(x)D_e\,d\pi_0
       +\iint|h_e(x)-h_e(y)|\,d\pi_0\\
 &+2L_f^0(e)+2L_g^0(e)
  +C\int\Omega(2\ell)h_e(fu+gv)+A_e.
 \end{aligned}
 \label{paper1:green-layer-pressure-bound}
\end{equation}
Lemma~\ref{r8:annular-correction} gives
$\mathbb E_eA_e\le\xi E+C_\xi$ after choosing the interval
scale. For the virtual losses, Lemma~\ref{r7:geometric-budgets}
uses only the superlevel geometry of the actual sources $f,g$, the
local $L^1$ bounds for $u_0,v_0$, and the virtual marginals
$fu_0,gv_0$. It requires no response relation for $u_0,v_0$.
The remaining geometric, cap, coefficient and discarded-layer
estimates are the estimates in the proof of
Proposition~\ref{r6:general-pressure}, applied to the virtual
Newton marginals.  These marginals are bounded by $c_V^{-1}$
times the actual marginals, so their free parameters can be chosen
after this comparison to make their total energy coefficient as
small as required.  Directional averaging of
\eqref{paper1:green-layer-pressure-bound} therefore proves
$W_P\le\lambda T_0+\xi E+C_\xi$.
In particular, the new term is the controlled endpoint difference
$P_I-P(x)$, rather than a multiple of an uncontrolled boundary
pressure.

\emph{Exact cancellation in the virial correction.}
The Newton representations above imply local $C^1$ regularity off zero.
Integrate by parts against $X=x\eta$, initially cut off at zero,
and use the radial Euler derivative measures as in
Lemma~\ref{r6:transport-virial}. Symmetry of $H_R$ gives the exact
interaction expression
\begin{equation}
 -\int fX\cdot\nabla u-\int gX\cdot\nabla v
 =\lambda\iint\Phi\,d\pi_0
  +\iint fg(X(x)\cdot\nabla_x+X(y)\cdot\nabla_y)H_R.
 \label{r7:green-virial-split}
\end{equation}
The correction obeys
\[
 (X(x)\cdot\nabla_x+X(y)\cdot\nabla_y)H_R
 =\eta(x)(-\lambda H_R+G_RQ_RG_R)
       +(\eta(y)-\eta(x))y\cdot\nabla_yH_R.
\]
To display the cancellation, set
\[
 E_f^0=\int\eta fu_0,
 \qquad
 Z_R=\iint\eta(x)f(x)g(y)H_R(x,y)\,dx\,dy
     =E_f^0-E_f.
\]
The geometric localization identity is
\begin{equation}
 \lambda\iint\Phi\,d\pi_0
 =\lambda E_f^0-\lambda T_0+\mathcal R_R,
 \qquad |\mathcal R_R|\le\xi E+C_\xi,
 \label{paper1:green-virtual-energy-localization}
\end{equation}
where
$\mathcal R_R=\lambda\iint(\Phi-\eta(x)+w(x)D)\,d\pi_0$.
Its estimate follows from the geometric input and
$E_0\le c_V^{-1}E$, after choosing the geometric error parameter.
Consequently \eqref{r7:green-virial-split} is exactly
\begin{equation}
 \begin{aligned}
 -\int fX\cdot\nabla u-\int gX\cdot\nabla v
 &=\lambda E_f^0-\lambda T_0-\lambda Z_R
       +\mathcal D_R+\mathcal C_R+\mathcal R_R\\
 &=\lambda E_f-\lambda T_0
       +\mathcal D_R+\mathcal C_R+\mathcal R_R,
 \end{aligned}
 \label{paper1:green-true-energy-cancellation}
\end{equation}
with
\[
 \begin{aligned}
 \mathcal D_R
 &=\iint\eta(x)f(x)g(y)(G_RQ_RG_R)(x,y)\,dx\,dy,\\
 \mathcal C_R
 &=\iint f(x)g(y)[\eta(y)-\eta(x)]
                       y\cdot\nabla_yH_R(x,y)\,dx\,dy.
 \end{aligned}
\]
The two one-body terms on the left denote their common
origin-cutoff limit, justified below; no separate global score
integral is asserted. The composition with $|Q_R|$ has a finite
relative bound, by \eqref{r7:explicit-green-budget} with $Q_+$
replaced by $|Q|$, so $\mathcal D_R$ is absolutely integrable.
Its upper bound uses only the positive part:
\[
 \mathcal D_R
 \le\iint\eta(x)fg(G_R(Q_R)_+G_R)(x,y)\,dx\,dy
 \le\mathcal N_VE_f.
\]
Lemma~\ref{r8:longrange-commutator} gives
$|\mathcal C_R|\le\xi E+C_\xi$, uniformly in scale.
The cancellation of $\lambda Z_R$ in
\eqref{paper1:green-true-energy-cancellation} is algebraic; it does
not estimate virtual energy by an extra multiple of actual energy.
Choosing the free error parameters smaller in advance and using
the weak Euler lower bounds proves
\eqref{r7:green-localized-virial}.

For completeness, removing the origin cutoff in
\eqref{r7:green-virial-split} does not require global energy.
The Newton part uses its localized finite marginals.
By \eqref{r8:longrange-gradient}, the correction on the
difference of the two radial fields is bounded by $C_V$ times the
virtual marginal energy on $B_{2\rho}$, which tends to zero and is
at most ${c_V}^{-1}$ times the true local energy there.
The primitive integration-by-parts origin errors have the same
local-energy bound. This justifies the asserted order of limits.
\end{proof}

\begin{proof}[Proof of Theorem~\ref{r7:green-local-energy}]
Put $\lambda=n-2$, $f=bF(u)$, $g=aG(v)$,
$q_*=\mu_F^{-1}-1$ and $p_*=\mu_G^{-1}-1$.
We first justify the source and amplitude inputs independently of
energy.  The lower comparison
$G_V(x,y)\ge{c_V}N(x-y)$ makes the Newton source potentials
finite.  In particular $f,g$ are locally integrable,
\[
 \int_{|y|>1}|y|^{-\lambda}(f+g)<\infty,
 \qquad R^{-\lambda}M_f(R),\ R^{-\lambda}M_g(R)\longrightarrow0.
\]
On the dilated coefficient patch $RD$, the same comparison gives
$u\ge cR^{-\lambda}M_g(R)$ and
$v\ge cR^{-\lambda}M_f(R)$.
The arguments on the right tend to zero.  Apply monotonicity,
then the small-argument response bounds
\eqref{r6:response-bounds}, and the weak Euler inequalities.
Exactly the source feedback \eqref{r7:unscaled-source-feedback}
results, with $\sigma=2$ and fixed additional kernel constants.
Consequently
\[
 \alpha_*:=\frac{2+A+p_*(2+B)}{p_*q_*-1}\le\lambda,
 \qquad
 \beta_*:=\frac{2+B+q_*(2+A)}{p_*q_*-1}\le\lambda.
\]
The effective-power gap identity is
\[
 \alpha_*+\beta_*-\lambda
 =\frac{(p_*+1)(q_*+1)}{p_*q_*-1}\Gamma_G>0.
\]
It follows that $\alpha_*,\beta_*>0$ and
$n+A=\lambda-\alpha_*+p_*\beta_*>0$,
$n+B=\lambda-\beta_*+q_*\alpha_*>0$.
The rectangle positivity \eqref{r7:rectangle-decomposition} follows
from the same algebra.

Choose the implicit amplitudes
$s_R=R^2a(Re_0)G(t_R)$,
$t_R=R^2b(Re_0)F(s_R)$ and normalize the responses and
coefficients as in \eqref{r6:complete-family}, but use
$G_R(x,y)=R^\lambda G_V(Rx,Ry)$ as the kernel.
The exact source and energy transformations remain
\eqref{r6:exact-source-transform} and
\eqref{r6:exact-energy-density}.
The unique scalar inverse and finite-increment argument in
Lemma~\ref{r7:normalization} use only the response bounds,
Euler inequalities, and rectangle positivity.  They therefore give
\[
 s_{R\rho}/s_R\le C\rho^{-\alpha_*},\qquad
 t_{R\rho}/t_R\le C\rho^{-\beta_*}.
\]
The normalized patch mass argument uses only the lower Newton
comparison for $G_R$ and gives uniform unit-ball source masses.
Larger balls and dyadic summation then give
\eqref{r7:green-source-input}, including the virtual Newton
potentials $u_0=N*g_R$, $v_0=N*f_R$.  These bounds precede every
uniform energy estimate.  The Green comparison constants,
$\mathcal N_V$, and response constants are unchanged under scaling.

Apply Lemma~\ref{r7:green-localized-interface} at sufficiently
large scales, using the common Dini-adapted cutoff.  Write
$\kappa_f=(n+B)\mu_F$, $\kappa_g=(n+A)\mu_G$.
The signs just established permit primitive coercivity.
Adding localized pressure to localized virial cancels
$\lambda T_0$ and yields
\[
 (\kappa_f-\lambda-\mathcal N_V)E_f+\kappa_gE_g
       \le2\xi E+C_\xi.
\]
The two coefficients have sum
$\Gamma_G-\mathcal N_V>0$.
Their signed balance identity, together with
$|E_f-E_g|\le\xi E+C_\xi$, absorbs the error after choosing
$\xi>0$ sufficiently small.  Thus $E_R\le C_*$ uniformly.

Finally, $\mathcal H_R=R^{-\lambda}/(s_Rt_R)$ tends to infinity,
since its growth exponent is
$\alpha_*+\beta_*-\lambda>0$.  The exact true-energy scaling gives
\[
 C_*\ge E_R\ge\mathcal H_R
           \int_{B_R}[bF(u)u+aG(v)v].
\]
The original energy on a fixed ball is finite and strictly positive,
which contradicts this inequality.  This uses only local nonlinear
energy throughout.
\end{proof}

\begin{corollary}
\label{r7:green-power-local-energy}
Under the same Green and coefficient hypotheses, let $p,q>0$,
$pq>1$.  If
\begin{equation}
 \gamma:=\frac{n+A}{p+1}+\frac{n+B}{q+1}-(n-2)>\mathcal N_V,
 \label{r7:green-gap}
\end{equation}
there is no finite positive continuous-off-zero pair satisfying
$u=G_V(av^p)$, $v=G_V(bu^q)$ and
$\int_{B_1}(bu^{q+1}+av^{p+1})<\infty$.
\end{corollary}
\begin{proof}
Take $F(s)=s^q$, $G(s)=s^p$,
$\mu_F=1/(q+1)$ and $\mu_G=1/(p+1)$ in
Theorem~\ref{r7:green-local-energy}.
\end{proof}

\begin{corollary}
\label{r7:small-green-perturbation}
\label{r7:small-green-primitive-perturbation}
Fix a nonzero $W\ge0$ in $C^1(\mathbb R^n)$ with
$A_W,\mathcal B(W),\mathcal B(|2W+x\cdot\nabla W|)<\infty$,
and responses
and coefficients satisfying the primitive, Euler and Dini hypotheses
of Theorem~\ref{r7:green-local-energy}.  Suppose
$\Gamma_G=(n+B)\mu_F+(n+A)\mu_G-(n-2)>0$.
Put $Q_W=2W+x\cdot\nabla W$.  For every
\begin{equation}
 0<\varepsilon<
 \frac{\Gamma_G}
 {\mathcal B((Q_W)_+)+\Gamma_G\mathcal B(W)},
 \label{r7:explicit-small-potential}
\end{equation}
the actual Green system for $-\Delta+\varepsilon W$ has no positive
pair in the local-energy class of that theorem.
For constant coefficients $a=b=1$ the gap is
$\Gamma_G=n(\mu_F+\mu_G)-(n-2)$.
\end{corollary}
\begin{proof}
More generally it is sufficient that
\begin{equation}
 \varepsilon e^{\varepsilon\mathcal B(W)}
               \mathcal B((Q_W)_+)<\Gamma_G.
 \label{r8:exponential-amplitude-budget}
\end{equation}
Indeed \eqref{r7:explicit-green-budget} then bounds
$\mathcal N_{\varepsilon W}$ by a number smaller than $\Gamma_G$.
The displayed rational range is a simpler sufficient condition:
it ensures $s=\varepsilon\mathcal B(W)<1$, and
$e^s\le(1-s)^{-1}$. Apply Theorem~\ref{r7:green-local-energy}.
No smallness of $\mathcal B(V)$ is required by that theorem itself.
\end{proof}

\subsection{Verifiable noncompact classes}

The concrete condition
$|V(x)|+|x\cdot\nabla V(x)|\le C\langle x\rangle^{-2-\varepsilon}$,
$\varepsilon>0$, implies $A_V<\infty$ and finiteness of
$\mathcal B(V)$ and $\mathcal B(|Q|)$. The same is true for an
inverse-square tail divided by $(\log(2+|x|))^{1+\varepsilon}$,
with the corresponding bound on its radial derivative. These
finiteness assumptions require no amplitude smallness; a small
amplitude supplies the strict relative-defect gap when the nonlinear
gap is positive.
An easily verifiable sufficient condition for the finiteness of
the displayed quantities is finiteness, for both
$h=V,Q$, of the annular envelope
\[
 \mathfrak A(h)=\|h\|_{L^\infty(B_2)}
   +\sum_{j=1}^{\infty}2^{2j}
           \|h\|_{L^\infty(\{2^j<|x|<2^{j+1}\})}.
\]
Indeed $\|I_2|h|\|_\infty\le C\mathfrak A(h)$, by a dyadic
split and the integral of $N$ over a ball. No annular summability
of the supremum is required in \eqref{r8:annular-potential-class} itself.

\begin{example}
\label{r8:explicit-longrange-example}
For $\varepsilon>0$ set
\[
 W_\varepsilon(x)=(1+|x|^2)^{-1-\varepsilon/2},\qquad
 Q_\varepsilon=2W_\varepsilon+x\cdot\nabla W_\varepsilon
       =W_\varepsilon\frac{2-\varepsilon|x|^2}{1+|x|^2}.
\]
Both $W_\varepsilon$ and $(Q_\varepsilon)_+$ are radial and
nonincreasing, so their Newton potentials attain their maxima at zero.
Using $c_{n,2}|\mathbb S^{n-1}|=1/(n-2)=1/\lambda$ gives
\[
 \mathcal B(W_\varepsilon)=\frac{2C_\lambda}{\lambda\varepsilon},
 \qquad
 \mathcal B((Q_\varepsilon)_+)
    =\frac{4C_\lambda}{\lambda\varepsilon}
            (1+2/\varepsilon)^{-1-\varepsilon/2}.
\]
For the second formula integrate only over
$0<r<\sqrt{2/\varepsilon}$ and substitute $t=1+r^2$.
Consequently, if $\Gamma=(n+B)\mu_F+(n+A)\mu_G-\lambda>0$,
then $V=tW_\varepsilon$ satisfies the sufficient Green budget whenever
\[
 t\,e^{t\mathcal B(W_\varepsilon)}\mathcal B((Q_\varepsilon)_+)<\Gamma.
\]
The rational sufficient range is given by
\eqref{r7:explicit-small-potential} with $W=W_\varepsilon$.
The exponential criterion also contains examples with arbitrarily
large Newton budget, as the following choice shows.
As $\varepsilon\downarrow0$,
$\mathcal B((Q_\varepsilon)_+)\to2C_\lambda/\lambda$.
Choose
\[
 t_\varepsilon=\frac{\lambda\varepsilon}{4C_\lambda}
                        \log(1/\varepsilon),\qquad 0<\varepsilon<1.
\]
Then
\[
 \mathcal B(t_\varepsilon W_\varepsilon)
       =\tfrac12\log(1/\varepsilon)\longrightarrow\infty,
\]
whereas
\[
 t_\varepsilon e^{t_\varepsilon\mathcal B(W_\varepsilon)}
                 \mathcal B((Q_\varepsilon)_+)
       =O\bigl(\sqrt\varepsilon\log(1/\varepsilon)\bigr)
           \longrightarrow0.
\]
Thus every fixed positive nonlinear gap admits genuine noncompact
examples with $\mathcal B(V)>1$, and even with arbitrarily large
$\mathcal B(V)$, while their relative defect is smaller than the gap.
Only local nonlinear energy is required. The parameter $\varepsilon$
may be arbitrarily small and positive; no assertion is made by setting
it equal to zero.
\end{example}

% General long-range entire-PDE consequence; no initial energy/tail assumption.
\subsection{Entire power systems with noncompact potentials}
\label{r8:longrange-entire-pde-section}

\begin{lemma}
\label{r8:potential-exterior-bound}
Let $n\ge3$, $p,q>0$, $pq>1$, and
\[
 \gamma=\frac n{p+1}+\frac n{q+1}-(n-2)>0,
 \qquad \alpha=\frac{2(p+1)}{pq-1},\quad
 \beta=\frac{2(q+1)}{pq-1}.
\]
For every $A<\infty$ there is $C=C(n,p,q,A)$ such that every
nonnegative classical pair on $\mathbb R^n\setminus\overline B_1$
solving
\[
 -\Delta u+Vu=v^p,\qquad -\Delta v+Vv=u^q,
 \qquad 0\le V(x)\le A|x|^{-2},
\]
satisfies, for $|x|>2$,
\begin{equation}
 u(x)+|x||\nabla u(x)|\le C|x|^{-\alpha},\qquad
 v(x)+|x||\nabla v(x)|\le C|x|^{-\beta}.
 \label{r8:potential-exterior-decay}
\end{equation}
% No derivative bound on $V$ is required for this estimate.
\end{lemma}

\begin{proof}
We prove the stronger pointwise estimate in terms of
$\operatorname{dist}(x,\partial B_1)=|x|-1$ at every $|x|>1$.
It follows by doubling. If it failed, there would
be potentials and pairs with a common bound $A$ and points $x_j$
such that
\[
 M_j(x_j)\operatorname{dist}(x_j,\partial B_1)>2j^2,
 \qquad M_j=u_j^{1/\alpha}+v_j^{1/\beta}.
\]
The doubling lemma of \cite{PolacikQuittnerSouplet2007},
with parameter $j^2$, gives $y_j$ and
$r_j=M_j(y_j)^{-1}$ such that
\[
 r_j<\frac{\operatorname{dist}(y_j,\partial B_1)}{2j^2}
          \le\frac{|y_j|}{2j^2},\qquad
 M_j(y_j+r_jz)\le2M_j(y_j)\quad(|z|\le j^2).
\]
Set $U_j(z)=r_j^\alpha u_j(y_j+r_jz)$ and
$W_j(z)=r_j^\beta v_j(y_j+r_jz)$. Their root sum is at most
two on $B_{j^2}$ and equals one at zero. The scaled equations are
\[
 -\Delta U_j+\widetilde V_jU_j=W_j^p,\qquad
 -\Delta W_j+\widetilde V_jW_j=U_j^q,
 \quad \widetilde V_j(z)=r_j^2V_j(y_j+r_jz).
\]
On $B_j$, the inverse-square bound gives
$0\le\widetilde V_j\le CAj^{-4}\to0$ uniformly.
The bounded right sides give uniform local $W^{2,s}$ estimates for
every finite $s$, hence compactness in $C^{1,\theta}_{\rm loc}$
for some $\theta>0$. Taking the equations in distributions gives
a bounded nonnegative entire nonzero limit satisfying
$-\Delta U=W^p$, $-\Delta W=U^q$.
Only the limit needs Schauder regularity: its positive powers are
locally H\"older, so it is a classical pair. Thus no oscillation or
H\"older bound on the approximating potentials has been invoked.
The strong maximum principle makes both limit components positive.

For completeness, this bounded pure pair is excluded directly by the
power specialization of Theorem~\ref{r7:master-front}. Increasing-ball
Newton Green minorants give $I_2(W^p)\le U$ and $I_2(U^q)\le W$.
Their nonnegative harmonic remainders are constants. A positive
constant remainder would make the opposite source bounded below by
a positive constant, contradicting finiteness of its Newton potential.
The representations are therefore exact; boundedness supplies local
energy, and the strict gap gives the required contradiction.

The pointwise estimate is
$u\le C(|x|-1)^{-\alpha}$ and
$v\le C(|x|-1)^{-\beta}$ for every $|x|>1$, which implies
the displayed exterior decay for $|x|>2$.
For the gradient, rescale a ball of radius $|x|/4$ centered at such
an $x$. The rescaled pair and rescaled potential are bounded, so
the right sides in the Laplace equations are bounded. Interior
$W^{2,s}$ estimates with $s>n$ give the displayed gradient bounds.
This final step again uses no derivatives of $V$.
\end{proof}

\begin{corollary}
\label{r8:longrange-entire-pde}\label{r7:green-entire-pde}
Assume the noncompact potential hypotheses
\eqref{r8:annular-potential-class}. If $p,q>0$, $pq>1$, and
\[
 \frac n{p+1}+\frac n{q+1}-(n-2)>\mathcal N_V,
\]
then
\[
 -\Delta u+Vu=v^p,\qquad -\Delta v+Vv=u^q
       \quad\hbox{in }\mathbb R^n
\]
has no strictly positive entire classical pair. 
% No boundedness,
% growth, tail, potential-representation or energy hypothesis is
% imposed on the pair.
\end{corollary}

\begin{proof}
Since $\mathcal N_V\ge0$, the strict hypothesis implies $\gamma>0$.
Lemma~\ref{r8:potential-exterior-bound} applies with $A=A_V$.
The algebraic identities
\[
 \alpha+2=p\beta,\quad\beta+2=q\alpha,\qquad
 \alpha(q+1)=\beta(p+1)=\alpha+\beta+2>n
\]
show that both nonlinear energies, $|\nabla u||\nabla v|$,
and $Vuv$ are integrable. Their finiteness is a conclusion of
exterior decay. Multiplying the equations by the opposite component
and integrating over increasing balls gives
\[
 J:=\int u^{q+1}=\int v^{p+1}
     =\int\nabla u\cdot\nabla v+\int Vuv\in(0,\infty).
\]
All boundary terms are $O(R^{n-\alpha-\beta-2})\to0$.

The Green representations are exact as well. Increasing Dirichlet
balls give $G_V(v^p)\le u$ and $G_V(u^q)\le v$.
The nonnegative remainders are entire $L_V$-harmonic and bounded
above by components tending to zero at infinity. The maximum principle
for $V\ge0$ makes both remainders zero.

One can now apply the finite-energy Green virial argument, whose
long-range validity follows from
Lemma~\ref{r8:longrange-kernel}. Alternatively, direct PDE integration
makes the limiting step explicit. Since
$\alpha+\beta>n-2=\lambda$, the decay of $uv$ and
$\mathcal B(|Q|)<\infty$ give
\[
 \int |Q|uv<\infty;
\]
indeed $uv\le C|x|^{-\alpha-\beta}\le C|x|^{-\lambda}$
outside a ball, and the latter is integrable against $|Q|$ by its
Newton potential at zero. Multiply the first PDE by $x\cdot\nabla v$
and the second by $x\cdot\nabla u$, and integrate over balls.
The gradient and potential boundary terms again vanish as
$O(R^{n-\alpha-\beta-2})$. The identity obtained after using the
common energy equality is
\[
 \gamma J=\int(2V+x\cdot\nabla V)uv=\int Q uv.
\]
Finally, positivity, the exact Green representations and Tonelli give
\[
\begin{aligned}
 \int Q_+uv
 &=\iint u(x)^qv(y)^p(G_VQ_+G_V)(x,y)\,dx\,dy\\
 &\le\mathcal N_V\iint u(x)^qv(y)^pG_V(x,y)\,dx\,dy
 =\mathcal N_VJ.
\end{aligned}
\]
This contradicts $\gamma>\mathcal N_V$ and $J>0$.
\end{proof}

\subsection{A borderline inverse-square counterexample}

The Newton-norm and defect restrictions cannot be silently dropped
when allowing critical long-range tails.

\begin{example}
\label{r8:inverse-square-counterexample}
Let $n\ge3$, $1<p<(n+2)/(n-2)$ and $\alpha=2/(p-1)$.
Choose $M>0$ with $M^{p-1}\ge\alpha n$, and put
\[
 U(x)=M(1+|x|^2)^{-\alpha/2},\qquad
 V(x)=\frac{M^{p-1}-\alpha n}{1+|x|^2}
       +\frac{\alpha(\alpha+2)|x|^2}{(1+|x|^2)^2}.
\]
Then $V\ge0$ is smooth with a positive inverse-square tail,
$\mathcal B(V)=\infty$, and the strictly subcritical symmetric
system has the smooth positive exact Green pair $u=v=U$:
\[
 -\Delta U+VU=U^p,\qquad U=G_V(U^p),\qquad
 \int U^{p+1}<\infty.
\]
\end{example}
\begin{proof}
Direct differentiation gives
\[
 \frac{\Delta U}{U}
 =-\frac{\alpha n}{1+r^2}
        +\frac{\alpha(\alpha+2)r^2}{(1+r^2)^2},
 \qquad U^{p-1}=\frac{M^{p-1}}{1+r^2}.
\]
These identities prove the PDE. The positive inverse-square
coefficient makes $I_2V(0)$ diverge logarithmically.
Since $\alpha(p+1)=2\alpha+2>n$, the energy is finite.
Increasing Dirichlet balls give $G_V(U^p)\le U$; the nonnegative
entire $L_V$-harmonic remainder is bounded by $U\to0$ and vanishes
by the maximum principle. This proves the exact representation.
The example makes no assertion that
$2n/(p+1)-(n-2)>\mathcal N_V$; that strict defect hypothesis
remains part of the proved theorem.
\end{proof}

\begin{remark}
\label{r7:green-open-boundary}
The Green argument uses a relative finite-difference estimate,
truncated core charge and a close/far commutator budget. It does
not require harmonicity on the active annulus or a bound on the
full killing charge. The exact resolvent dilation term remains
essential. The results do not cover arbitrary inverse-square
tails, sign-changing potentials, variable principal matrices,
rough domains or two different operator kernels. The equality
$\Gamma_G=\mathcal N_V$ is not settled by the absorption argument.
\end{remark}

\section{PDE representation and consequences}
\label{sec:paper1-pde-representation}

\subsection{Linear tools and source inversion}
\label{ho:section}

Throughout this subsection $0<\sigma<n$, $\lambda=n-\sigma$,
and $L_\sigma=(-\Delta)^{\sigma/2}$. The tail classes are
\eqref{ho:tail}. We divide by the Fourier multiplier only away from
zero frequency; no nonsmooth multiplier is applied to an arbitrary
distribution supported at zero.

For real $t$, write $w\in\mathsf M_t$ if
\[
 \int_{|x|>1}|w(x)||x|^{-t}\,dx<\infty.
\]
Thus $\mathsf M_t\subset\mathsf M_{t'}$ when $t'\ge t$.

\subsubsection{Linear estimates without a maximum principle}

\begin{lemma}
\label{ho:linear}
For every $\sigma>0$, the map
\[
 \mathcal S(\mathbb R^n)\longrightarrow
 X_\sigma:=\{a:\sup_x(1+|x|)^{n+\sigma}|a(x)|<\infty\},
 \qquad \phi\longmapsto L_\sigma\phi,
\]
is continuous.  If $h\in\mathcal L_0$ and $L_\sigma h=0$ on the
whole space in the compact-test sense, then $h=0$ almost everywhere.
\end{lemma}
\begin{proof}
For an even integer order the continuity assertion follows from
Schwartz seminorms of derivatives.  Otherwise write
$\sigma=2m+s$, $m\geq0$, $0<s<2$.  The symmetrized singular integral
for $(-\Delta)^{s/2}a$ gives
\begin{equation}
 \|L_\sigma a\|_\infty
 \leq C\bigl(\|(-\Delta)^m a\|_\infty+
                  \|D^2(-\Delta)^m a\|_\infty\bigr).
 \label{ho:bounded-operator}
\end{equation}
If $a\in C_c^\infty$ and $x$ is outside its support, integration by
parts in the nonsingular integral gives
\begin{equation}
 L_\sigma a(x)=d_{n,\sigma}
       \int a(y)|x-y|^{-n-\sigma}\,dy.
 \label{ho:off-support}
\end{equation}
The coefficient may have either sign; only its absolute value is used.
Indeed one first applies the order-$s$ exterior formula to
$(-\Delta)^m a$, then transfers the $m$ Laplacians to its kernel.
There are no boundary terms because $a$ is compactly supported.

For $|x|=r\geq4$, split a Schwartz function into
$a_1(y)=\chi(y/r)\phi(y)$ and $a_2=\phi-a_1$, where the fixed smooth
cutoff $\chi$ equals one on $B_{1/4}$ and is supported in $B_{1/2}$.
Formula~\eqref{ho:off-support} bounds
$|L_\sigma a_1(x)|$ by $Cr^{-n-\sigma}\|\phi\|_1$.
For any fixed $N>n+\sigma$, all derivatives of $a_2$ through order
$2m+2$ have supremum norm bounded by $Cr^{-N}$ times finitely many
Schwartz seminorms of $\phi$.  Equation~\eqref{ho:bounded-operator}
therefore gives the same bound, or better, for $L_\sigma a_2(x)$.
The bounded region uses \eqref{ho:bounded-operator}.  This proves the
claimed continuous map.  In particular compact tests satisfy
$|L_\sigma\phi(x)|\leq C_\phi(1+|x|)^{-n-\sigma}$.

Now let $h\in\mathcal L_0$ be entire $L_\sigma$-harmonic.  It defines
a tempered distribution, and belongs to $\mathcal L_\sigma$.
For every $\phi\in\mathcal S$, the cutoff functions
$\chi(\cdot/R)\phi$ converge to $\phi$ in $\mathcal S$ as
$R\to\infty$.  The first assertion implies convergence of their
$L_\sigma$ images in $X_\sigma$.  Pairing with $h\in\mathcal L_\sigma$
thus extends the compact-test equation to
\begin{equation}
 \int hL_\sigma\phi=0\qquad(\phi\in\mathcal S).
 \label{ho:schwartz-harmonic}
\end{equation}
For $\theta\in C_c^\infty(\mathbb R^n\setminus\{0\})$, the function
$\phi=\mathcal F^{-1}(|\xi|^{-\sigma}\theta)$ is Schwartz and its
Fourier multiplier satisfies
$L_\sigma\phi=\mathcal F^{-1}\theta$.  Equation~\eqref{ho:schwartz-harmonic}
shows that $\widehat h$ vanishes against every test supported away from
zero.  Thus $\operatorname{supp}\widehat h\subset\{0\}$.
The structure theorem for distributions supported at a point implies
that $\widehat h$ is a finite sum of derivatives of $\delta_0$;
see \cite[Section~2.3]{Hormander2003}. Consequently $h$ is a
polynomial almost everywhere.
No nonzero polynomial is in $\mathcal L_0$: a polynomial of degree
$d$ has size at least $cr^d$ on a cone for large $r$, and its weighted
absolute integral dominates
$c\int_1^\infty r^{d-1}\,dr$, also divergent when $d=0$.
Consequently $h=0$.
\end{proof}

Origin removal for fractional Hardy--H\'enon systems is treated in
\cite[Lemma~2.1]{LiYuZhang2024}; the local Newtonian version also
appears in \cite[Lemma~2.2]{OpenAI2026}. The next argument states
the natural-tail and local even-integer alternatives together.

\begin{lemma}
\label{ho:origin-extension}
Let $0<\sigma<n$, $u\in C(\mathbb R^n)$, and
$g\geq0$ belong to $L^1_{\rm loc}(\mathbb R^n\setminus\{0\})$.
Assume either $u\in\mathcal L_\sigma$, or $\sigma$ is an even integer
and all operator pairings have their local compact-test meaning.
If $L_\sigma u\geq g$ in distributions off zero, then
$g\in L^1_{\rm loc}(\mathbb R^n)$ and the inequality holds on the
whole space. A punctured equality extends as an equality.
No sign of $u$ or initial source integrability at zero is needed;
no tail is needed in the even-integer case.
\end{lemma}
\begin{proof}
For $\phi\in C_c^\infty$, let
$h_\varepsilon(x)=\phi(x)\zeta(x/\varepsilon)$, where
$0\leq\zeta\leq1$ is smooth, supported in $B_2$, and equals one on
$B_1$. The fixed-support profiles $\phi(\varepsilon z)\zeta(z)$ have
uniformly bounded derivatives of every fixed order. Scaling and
Lemma~\ref{ho:linear} give
\[
 \|L_\sigma h_\varepsilon\|_\infty
 \leq C_\phi\varepsilon^{-\sigma},\qquad
 |L_\sigma h_\varepsilon(x)|
 \leq C_\phi\varepsilon^n|x|^{-n-\sigma}
 \quad(|x|\geq4\varepsilon)
\]
for non-even orders. For even integer orders the first bound follows
from Leibniz's rule, and the image vanishes outside $B_{2\varepsilon}$.
Local boundedness, integration over $B_{4\varepsilon}$ and
$B_1\setminus B_{4\varepsilon}$, and the natural tail in the
nonlocal case therefore give
\begin{equation}
 \left|\int uL_\sigma h_\varepsilon\right|
 \leq C_\phi\|u\|_{L^\infty(B_1)}\varepsilon^{n-\sigma}
 +C_\phi\varepsilon^n
 \int_{|x|\geq1}|u(x)||x|^{-n-\sigma}\,dx\longrightarrow0.
 \label{ho:puncture-error}
\end{equation}
The final term is omitted in the even-integer case, so no tail has
entered that argument. For $\phi\geq0$, the punctured test
$\phi-h_\varepsilon=\phi(1-\zeta(\cdot/\varepsilon))$ is nonnegative.
The inequality bounds its nonnegative source integral by
$\int uL_\sigma(\phi-h_\varepsilon)$, which has a finite limit.
Fatou first proves $\int g\phi<\infty$. Choosing $\phi$ at least one
on any prescribed compact set gives local integrability at zero.
Dominated convergence now passes the source integral to its limit and
proves the entire inequality. For a punctured equality the same
calculation for signed tests proves the entire equation.
\end{proof}

\subsubsection{Moment and annular estimates from an exterior lower bound}

\begin{lemma}
\label{gc:source-transfer}
Let $r\geq1$, $B\in\mathbb R$, $m>0$, and let $H,w\geq0$ be
measurable with $H(x)\geq m|x|^B w(x)^r$ for $|x|>1$.
For $b\in\mathbb R$ set
\[
 \mathcal T_r(b,B)=n(1-1/r)+(b-B)/r.
\]
Then the following two source-to-solution transfers hold:
\begin{enumerate}[label=(\roman*)]
\item $H\in\mathsf M_b$ implies $w\in\mathsf M_t$ for every
$t>\mathcal T_r(b,B)$;
\item if $\int_{B_R}H\leq CR^b$ for $R\geq1$, then
\begin{equation}
 \int_{R<|x|<2R}w\leq C'R^{\mathcal T_r(b,B)}
 \quad(R\geq1).
 \label{gc:annular-transfer}
\end{equation}
\end{enumerate}
This includes the endpoint $r=1$ and requires the coefficient lower bound
only on the exterior region.
\end{lemma}
\begin{proof}
For $r>1$, the pointwise inequality
$w\leq m^{-1/r}|x|^{-B/r}H^{1/r}$ and H\"older's inequality give
\[
 \int_{|x|>1}w|x|^{-t}
 \leq m^{-1/r}\left(\int_{|x|>1}H|x|^{-b}\right)^{1/r}
 \left(\int_{|x|>1}
 |x|^{r'[-t+(b-B)/r]}\right)^{1/r'}.
\]
The last integral is finite precisely when
$t>\mathcal T_r(b,B)$, proving (i). On
$\mathcal A_R=\{R<|x|<2R\}$ the same inequality gives
\[
 \int_{\mathcal A_R}w
 \leq C R^{-B/r}|\mathcal A_R|^{1-1/r}
       \left(\int_{B_{2R}}H\right)^{1/r}
 \leq C'R^{\mathcal T_r(b,B)},
\]
which proves (ii). For $r=1$, integrate
$w\leq m^{-1}|x|^{-B}H$ directly, using
$|x|^{-t-B}\leq|x|^{-b}$ in (i) and annular comparability in (ii).
\end{proof}

\subsubsection{Integral representation for nonnegative sources}

Representation of integer-order equations with positive Radon
right-hand sides was studied in
\cite{CaristiDAmbrosioMitidieri2008}. The next statement uses a
strict exterior source moment and $\mathcal L_0$ to identify the
potential at every point of a continuous solution, for all $0<s<n$.

\begin{proposition}
\label{gc:source-inversion}
Let $0<s<n$, $\lambda_s=n-s$, and let $\mu$ be a nonzero nonnegative
Radon measure on $\mathbb R^n$ such that, for some $0<t<\lambda_s$,
\[
 \int_{|y|>1}|y|^{-t}\,d\mu(y)<\infty.
\]
Then the extended potential $P=I_s\mu$ belongs to $\mathcal L_0$
and satisfies $L_sP=\mu$ globally in distributions.
If $w\in C(\mathbb R^n)\cap\mathcal L_0$ satisfies $L_sw=\mu$
globally, then $P(x)=w(x)>0$ at every point. In particular $P$ is
finite and continuous everywhere, and $\mu$ has no atoms.
% No local boundedness or density of the source is required.
\end{proposition}
\begin{proof}
Local finiteness and the strict moment give
$M_\mu(R):=\mu(B_R)\leq CR^t$ for $R\geq1$.
Tonelli and a near/far split show
\[
 \int_{B_R}P\leq CR^sM_\mu(2R)
 +CR^n\int_{|y|>2R}|y|^{-\lambda_s}\,d\mu(y)
 \leq CR^{s+t}\qquad(R\geq1).
\]
The last estimate follows by dyadic summation since $t<\lambda_s$.
It gives local integrability, almost-everywhere finiteness and
$P\in\mathcal L_0$. Lemma~\ref{ho:linear} gives
$\int P|L_s\phi|<\infty$ for compact smooth tests. The normalized
Riesz kernel satisfies
\begin{equation}
 \int K_s(x-y)L_s\phi(x)\,dx=\phi(y),
 \qquad \phi\in C_c^\infty,\quad 0<s<n.
 \label{gc:inverse}
\end{equation}
This is the fundamental-solution identity for the normalized Riesz
kernel; see \cite[Chapter~V, Section~1]{Stein1970}. The integral is
absolutely convergent. The identity follows from Fourier inversion
for the locally integrable multiplier $|\xi|^{-s}$, or Gaussian
regularization. Its cancellation with
$|\xi|^s\widehat\phi$ is between ordinary functions, without applying
a nonsmooth multiplier to a point-supported distribution.
Absolute Fubini is justified by
\[
 \iint K_s(x-y)|L_s\phi(x)|\,dx\,d\mu(y)
 =\int P|L_s\phi|<\infty,
\]
and proves $L_sP=\mu$.

Now let the continuous solution $w$ be given. The remainder $w-P\in\mathcal L_0$ is $L_s$-harmonic on
$\mathbb R^n$ and vanishes by Lemma~\ref{ho:linear}.
Thus $w=P$ almost everywhere, without a prior continuity assertion
for $P$. Fix any $x_0$ and choose points $x_j\to x_0$ where this
equality holds. Positive-kernel Fatou gives
$0<P(x_0)\leq\liminf_jP(x_j)=w(x_0)<\infty$; strict positivity
uses a positive compact source mass. In particular
$\mu(\{x_0\})=0$.

An averaging argument over balls identifies every point. Write
\[
 A_r(z)=\frac1{|B_r|}\int_{B_r}K_s(x-z)\,dx.
\]
For $z\ne0$, $A_r(z)\leq CK_s(z)$ uniformly in $r>0$:
if $|z|\geq2r$ the kernels are comparable, and if $|z|<2r$,
local kernel integrability gives
$A_r(z)\leq Cr^{-\lambda_s}\leq C'|z|^{-\lambda_s}$.
Also $A_r(z)\to K_s(z)$ for each $z\ne0$. Since
$\int K_s(x_0-y)\,d\mu(y)=P(x_0)<\infty$, dominated convergence
and Tonelli give
\[
 P(x_0)=\lim_{r\downarrow0}\int A_r(y-x_0)\,d\mu(y)
 =\lim_{r\downarrow0}\fint_{B_r(x_0)}P
 =\lim_{r\downarrow0}\fint_{B_r(x_0)}w=w(x_0).
\]
The argument applies to every $x_0$, including the origin, and proves
all remaining assertions.
\end{proof}

\begin{corollary}
\label{gc:pair-inversion}
Let $0<s_1,s_2<n$, $p,q>0$, and let $a,b>0$ be continuous off zero.
Suppose $u,v\in C(\mathbb R^n)\cap\mathcal L_0$ are nonnegative,
strictly positive off zero, and solve
$L_{s_1}u=g:=av^p$, $L_{s_2}v=f:=bu^q$ globally in distributions.
Assume $f,g\in L^1_{\mathrm{loc}}$ and strict source moments
$g\in\mathsf M_{t_g}$, $f\in\mathsf M_{t_f}$ for some
$0<t_g<n-s_1$ and $0<t_f<n-s_2$.
Then the two potential identities hold everywhere, both origin values
are positive, and
\begin{equation}
 \int_{B_1}a(y)|y|^{s_1-n}\,dy<\infty,\qquad
 \int_{B_1}b(y)|y|^{s_2-n}\,dy<\infty,
 \qquad fu+gv\in L^1_{\mathrm{loc}}.
 \label{gc:pair-origin}
\end{equation}
% No product restriction, source scaling, origin coefficient bound or
% sum-of-orders restriction is used in this conditional inversion.
\end{corollary}
\begin{proof}
Apply
Proposition~\ref{gc:source-inversion} separately to $(s_1,u,g)$ and
$(s_2,v,f)$.  It identifies every point, including the origin, and gives positive
origin values.  Positive lower bounds for
$v,u$ on a small ball then turn the finite origin potentials into
the corresponding coefficient integrals in \eqref{gc:pair-origin}.
Boundedness of the coefficients on compact annuli away from zero
extends these integrability conditions to
$B_1$.  Source local integrability times locally bounded solutions
gives the energy assertion.
\end{proof}

\subsection{Power-source representation in three PDE classes}
\label{sec:paper1-power-representation}

\subsubsection{Integer-order representation without initial tails}
\label{ui:section}

At each even integer order, the operator image of a compact test remains
compactly supported.  Local Leibniz estimates and balanced cutoffs
therefore yield bounds for the source masses on balls from exterior coefficient lower bounds
at different even integer orders, without an initial tail assumption.
These bounds produce the solution class $\mathcal L_0$ and strict
source moments before the linear representation theorem is applied.
The exponents $\alpha_\infty,\beta_\infty$ are determined by the coefficient lower bounds; the
variable-coefficient equations need not have a dilation symmetry.
Each order must be below dimension, but their sum is unrestricted.
The compact-test argument below belongs to the general test-function
approach to a priori estimates and nonexistence; see
Mitidieri and Pokhozhaev~\cite{MitidieriPokhozhaev2001}; related source
estimates are given in \cite[Proposition~2.6]{Yang2021}.
We give the estimates needed here, including the exponent one, directly.

Let $m_1,m_2\geq1$ be integers, and put
\begin{equation}
 \begin{gathered}
 \sigma_i=2m_i<n,\qquad \lambda_i=n-\sigma_i,\qquad
 L_i=(-\Delta)^{m_i}\quad(i=1,2),\\
 p,q\geq1,\qquad pq>1.
 \end{gathered}
 \label{ui:parameters}
\end{equation}
Let $a,b>0$ be continuous on $\mathbb R^n\setminus\{0\}$, and suppose
\begin{equation}
 \begin{split}
 a(x)&\geq m_a|x|^{A_\infty},\qquad
 b(x)\geq m_b|x|^{B_\infty}\quad(|x|\geq1),\\
 m_a,m_b&>0,\qquad A_\infty,B_\infty\in\mathbb R.
 \end{split}
 \label{ui:exterior-floors}
\end{equation}
Assume that the solutions belong to the following class
\begin{equation}
 \begin{gathered}
 u\in C(\mathbb R^n)\cap
 C^{2m_1}_{\mathrm{loc}}(\mathbb R^n\setminus\{0\}),\qquad
 v\in C(\mathbb R^n)\cap
 C^{2m_2}_{\mathrm{loc}}(\mathbb R^n\setminus\{0\}),\\
 u,v\geq0\ \hbox{on }\mathbb R^n,\qquad
 u,v>0\ \hbox{off zero},
 \end{gathered}
 \label{ui:class}
\end{equation}
and the pointwise equations
\begin{equation}
 L_1u=g:=av^p,\qquad L_2v=f:=bu^q
 \quad\hbox{on }\mathbb R^n\setminus\{0\}.
 \label{ui:pde}
\end{equation}
Assign $f(0)=g(0)=0$ only to choose representatives.  Source local
integrability at zero and positive origin values are not assumed.
Define
\begin{equation}
 \alpha_\infty=
 \frac{\sigma_1+A_\infty+p(\sigma_2+B_\infty)}{pq-1},\qquad
 \beta_\infty=
 \frac{\sigma_2+B_\infty+q(\sigma_1+A_\infty)}{pq-1}.
 \label{ui:floor-scales}
\end{equation}
The balance identities are
\begin{equation}
 p\beta_\infty=\alpha_\infty+\sigma_1+A_\infty,\qquad
 q\alpha_\infty=\beta_\infty+\sigma_2+B_\infty.
 \label{ui:balance}
\end{equation}
No sign of these exponents is imposed until the tail conclusions below.

The even-integer alternative of Lemma~\ref{ho:origin-extension}
applies separately to the two equations whose solutions are continuous
at zero. It supplies
locally integrable sources and entire compact-test equations before
any source bound, tail estimate or potential representation is used.

\begin{lemma}
\label{ui:source-bounds}
Under \eqref{ui:parameters}--\eqref{ui:floor-scales}, there is a finite
constant $C$, independent of $R\geq1$, such that
\begin{equation}
 M_g(R):=\int_{B_R}g\leq CR^{\lambda_1-\alpha_\infty},\qquad
 M_f(R):=\int_{B_R}f\leq CR^{\lambda_2-\beta_\infty}
 \quad(R\geq1).
 \label{ui:source-masses}
\end{equation}
Every putative pair consequently satisfies
\begin{equation}
 \alpha_\infty\leq\lambda_1,\qquad
 \beta_\infty\leq\lambda_2.
 \label{ui:upper-scales}
\end{equation}
The estimates include either linear endpoint and precede every
tail or potential representation.
If additionally $\alpha_\infty,\beta_\infty>0$, then
$u,v\in\mathcal L_0$ and
\begin{equation}
 g\in\mathsf M_{\lambda_1-\alpha_\infty/2},\qquad
 f\in\mathsf M_{\lambda_2-\beta_\infty/2},
 \qquad
 0<\lambda_1-\alpha_\infty/2<\lambda_1,\quad
 0<\lambda_2-\beta_\infty/2<\lambda_2.
 \label{ui:strict-moments}
\end{equation}
% These tails and strict source moments are conclusions, not inputs.
\end{lemma}
\begin{proof}
Lemma~\ref{ho:origin-extension} permits entire compact testing.
Choose $0\leq\chi\leq1$ in $C_c^\infty(B_2)$, equal to one
on $B_1$, put $\chi_R=\chi(\cdot/R)$, and choose
\begin{equation}
 \begin{gathered}
 d_0=pq-1,\qquad
 \ell=\frac{\max\{q\sigma_1,p\sigma_2\}+2}{d_0},\\
 k_g=\lceil(p+1)\ell\rceil,\qquad
 k_f=\lceil(q+1)\ell\rceil.
 \end{gathered}
 \label{ui:cutoff-powers}
\end{equation}
The ceiling bounds give
\[
 qk_g-k_f\geq d_0\ell-1\geq q\sigma_1+1,\qquad
 pk_f-k_g\geq d_0\ell-1\geq p\sigma_2+1.
\]
Thus $k_g>\sigma_1$, $k_f>\sigma_2$, and both
$k_g-\sigma_1-k_f/q$ and $k_f-\sigma_2-k_g/p$ are nonnegative.
For each $i$ and any integer $k>\sigma_i$, Leibniz's rule gives
\begin{equation}
 |L_i(\chi_R^k)|\leq C_{i,k}R^{-\sigma_i}
                  \chi_R^{k-\sigma_i},\qquad
 \operatorname{supp}L_i(\chi_R^k)
 \subset\{R\leq|x|\leq2R\}.
 \label{ui:cutoff-image}
\end{equation}
At most $\sigma_i$ factors are differentiated in any product;
the remaining factors give the displayed cutoff power.  The total
derivative order contributes $R^{-\sigma_i}$, and the image vanishes
where the cutoff is constant.

For $R\geq1$, let $\mathcal A_R=\{R<|x|<2R\}$ and set
\[
 G_R=\int g\chi_R^{k_g},\qquad
 F_R=\int f\chi_R^{k_f}.
\]
These masses are finite and strictly positive.
The lower bound for $b$ is used only on $\mathcal A_R$, where
$u=b^{-1/q}f^{1/q}\leq
m_b^{-1/q}|x|^{-B_\infty/q}f^{1/q}$.
Testing $L_1u=g$ and using \eqref{ui:cutoff-image} gives
\[
 G_R\leq CR^{-\sigma_1}\int_{\mathcal A_R}
 |x|^{-B_\infty/q}(f\chi_R^{k_f})^{1/q}
                 \chi_R^{k_g-\sigma_1-k_f/q}.
\]
Factoring is done where $\chi_R>0$; the original integrand vanishes
on the remaining set.  If $q>1$, H\"older with $q,q'$ and annular
comparability give
$CR^{-\sigma_1-B_\infty/q}
|\mathcal A_R|^{1-1/q}F_R^{1/q}$.
If $q=1$, the nonnegative remaining cutoff power instead gives the
direct bound $CR^{-\sigma_1-B_\infty}F_R$.
Both cases imply
\begin{equation}
 G_R\leq CR^{\xi_g}F_R^{1/q},\qquad
 \xi_g=\lambda_1-\frac{n+B_\infty}{q}.
 \label{ui:coupled-g}
\end{equation}
Testing $L_2v=f$ and using the lower bound for $a$ in the same way, including
direct integration when $p=1$, gives
\begin{equation}
 F_R\leq CR^{\xi_f}G_R^{1/p},\qquad
 \xi_f=\lambda_2-\frac{n+A_\infty}{p}.
 \label{ui:coupled-f}
\end{equation}
No division by $q-1$ or $p-1$ and no global inverse-weight norm
is used at a linear endpoint.  All coefficient factors are estimated
on the exterior annulus, so both coefficient exponents may be arbitrary real
numbers.

Eliminating the positive masses and using $pq>1$ yields
\[
 G_R\leq CR^{\frac{pq}{pq-1}(\xi_g+\xi_f/q)}
       =CR^{\lambda_1-\alpha_\infty},\qquad
 F_R\leq CR^{\frac{pq}{pq-1}(\xi_f+\xi_g/p)}
       =CR^{\lambda_2-\beta_\infty}.
\]
The last equalities use the definitions \eqref{ui:floor-scales}
with the corresponding order on each component.
Both cutoff powers equal one on $B_R$, proving
\eqref{ui:source-masses}.  If $\alpha_\infty>\lambda_1$, the first
bound tends to zero as $R\to\infty$, contradicting
$M_g(R)\geq M_g(1)>0$.  The second bound similarly excludes
$\beta_\infty>\lambda_2$.  No positivity of the algebraic exponents has
been used.

If the same coefficient lower bounds hold for every $x\ne0$, the entire
compact-test argument above is valid for every $R>0$, with one
constant independent of $R$.  The origin-extension lemma supplies
the local source integrability needed even when $R<1$.
Under only the exterior lower bounds \eqref{ui:exterior-floors}, the
asserted source bounds remain restricted to $R\geq1$.

Finally suppose $\alpha_\infty,\beta_\infty>0$.
Apply Lemma~\ref{gc:source-transfer}(ii) to the two source bounds.
The balance identities \eqref{ui:balance} give
\[
 \int_{\mathcal A_R}u\leq CR^{n-\alpha_\infty},\qquad
 \int_{\mathcal A_R}v\leq CR^{n-\beta_\infty}
 \quad(R\geq1).
\]
Dyadic summation, with local integrability supplied by continuity,
therefore yields $u,v\in\mathcal L_0$. More generally it gives
$u\in\mathsf M_T$ for $T>n-\alpha_\infty$ and
$v\in\mathsf M_T$ for $T>n-\beta_\infty$.
The source bounds similarly imply
$g\in\mathsf M_t$ for $t>\lambda_1-\alpha_\infty$ and
$f\in\mathsf M_t$ for $t>\lambda_2-\beta_\infty$.
Choosing the two midpoint exponents gives \eqref{ui:strict-moments};
the upper bounds on the exponents make these midpoint exponents
positive, while the positivity of $\alpha_\infty,\beta_\infty$ places them below
their respective Riesz thresholds. No limiting or endpoint
moment is used.
\end{proof}

\begin{proposition}
\label{ui:representation}
Under \eqref{ui:parameters}--\eqref{ui:floor-scales}, suppose
$\alpha_\infty,\beta_\infty>0$.  Then $u,v\in\mathcal L_0$,
$u(0),v(0)>0$, and
\begin{equation}
 u=I_{\sigma_1}(av^p),\qquad
 v=I_{\sigma_2}(bu^q)
 \quad\hbox{at every point of }\mathbb R^n.
 \label{ui:exact}
\end{equation}
Moreover,
\begin{equation}
 \int_{B_1}a(y)|y|^{-\lambda_1}\,dy<\infty,\qquad
 \int_{B_1}b(y)|y|^{-\lambda_2}\,dy<\infty,\qquad
 fu+gv\in L^1_{\mathrm{loc}}(\mathbb R^n).
 \label{ui:origin-admissibility}
\end{equation}
% There is no restriction on $\sigma_1+\sigma_2$ and no initial natural
% tail, global boundedness, prescribed growth, origin coefficient bound,
% source-decay input, Euler bound, orbit condition, or nonlinear-energy
% assumption.
\end{proposition}
\begin{proof}
Lemma~\ref{ho:origin-extension} supplies the entire distributional
equations and source local integrability.
Lemma~\ref{ui:source-bounds} supplies $u,v\in\mathcal L_0$ and the
strict source moment for each equation's own order.
Both sources are nonzero because the solutions are positive off zero.
These are precisely the inputs of Corollary~\ref{gc:pair-inversion}
with the two orders $\sigma_1,\sigma_2$.
Its representation statement identifies every point and gives the
two kernel-weighted coefficient integrals and local nonlinear energy.
All inversion inputs have been derived before this invocation.
\end{proof}

\begin{proof}[Proof of Corollary~\ref{thm:integer-front}]
Set $\sigma=2m$, $\lambda=n-2m$ and take
$a(x)=|x|^A$, $b(x)=|x|^B$. Lemma~\ref{ui:source-bounds}, applied
without an initial tail, gives
$\alpha_\infty,\beta_\infty\le\lambda$ for the exponents in
\eqref{ui:floor-scales}. Their sum satisfies
\[
 \alpha_\infty+\beta_\infty-\lambda
 =\frac{(p+1)(q+1)}{pq-1}
   \left(\frac{n+A}{p+1}+\frac{n+B}{q+1}-\lambda\right)>0.
\]
Both exponents are therefore positive. Proposition~\ref{ui:representation}
gives the exact Riesz identities everywhere and finite local nonlinear
energy. Finally, the Euler derivatives equal $A,B$, and the normalized
coefficient logarithms are Lipschitz on $\mathcal A$.
Theorem~\ref{thm:dini-front}, with $\Omega(s)=s$, gives the
contradiction. No initial-tail moment iteration or low-order comparison
principle is used in this proof.
\end{proof}

\subsubsection{Representation from initial tail assumptions}
\label{r5:natural-tail-section}

This part treats the natural-tail class of
Corollary~\ref{cor:pde-front}(i) at arbitrary orders $0<\sigma_i<n$.
Its source-moment iteration is independent of the integer-order argument
and is not used in the proof of Corollary~\ref{thm:integer-front}.

\paragraph{Weighted tail estimates for nonnegative sources}
\begin{lemma}
\label{tb:positive-test}
Suppose $w\in C(\mathbb R^n)\cap\mathcal L_\sigma$, $L_\sigma w=H\geq0$
on the whole space, and $H\in L^1_{\mathrm{loc}}$. If
$w\in\mathsf M_t$ for some $t\leq n+\sigma$, then
\[
 H\in\mathsf M_\tau\qquad
       \bigl(\tau>\max\{0,t-\sigma\}\bigr).
\]
\end{lemma}
\begin{proof}
For $\tau>0$ choose a positive smooth $\psi_\tau$, constant near zero
and equal to $|x|^{-\tau}$ for $|x|\geq2$. Let $0\leq\chi\leq1$ be
smooth, equal to one on $B_1$ and supported in $B_2$, and set
$\psi_{\tau,R}=\psi_\tau\chi(x/R)$. The cutoff estimate of
\cite{AoGonzalezHyderWei2022}, in the form of
\cite[Lemma~2.3]{Yang2021} with order parameter $\sigma/2$, gives
local uniform convergence
$L_\sigma\psi_{\tau,R}\to L_\sigma\psi_\tau$ and, uniformly for large
$R$,
\[
 |L_\sigma\psi_{\tau,R}(x)|\leq C_\tau
 \begin{cases}
 (1+|x|)^{-\sigma-\tau},&0<\tau<n,\\
 (1+|x|)^{-n-\sigma},&\tau>n.
 \end{cases}
\]
Its exact tail starts at radius one; applying it to
$2^\tau\psi_\tau(2z)$ and undoing this fixed rescaling gives the stated
bound. No sign of the operator image is asserted.

When $\tau<n$, its image exponent satisfies $\sigma+\tau>t$;
when $\tau>n$, the majorant is the initial natural tail. Thus dominated
convergence against $w$, with local boundedness on bounded sets, gives
\[
 \int H\psi_{\tau,R}=\int wL_\sigma\psi_{\tau,R}
 \longrightarrow\int wL_\sigma\psi_\tau\in\mathbb R.
\]
Fatou on the nonnegative source side first proves
$\int H\psi_\tau<\infty$. Only afterwards does dominated convergence,
using $0\leq\chi\leq1$, identify it with the displayed limit. This proves
the asserted tail moment. The exponent $\tau=n$ has a logarithmic image
bound; when it is requested here, choose
$\max(0,t-\sigma)<\tau'<n$ and use
$\mathsf M_{\tau'}\subset\mathsf M_n$.
\end{proof}

\begin{lemma}
\label{tb:cutoff-obstruction}
Under the entire equation $L_\sigma w=H\geq0$, with a nonzero locally
integrable source and $w\in C(\mathbb R^n)\cap\mathcal L_\sigma$, there
cannot be a moment $w\in\mathsf M_t$ with $t\le\sigma$.
\end{lemma}
\begin{proof}
Such a moment implies $w\in\mathsf M_\sigma$. For
$\chi_R=\chi(x/R)$, the linear estimates of Lemma~\ref{ho:linear}
and scaling give
\[
 |L_\sigma\chi_R(x)|\leq CR^{-\sigma}\quad(|x|\leq4R),\qquad
 |L_\sigma\chi_R(x)|\leq CR^n|x|^{-n-\sigma}\quad(|x|>4R).
\]
The latter image vanishes outside the support in the even-integer case.
Consequently, for $R\geq1$, these images are bounded by
$C(1+|x|)^{-\sigma}$ and tend to zero at every fixed point. Local
boundedness and the moment permit dominated convergence in
\[
 \int H\chi_R=\int wL_\sigma\chi_R\longrightarrow0.
\]
The left side is bounded below by a fixed positive compact source mass
for all sufficiently large $R$, a contradiction.
\end{proof}

\paragraph{Exterior coefficient lower bounds}
\label{gc:section}

The source-to-solution estimate of Lemma~\ref{gc:source-transfer}
permits general exterior lower bounds. The representation statement
allows unequal orders; nonlinear nonexistence is deduced only at equal
orders.

For the following representation results let $0<\sigma_1,\sigma_2<n$,
$\lambda_i=n-\sigma_i$, $p,q\geq1$ and $pq>1$. Let $a,b>0$ be
continuous off zero and suppose
\begin{equation}
 a(x)\geq m_a|x|^{A_\infty},\qquad
 b(x)\geq m_b|x|^{B_\infty}\quad(|x|\geq1),
 \qquad m_a,m_b>0,\quad A_\infty,B_\infty\in\mathbb R.
 \label{gc:exterior-floors}
\end{equation}
Assume $u,v\in C(\mathbb R^n)$ are nonnegative, strictly positive off
zero, $u\in\mathcal L_{\sigma_1}$, $v\in\mathcal L_{\sigma_2}$, and
\begin{equation}
 L_{\sigma_1}u=g:=av^p,\qquad L_{\sigma_2}v=f:=bu^q
 \quad\hbox{in distributions off zero}.
 \label{gc:pde}
\end{equation}
Define the algebraic exponents
\begin{equation}
 \alpha_\infty=
 \frac{\sigma_1+A_\infty+p(\sigma_2+B_\infty)}{pq-1},\qquad
 \beta_\infty=
 \frac{\sigma_2+B_\infty+q(\sigma_1+A_\infty)}{pq-1}.
 \label{gc:floor-scales}
\end{equation}

\begin{lemma}
\label{gc:moment-bootstrap}
Under \eqref{gc:exterior-floors}--\eqref{gc:floor-scales}, every putative
pair has $\alpha_\infty\leq\lambda_1$, $\beta_\infty\leq\lambda_2$.
If both algebraic exponents are positive, then $u,v\in\mathcal L_0$ and
\[
 g\in\mathsf M_{t_g},\quad 0<t_g<\lambda_1,
 \qquad f\in\mathsf M_{t_f},\quad 0<t_f<\lambda_2
\]
for some exponents satisfying the displayed strict inequalities.
\end{lemma}
\begin{proof}
Lemma~\ref{ho:origin-extension} applies separately to both equations.
It supplies locally integrable sources and entire distributional
equations. By Lemma~\ref{tb:positive-test}, initially
$g,f\in\mathsf M_{n+\varepsilon}$ for every $\varepsilon>0$.
Set
\[
 C=\lambda_1-(n+B_\infty)/q,\quad
 D=\lambda_2-(n+A_\infty)/p,\quad
 a_0=\lambda_1-\alpha_\infty,\quad
 b_0=\lambda_2-\beta_\infty.
\]
The balance identities imply
\[
 a_0=C+b_0/q,\quad b_0=D+a_0/p,\quad
 C+D/q=(1-1/(pq))a_0,\quad
 D+C/p=(1-1/(pq))b_0.
\]
Lemma~\ref{gc:source-transfer}(i), applied to
$(f,u,q,B_\infty)$ and $(g,v,p,A_\infty)$, transfers each known
source moment to the other solution. In the notation above the
thresholds are $\sigma_1+C+b/q$ and $\sigma_2+D+a/p$,
respectively. The same formula covers either linear endpoint.

Let $P(a,b)$ mean that $g\in\mathsf M_{a+\varepsilon}$ and
$f\in\mathsf M_{b+\varepsilon}$ for every $\varepsilon>0$.
The preceding inequalities and Lemma~\ref{tb:positive-test} give
$P(a,b)\Rightarrow P(G(a,b))$, where
\[
 F(a,b)=((C+b/q)_+,(D+a/p)_+),\qquad
 G(a,b)=(\min(n,F_1(a,b)),\min(n,F_2(a,b))).
\]
The cap retains the initial known moment when a proposed improvement
would exceed $n$. Every application of the test estimate uses a strictly positive surplus.
For example, if $F_1<n$, choose a solution moment strictly above
$\sigma_1+C+b/q$ and at most $n+\sigma_1$, close enough for the
desired new source moment. If $F_1\ge n$, retain the initial test.
The same argument treats the second coordinate and both linear endpoints.
The raw solution thresholds are therefore
\[
 T_u(a,b)=\sigma_1+C+b/q,\qquad
 T_v(a,b)=\sigma_2+D+a/p.
\]

Starting at $z_0=(n,n)$, $z_{k+1}=G(z_k)$ decreases.
Both $F^2$ and $G^2$ have Lipschitz constant at most $(pq)^{-1}<1$
in the maximum norm, so they have unique fixed points and their
iterates converge. Direct substitution gives
\[
 z_F=(\max(0,C,a_0),\max(0,D,b_0)).
\]
Indeed the balance identities give
\[
 \bigl(C+\max\{0,D,b_0\}/q\bigr)_+
 =\bigl(\max\{C,(1-1/(pq))a_0,a_0\}\bigr)_+
 =\max\{0,C,a_0\},
\]
and the second coordinate is symmetric. Starting the $G$ iteration at $z_F$
shows $z_G\le z_F$. All strict inequalities at $z_G$ hold at a
finite iterate, and hence give tail integrability with strictly admissible exponents.

Suppose $a_0<0$. If $b_0\ge0$, then $C<0$,
$D>b_0\ge0$, and $z_F=(0,D)$, so
$T_u(z_F)-\sigma_1=C+D/q<0$.
If $b_0<0$ and $D\ge0$, the same conclusion holds.
If $b_0<0$, $D<0$ and $C\ge0$, then $z_F=(C,0)$ and
$T_v(z_F)-\sigma_2=D+C/p<0$.
If $C,D<0$, both raw thresholds at $(0,0)$ are below their orders.
Thus at a finite iterate at least one solution has a moment
strictly below its own order. Lemma~\ref{tb:cutoff-obstruction}
contradicts the positive source. This excludes $a_0<0$;
exchanging the components excludes $b_0<0$.

When both algebraic exponents are positive,
$0\le a_0<\lambda_1$ and $0\le b_0<\lambda_2$.
Now $C\le a_0$, $D\le b_0$, and both fixed points equal
$(a_0,b_0)$, with the cap inactive. Finite iteration gives source
moments below $\lambda_1,\lambda_2$ respectively. The limiting raw
solution thresholds are $n-\alpha_\infty<n$ and
$n-\beta_\infty<n$; a finite iterate gives solution moments
below $n$, hence $u,v\in\mathcal L_0$. No limiting endpoint moment
or uniform moment constant is claimed.
\end{proof}

\begin{proposition}
\label{gc:representation}
Under \eqref{gc:exterior-floors}--\eqref{gc:floor-scales}, suppose
$\alpha_\infty,\beta_\infty>0$. Then $u(0),v(0)>0$, $u,v\in\mathcal L_0$,
and
\begin{equation}
 u=I_{\sigma_1}(av^p),\qquad v=I_{\sigma_2}(bu^q)
 \quad\hbox{at every point of }\mathbb R^n.
 \label{gc:exact}
\end{equation}
Moreover,
\begin{equation}
 \int_{B_1}a(y)|y|^{-\lambda_1}\,dy<\infty,
 \qquad \int_{B_1}b(y)|y|^{-\lambda_2}\,dy<\infty,
 \qquad fu+gv\in L^1_{\mathrm{loc}}.
 \label{gc:origin-admissibility}
\end{equation}
% No Euler bound, orbit condition, origin coefficient bound,
% source-decay assumption or nonlinear energy is imposed. The sum of the
% orders need not be below dimension.
\end{proposition}
\begin{proof}
Lemma~\ref{gc:moment-bootstrap} supplies the entire equations, local
source integrability, both solution $\mathcal L_0$ inputs and strict
source moments below the respective Riesz thresholds.
Corollary~\ref{gc:pair-inversion}, with $s_i=\sigma_i$, gives all
conclusions.  The lower-bound exponents play no role in the representation or
coefficient integrability conditions at zero after those inputs have been obtained.
\end{proof}

\begin{proposition}
\label{gc:floor-gap-counterexample}
For $n=5$, $\sigma=2$, $p=q=2$, the functions
\[
 u=v=(1+|x|^2)^{-1},\qquad
 a=b=\frac{2(|x|^2+5)}{1+|x|^2}\geq2
\]
give a positive PDE pair in the natural-tail class, continuous at zero, with smooth
positive coefficients. Their exterior lower-bound exponents $A_\infty=B_\infty=0$
give a gap of $1/3>0$. Their optimal gap based on global Euler lower bounds is
negative.
\end{proposition}
\begin{proof}
Direct differentiation gives $-\Delta u=av^2$ and the symmetric
equation. Their tails and origin properties are immediate.
Writing $s=|x|^2$, the Euler derivative is
$-8s/[(s+1)(s+5)]$, whose infimum is $\sqrt5-3$.
The largest possible gap from global Euler lower bounds is therefore
$(2\sqrt5-5)/3<0$. Thus the general representation may be used, but
the pure-power Corollary~\ref{cor:power} cannot be applied to these
variable coefficients.
\end{proof}

\begin{remark}
\label{gc:scope}
The representation conclusions do not assert
$A_\infty>-\sigma_1$ or $B_\infty>-\sigma_2$.
For example, with $n=5$, $\sigma_1=\sigma_2=2$, $p=q=2$, take
\[
 u=(1+r^2)^{-7/5},\qquad v=(1+r^2)^{-1/5},
\]
\[
 a=\frac{14}{5}(5+r^2/5)(1+r^2)^{-3},\qquad
 b=\frac25(5+13r^2/5)(1+r^2)^{3/5}.
\]
The identity
$-\Delta(1+r^2)^{-h}=2h[5+(3-2h)r^2](1+r^2)^{-h-2}$
verifies both PDEs. Their exterior lower-bound exponents may be chosen as
$A_\infty=-4$, $B_\infty=16/5$, with positive algebraic exponents
$\alpha_\infty=14/5$, $\beta_\infty=2/5$.
The coefficients are smooth and positive at zero, although the first
exterior exponent is below $-2$. The origin statement is instead
\eqref{gc:origin-admissibility}.

The initial natural-tail assumptions, continuity at the origin and
$p,q\geq1$ remain hypotheses of this representation argument. Singular-origin defects and sublinear
source transfer are not treated. Unequal-order representation is not a
general unequal-order nonlinear Liouville theorem. The Euler lower bounds in the nonexistence theorem are global;
an exterior bound alone is insufficient, as
Proposition~\ref{gc:floor-gap-counterexample} shows.

\end{remark}

\subsubsection{Low-order positive minorants}
\label{pd:section}\label{pp:section}

The low-order argument uses a positive Green minorant and continuity
at zero. This representation strategy appears in
\cite[Lemmas~2.2 and~3.1]{LiYuZhang2024} and
\cite[Proposition~1.1]{MeiWei2026}. We give the source truncation and
punctured distributional comparison arguments needed here; the
comparison principle is the standard low-order one, cf.
\cite{Silvestre2007}. No singular-origin classification is required.

Throughout this part let
\[
 n\geq1,\qquad 0<\sigma<\min\{2,n\},\qquad
 \lambda=n-\sigma,\qquad L=(-\Delta)^{\sigma/2},
\]
and choose the standard positive normalizations so that
\[
 L\phi(x)=C_{n,\sigma}\operatorname{PV}
 \int_{\mathbb R^n}\frac{\phi(x)-\phi(y)}{|x-y|^{n+\sigma}}\,dy,
 \qquad K_\sigma(x)=c_{n,\sigma}|x|^{-\lambda},
 \qquad LK_\sigma=\delta_0.
\]
Write $I_\sigma H=K_\sigma*H$ and use the tail class \eqref{ho:tail}.
For $w\in\mathcal L_\sigma$ the distribution $Lw$ is defined by
$\langle Lw,\phi\rangle=\int wL\phi$, since
\begin{equation}\label{pd:test-tail}
 |L\phi(x)|\leq C_\phi(1+|x|)^{-n-\sigma}
 \quad\bigl(\phi\in C_c^\infty(\mathbb R^n)\bigr).
\end{equation}
The tail estimate is Lemma~\ref{ho:linear}; the
fundamental-solution formula is \eqref{gc:inverse} with $s=\sigma$.
Unless a pointwise hypothesis is explicitly stated, equations and
inequalities involving $L$ below are distributional on the stated
domains.  A function in
$C^{1,1}_{\mathrm{loc}}(\mathbb R^n\setminus\{0\})\cap\mathcal L_\sigma$
has a pointwise operator off zero and its pointwise equation implies the
distributional one there.  Indeed, symmetrizing the truncated integral
against a compact test supported away from zero gives $\int wL\phi$
in the limit: local Lipschitz differences give an integrable
$O(|x-y|^{2-n-\sigma})$ diagonal, local $L^1$ handles separated bounded
regions, and the natural tail handles infinity.  The local
$C^{1,1}$ bound also gives uniform convergence of the pointwise
truncations on the test support.  No pointwise operator is inferred
from $C^1$ alone when $\sigma\geq1$.

The nonnegative zero-exterior Green kernel of $B_R$ is
\begin{equation}\label{pd:green-formula}
 \begin{split}
 G_R(x,y)&=\kappa_{n,\sigma}|x-y|^{-\lambda}
 \int_0^{\tau_R(x,y)}t^{\sigma/2-1}(1+t)^{-n/2}\,dt,\\
 \tau_R(x,y)&=
 \frac{(R^2-|x|^2)(R^2-|y|^2)}{R^2|x-y|^2},
 \end{split}
\end{equation}
for distinct $x,y\in B_R$, extended by zero otherwise.  The constant
is fixed by $\kappa_{n,\sigma}\int_0^\infty
t^{\sigma/2-1}(1+t)^{-n/2}\,dt=c_{n,\sigma}$.  Thus
\begin{equation}\label{pd:green-limit}
 0\leq G_R(x,y)\uparrow K_\sigma(x-y)\quad(R\to\infty).
\end{equation}
The upper limit in \eqref{pd:green-formula} increases with $R$;
the integral to infinity is finite because $\sigma<n$.
For a smooth source near $\overline B_R$, the Green potential is
globally continuous, zero outside $B_R$, and solves its Poisson
equation inside.  These facts, in all dimensions in the present range,
are the linear inputs of \cite[Theorems~3.1--3.2]{Bucur2016}.

\paragraph{A source and barrier lemma for solutions and supersolutions}

\begin{lemma}
\label{pp:annulus-comparison}
Let $D$ be a bounded ball or annulus.  Suppose $z\in\mathcal L_\sigma$
is continuous on a neighborhood of $\overline D$, $z\geq0$ almost
everywhere on $D^c$, and $Lz\geq0$ in $\mathcal D'(D)$.  Then
$z\geq0$ on $D$.
\end{lemma}
\begin{proof}
Continuity and exterior nonnegativity give nonnegative boundary values.
If $m=\min_{\overline D}z<0$, a fixed boundary collar has $z>m/2$.
Convolve with a nonnegative smooth approximate identity
$\rho_\varepsilon$ supported in $B_\varepsilon$.
Bounded translates have comparable natural weights, so
$z_\varepsilon=z*\rho_\varepsilon\in\mathcal L_\sigma$, and
$z_\varepsilon\to z$ uniformly on compact neighborhoods of
$\overline D$.  For small $\varepsilon$, the convolution is greater
than $3m/5$ on a smaller fixed boundary collar and less than $3m/4$
at an original minimum point.  Outside the $\varepsilon$-neighborhood
of $D$ it is nonnegative.  Hence its negative global minimum is
attained in a fixed interior region at distance greater than
$\varepsilon$ from $\partial D$.

There $Lz_\varepsilon=(Lz)*\rho_\varepsilon\geq0$: this local
convolution tests $Lz$ only inside $D$.  Commutation follows by
translation invariance, \eqref{pd:test-tail}, and the natural tail.
The smooth convolution has a pointwise operator, with its second
difference integrable at zero because $\sigma<2$.  At the negative
global minimum this operator is strictly negative.  All differences
in the defining integral are nonpositive, and those against the
nonnegative values outside a large ball are strictly negative.
This contradiction proves comparison without a pointwise operator
assumption on $z$.
\end{proof}

\begin{lemma}
\label{pp:automatic-source}\label{pd:removable-minorant}
Let $w\geq0$ belong to $\mathcal L_\sigma$ and have a finite continuous
representative on $\mathbb R^n\setminus\{0\}$.  Let
$H\geq0$ belong to $L^1_{\mathrm{loc}}(\mathbb R^n\setminus\{0\})$.
If $Lw\geq H$ there, then $H\in L^1_{\mathrm{loc}}(\mathbb R^n)$ and
\begin{equation}\label{pd:linear-minorant}
 I_\sigma H(x)\leq w(x)<\infty\quad(x\ne0).
\end{equation}
If $w$ is also continuous at zero, then $I_\sigma H(0)\leq w(0)$
and $Lw\geq H$ on the whole space in distributions.  In this
regular-origin case a punctured equality extends as an equality.
% No origin integrability of $H$, energy condition, or power-weight
% restriction is assumed.  The source may be merely measurable.
\end{lemma}
\begin{proof}
Fix $R>0$.  We first justify the Green potential for a bounded measurable
$h\geq0$ supported in a compact annulus inside $B_R$.
Choose smooth compactly supported approximations $h_m\to h$ in
$L^r(B_R)$ for some $r>n/\sigma$, and denote their Green potentials
by $P_m$.  The sources need not minorize $H$ at this approximation
stage.  With $r'=r/(r-1)$, one has $\lambda r'<n$ and
\[
 \sup_x\left|\int_{B_R}G_R(x,y)(h_m-h)(y)\,dy\right|
 \leq C_R\|h_m-h\|_{L^r(B_R)},
\]
by $G_R\leq K_\sigma$ and H\"older's inequality.
Consequently $P=G_Rh$ is a uniform limit of the globally continuous,
zero-exterior $P_m$ and is bounded and continuous on all of
$\mathbb R^n$.  Uniform convergence, their bounded common support,
and \eqref{pd:test-tail} pass $LP_m=h_m$ to $LP=h$ in $B_R$.
Also $P\geq0$ and $P\in\mathcal L_\sigma$.

For any such $h\leq H$, put $M=\|P\|_\infty$ and
\[
 W_\delta(x)=w(x)-P(x)+M\delta^\lambda|x|^{-\lambda},
 \qquad 0<\delta<R.
\]
The barrier lies in $\mathcal L_\sigma$ and is $L$-harmonic off zero.
Thus $LW_\delta\geq H-h\geq0$ in
$B_R\setminus\overline{B_\delta}$.  For $0<|x|\leq\delta$ the
barrier is at least $M\geq P(x)$; on $B_R^c$, $P=0$.
Comparison therefore gives $W_\delta\geq0$ on the annulus.
For fixed $x\ne0$, letting $\delta\downarrow0$ proves $w(x)\geq P(x)$.

Now use the actual bounded truncations
\[
 h_j=\min\{H,j\}\,
 \mathbf1_{\{1/j<|y|<R-1/j\}},
\]
with an empty support when necessary.  They increase to $H$ almost
everywhere in the punctured ball, and each is supported in a compact
annulus and satisfies $h_j\leq H$.  Monotone convergence gives
\begin{equation}\label{pp:green-source-minorant}
 \int_{B_R}G_R(x,y)H(y)\,dy\leq w(x)<\infty
 \quad(x\in B_R\setminus\{0\}).
\end{equation}
For a fixed $x_0\ne0$ inside the ball, $G_R(x_0,y)$ has a positive
lower bound on a sufficiently small closed ball around zero.
The last inequality makes $H$ integrable there.  Its assumed punctured
local integrability gives $H\in L^1_{\mathrm{loc}}(\mathbb R^n)$.
Finally exhaust $R\to\infty$ using \eqref{pd:green-limit} to obtain
\eqref{pd:linear-minorant}.  If $w$ is continuous at zero, take any
$x_j\ne0$ tending to zero.  Fatou, applied to
$K_\sigma(x_j-y)H(y)\geq0$, gives
$I_\sigma H(0)\leq\liminf_jI_\sigma H(x_j)\leq w(0)$.

When $w$ is continuous at zero, Lemma~\ref{ho:origin-extension}
now gives the full-space inequality, and the full-space equality for a
punctured equality. Its only nonlocal input is the assumed natural
tail; source integrability at zero has already been proved and is
also a conclusion of that lemma.
\end{proof}

\paragraph{Integral representation from source mass estimates}
\label{lc:section}

The low-order positive minorants permit a source mass argument
that is not supplied by the higher-order representation argument above.
We prove the integral representation independently of pressure:
normalized integral lower bounds for the coefficients on a fixed ball
and global Euler lower bounds remove the harmonic constants when
$pq>1$ and $\gamma_->0$, including a sublinear component.  The nonexistence consequence
then follows from the argument in Subsubsection~\ref{pt:section}.

\begin{lemma}
\label{lc:exactness}
Let $0<\sigma<\min\{2,n\}$, $p,q>0$, $pq>1$, and let $a,b>0$
be continuous away from zero with the global weak Euler lower bounds
\eqref{v5:euler-lower} and strict gap \eqref{v5:master-gap}.
Fix $e_0\in\mathbb S^{n-1}$ and $R_*\geq1$, and assume only the
normalized integral lower bounds on a fixed ball
\begin{equation}
 \inf_{R\geq R_*}\int_D\frac{a(Rx)}{a(Re_0)}\,dx>0,
 \qquad
 \inf_{R\geq R_*}\int_D\frac{b(Rx)}{b(Re_0)}\,dx>0,
 \qquad D=B_{1/8}(3e_0/4).
 \label{lc:patch-floors}
\end{equation}
Suppose $u,v\in C(\mathbb R^n)\cap\mathcal L_\sigma$ are nonnegative,
strictly positive off zero, and solve
$L_\sigma u=g:=av^p$, $L_\sigma v=f:=bu^q$ in distributions off zero.
Then $u,v\in\mathcal L_0$, both origin values are
positive, and
\begin{equation}
 u=I_\sigma(av^p),\qquad v=I_\sigma(bu^q)
 \quad\hbox{everywhere},\qquad fu+gv\in L^1_{\mathrm{loc}}.
 \label{lc:exact}
\end{equation}
The kernel-weighted coefficient integrals of
\eqref{gc:pair-origin} are also finite with $s_1=s_2=\sigma$.
% This representation requires no modulus assumption on the coefficient
% dilation family, pressure estimate, or nonlinear energy bound.
% The exterior power lower bounds used in the proof follow from the
% Euler lower bounds.
\end{lemma}
\begin{proof}
The positive minorants of Lemma~\ref{pp:automatic-source} give
$u\ge I_\sigma g$ and $v\ge I_\sigma f$, and their continuity
at zero gives entire equations and locally integrable sources.
Lemma~\ref{gp:semiorbit-source} applies before pressure and provides
\begin{equation}
 M_g(R)\le CR^{\lambda-\alpha_-},\qquad
 M_f(R)\le CR^{\lambda-\beta_-},\qquad
 0<\alpha_-,\beta_-\le\lambda,
 \label{lc:source-masses}
\end{equation}
where $\lambda=n-\sigma$ and the scaling exponents are
\eqref{gp:floor-scales}. These bounds supply the strict source moments
\[
 g\in\mathsf M_{\lambda-\alpha_-/2},\qquad
 f\in\mathsf M_{\lambda-\beta_-/2},\qquad
 0<\lambda-\alpha_-/2,\lambda-\beta_-/2<\lambda,
\]
by dyadic summation. Consequently
Proposition~\ref{gc:source-inversion} places the two source potentials
$P_u=I_\sigma g$, $P_v=I_\sigma f$ in $\mathcal L_0$, hence also
in $\mathcal L_\sigma$, and gives their entire distributional equations.
The nonnegative remainders $u-P_u$ and $v-P_v$ are entire
$L_\sigma$-harmonic in $\mathcal L_\sigma$. Applying
\cite[Theorem~1.1]{Fall2016} with $s=\sigma/2\in(0,1)$ shows that
they are affine. Their nonnegativity on the whole space makes them
constants $C_u,C_v\ge0$.
Lemma~\ref{pd:local-gain} gives continuity of the potentials off
zero, since their sources are locally bounded there. The affine
identities therefore hold at every nonzero point.

The balance identities give
\[
 n+B_-=q\alpha_-+\lambda-\beta_->0,\qquad
 n+A_-=p\beta_-+\lambda-\alpha_->0.
\]
Integrating the global Euler lower bounds along rays yields positive exterior
power lower bounds with exponents $A_-,B_-$. If $C_u>0$, then
$u\ge C_u$ and $M_f(R)\ge cR^{n+B_-}$ for large $R$.
Its lower exponent exceeds $\lambda-\beta_-$ by $q\alpha_->0$,
contradicting \eqref{lc:source-masses}. Similarly $C_v>0$ contradicts
the other source bound. Thus $C_u=C_v=0$.

The off-zero identities and $P_u,P_v\in\mathcal L_0$ give
$u,v\in\mathcal L_0$. The already established strict source moments
allow Corollary~\ref{gc:pair-inversion} to identify the origin values,
prove their positivity and the kernel-weighted coefficient integrability conditions at zero,
and give the local nonlinear energy. No pressure or nonexistence
conclusion is used in this representation argument.
\end{proof}

\subsubsection{Proofs of the PDE consequences}
\label{pt:section}

Each PDE argument first proves the integral representation and finite
local energy. The integral nonexistence theorem is invoked only after these independent representation inputs have been established.

\begin{proof}[Proof of Corollary~\ref{cor:pde-front}]
For the Riesz kernel and pure powers of Corollary~\ref{cor:pde-front},
the global weak Euler lower bounds imply exterior power lower bounds:
\[
 a(r\omega)\geq a(\omega)r^{A_-},\qquad
 b(r\omega)\geq b(\omega)r^{B_-}
 \quad(r\geq1,\ \omega\in\mathbb S^{n-1}).
\]
Positive minima on the unit sphere give fixed positive lower-bound constants.
In dimension one this uses the minimum of two positive values and does
not introduce a comparison between the two normalized sign branches.
Write $\lambda=n-\sigma$ and define the corresponding algebraic exponents
\[
 \alpha_-=
 \frac{\sigma+A_-+p(\sigma+B_-)}{pq-1},\qquad
 \beta_-=
 \frac{\sigma+B_-+q(\sigma+A_-)}{pq-1}.
\]
Their gap identity is
\begin{equation}
 (pq-1)(\alpha_-+\beta_--\lambda)
       =(p+1)(q+1)\gamma_->0.
 \label{pt:gap-scales}
\end{equation}

In case (i), Lemma~\ref{gc:moment-bootstrap} gives the two upper bounds on the
exponents $\alpha_-,\beta_-\leq\lambda$ from the initial natural tails.
Equation~\eqref{pt:gap-scales} then forces both exponents to be positive;
a nonpositive exponent would force the other to exceed $\lambda$.
Proposition~\ref{gc:representation} gives exact potentials everywhere
and finite local nonlinear energy before a nonlinear theorem is invoked.

In case (ii), the orbit implies the patch input
\eqref{lc:patch-floors}.  Indeed, for $R\geq R_*:=\max\{1,R_0\}$
and $x\in D=B_{1/8}(3e_0/4)$, its same-component bound gives
\[
 a_R(x)\geq e^{-H_a\Omega(3/8)},\qquad
 b_R(x)\geq e^{-H_b\Omega(3/8)}.
\]
Both patch integrals therefore have positive lower bounds uniform
in $R$.  Lemma~\ref{lc:exactness} gives exact potentials and local energy after excluding the harmonic constants.  That lemma
uses positive minorants and source masses, and does not impose
$p,q\geq1$ or use a pressure conclusion to remove the defects.

In case (iii), Lemma~\ref{ui:source-bounds} obtains the upper bounds on the
exponents by local compact integer tests without an initial tail.
Equation~\eqref{pt:gap-scales} again gives positive exponents, and
Proposition~\ref{ui:representation} gives $u,v\in\mathcal L_0$, the integral representations and local nonlinear energy.

In each case the obtained pair is finite, positive and continuous off
zero, satisfies the exact equal-order system, and has the local origin
energy required by Theorem~\ref{thm:dini-front}.  The coefficients retain
the same coefficient dilation family, global Euler lower bounds
and strict gap.  That independent
integral theorem now gives the contradiction.
\end{proof}

\begin{remark}
\label{pt:scope}
The three PDE consequences use the same integral theorem after
independent representation arguments. Proposition~\ref{ui:representation}
also treats unequal integer orders, but the nonlinear Liouville
theorem requires equal orders. The compact-test method uses
locality and $p,q\ge1$; the positive low-order minorant permits a
sublinear component. No noninteger tail removal or singular-origin
extension is asserted.
\end{remark}

\subsection{Nonlinear PDE representation and nonexistence}
\label{r7:nonlinear-pde-section}

The reduction to potentials extends to the primitive response class.
We separate this reduction from the pressure argument.  In this
subsection $L_s=(-\Delta)^{s/2}$ has the Riesz inverse $I_s$.
For the representation theorem the two orders may differ; write
$\lambda_i=n-\sigma_i$.

\begin{remark}
\label{r7:lipschitz-response-extension}
All responses in this subsection have the locally Lipschitz regularity
of Theorem~\ref{r7:master-front}. Logarithmic derivative inequalities
are interpreted almost everywhere. Their integrated ratio bounds,
locally Lipschitz inverses and locally $C^{1,1}$ primitives are the
only response regularity used below.
\end{remark}

Fix positive responses $F,G$ vanishing at zero, with the regularity in
Remark~\ref{r7:lipschitz-response-extension}, and suppose
\begin{equation}
 \begin{gathered}
 \frac{sF'(s)}{F(s)}\ge q_0>0,\qquad
 \frac{sG'(s)}{G(s)}\ge p_0>0\quad\hbox{a.e. }s>0,
 \qquad p_0q_0>1,\\
 P_F(s):=\int_0^sF(t)\,dt\ge m_FsF(s),\qquad
 P_G(s):=\int_0^sG(t)\,dt\ge m_GsG(s),\qquad m_F,m_G>0.
 \end{gathered}
 \label{r7:pde-response-class}
\end{equation}
Write
\begin{equation}
 q_*:=m_F^{-1}-1\ge q_0,\qquad
 p_*:=m_G^{-1}-1\ge p_0.
 \label{r7:pde-effective-exponents}
\end{equation}
For $F_s(z)=F(sz)/F(s)$, the primitive argument gives
\begin{equation}
 F_s(z)\ge
 \begin{cases}m_Fz^{q_*},&0<z\le1,\\ z^{q_0},&z\ge1,
 \end{cases}
 \qquad
 z\le m_F^{-1/q_*}F_s(z)^{1/q_*}+F_s(z)^{1/q_0}.
 \label{r7:inverse-response}
\end{equation}
Indeed $P_F(sz)\ge z^{1/m_F}P_F(s)$ when $z\le1$, and
$F(sz)\ge P_F(sz)/(sz)$.  The large-argument bound follows from
the lower logarithmic derivative.  The same assertions hold for $G$.

\begin{theorem}
\label{r7:nonlinear-pde-representation}
Let $0<\sigma_1,\sigma_2<n$, and let $a,b>0$ be continuous off zero
with global weak Euler lower bounds $A>-\sigma_1$, $B>-\sigma_2$
and the normalized
Dini orbit hypothesis \eqref{eq:front-dini-orbit}.  Suppose
\eqref{r7:pde-response-class} holds.  Let $u,v\in C(\mathbb R^n)$
be nonnegative and not both zero, solving
\begin{equation}
 L_{\sigma_1}u=g:=aG(v),\qquad L_{\sigma_2}v=f:=bF(u)
 \quad\hbox{off zero}.
 \label{r7:nonlinear-pde-system}
\end{equation}
Assume one of the following:
\begin{enumerate}[label=(\roman*)]
\item $p_0,q_0\ge1$, $u\in\mathcal L_{\sigma_1}$,
$v\in\mathcal L_{\sigma_2}$, and the equations
hold in distributions off zero;
\item $\sigma_i=2m_i<n$ for integers $m_i\ge1$, $p_0,q_0\ge1$,
$u\in C^{2m_1}_{\mathrm{loc}}(\mathbb R^n\setminus\{0\})$,
$v\in C^{2m_2}_{\mathrm{loc}}(\mathbb R^n\setminus\{0\})$,
and the equations hold classically there;
\item $0<\sigma_i<\min\{2,n\}$,
$u\in\mathcal L_{\sigma_1}$, $v\in\mathcal L_{\sigma_2}$,
and the equations hold in distributions off zero.
\end{enumerate}
Then $f,g$ are nonzero, locally integrable on the whole space,
$u,v\in\mathcal L_0$, and
\begin{equation}
 u=I_{\sigma_1}g,\qquad v=I_{\sigma_2}f
 \quad\hbox{at every point of }\mathbb R^n.
 \label{r7:nonlinear-pde-exact}
\end{equation}
Both solutions are strictly positive everywhere,
$fu+gv\in L^1_{\mathrm{loc}}(\mathbb R^n)$, and
$\int_{B_1}a(x)|x|^{-\lambda_1}dx$ and
$\int_{B_1}b(x)|x|^{-\lambda_2}dx$ are finite.  
% In case~(ii) no
% initial tail or sign condition on intermediate Laplacians is imposed.
\end{theorem}

\begin{proof}
In every case the linear origin-extension lemma gives entire
 distributional equations and $f,g\in L^1_{\mathrm{loc}}$; in
case~(iii) one may first use Lemma~\ref{pp:automatic-source}.
The two sources are nonzero: if, for example, $g=0$ off zero,
positivity of $a$ and $G$ gives $v=0$, and the other equation then
gives $u=0$.  Continuity rules out a nonzero value supported only at
zero.

\emph{Case (i): a finite moment iteration.}
The unnormalized form of \eqref{r7:inverse-response}, together with
the exterior coefficient bounds implied by the Euler inequalities,
gives
\[
 u\le C\sum_{Q\in\{q_0,q_*\}}|x|^{-B/Q}f^{1/Q},\qquad
 v\le C\sum_{P\in\{p_0,p_*\}}|x|^{-A/P}g^{1/P}
 \quad(|x|>1).
\]
For $Q>1$, H\"older's inequality against $|x|^{-t}$ shows that
$f\in\mathsf M_b$ implies
\[
 u\in\mathsf M_t\quad\hbox{for all }
 t>\max_{Q\in\{q_0,q_*\}}
       \{n(1-1/Q)+(b-B)/Q\}.
\]
When $Q=1$, the same conclusion follows by direct integration of
$|x|^{-B}f$; no conjugate exponent is used.  The analogous transfer
holds for $g,v$ and the exponents $P$.

Lemma~\ref{tb:positive-test} initially gives
$f,g\in\mathsf M_{n+\varepsilon}$ for every $\varepsilon>0$.
Let $\mathcal P(a,b)$ mean that
$g\in\mathsf M_{a+\varepsilon}$ and
$f\in\mathsf M_{b+\varepsilon}$ for every $\varepsilon>0$.
The preceding transfers and the same positive-source test give
$\mathcal P(a,b)\Rightarrow\mathcal P(\mathcal S(a,b))$, where
\begin{equation}
 \begin{split}
 T_1(a,b)&=\left[\max_{Q\in\{q_0,q_*\}}
              \left\{\lambda_1-\frac{n+B-b}{Q}\right\}\right]_+,\\
 T_2(a,b)&=\left[\max_{P\in\{p_0,p_*\}}
              \left\{\lambda_2-\frac{n+A-a}{P}\right\}\right]_+,\\
 \mathcal S_i(a,b)&=\min\{n,T_i(a,b)\}\quad(i=1,2).
 \end{split}
 \label{r7:nonlinear-moment-map}
\end{equation}
The cap retains the initial known moments whenever a proposed
threshold is larger.  All uses of the transfer and test lemmas take
strictly positive surplus exponents, exactly as in
Lemma~\ref{gc:moment-bootstrap}; thus no endpoint moment is asserted.
The first coordinate depends only on $b$ with Lipschitz constant at
most $1/q_0$, and the second only on $a$ with constant at most
$1/p_0$.  Hence $\mathcal S^2$ is a contraction with constant at
most $(p_0q_0)^{-1}<1$.

The iteration from $(n,n)$ decreases to the unique fixed point
$z$.  Moreover $(\lambda_1,\lambda_2)$ is a strict supersolution:
\[
 T_1(\lambda_1,\lambda_2)
 =\left[\lambda_1-\frac{\sigma_2+B}{q_*}\right]_+<\lambda_1,
 \qquad
 T_2(\lambda_1,\lambda_2)
 =\left[\lambda_2-\frac{\sigma_1+A}{p_*}\right]_+<\lambda_2.
\]
Iterating from this supersolution shows $z_i<\lambda_i$.
Consequently a finite iteration from $(n,n)$ supplies strict source
moments $g\in\mathsf M_{t_g}$, $f\in\mathsf M_{t_f}$ with
$0<t_g<\lambda_1$, $0<t_f<\lambda_2$.  At the fixed point the raw
solution thresholds are at most $\sigma_1+z_1$ and $\sigma_2+z_2$,
both strictly below $n$.
A further finite transfer therefore gives $u,v\in\mathcal L_0$.

\emph{Case (ii): compact tests with normalized responses.}
For each $R>0$ choose the unique positive pair
\begin{equation}
 s_R=R^{\sigma_1}a(Re_0)G(t_R),\qquad
 t_R=R^{\sigma_2}b(Re_0)F(s_R).
 \label{r7:pde-amplitudes}
\end{equation}
For $H_R(s)=s/G(R^{\sigma_2}b(Re_0)F(s))$, the integrated lower
response inequalities give
$H_R(s_2)/H_R(s_1)\le(s_2/s_1)^{1-p_0q_0}$ whenever $s_2>s_1$.
Thus $H_R$ is strictly decreasing, with endpoint limits infinity and
zero, which proves existence and uniqueness.  Normalize
\[
 U_R(x)=u(Rx)/s_R,\quad V_R(x)=v(Rx)/t_R,\quad
 g_R=a_R G_{t_R}(V_R),\quad f_R=b_R F_{s_R}(U_R).
\]
Then $L_{\sigma_1}U_R=g_R$, $L_{\sigma_2}V_R=f_R$, and
$g_R=R^{\sigma_1}g(R\,\cdot\,)/s_R$,
$f_R=R^{\sigma_2}f(R\,\cdot\,)/t_R$.  No potential identity is used
in these transformations.

The orbit hypothesis gives uniform positive lower bounds for $a_R,b_R$
on $1\le|x|\le2$ for all sufficiently large $R$.  Here $n>2m_i$
implies $n\ge3$, so the fixed annulus is connected.  Choose
$0\le\chi\le1$, equal to one on $B_1$ and supported in $B_2$,
and integers $k_g,k_f$ so large that
\[
 k_g-\sigma_1-k_f/q_0\ge0,\qquad
 k_f-\sigma_2-k_g/p_0\ge0.
\]
Such integers exist because $p_0q_0>1$.
The compact-test estimate \eqref{ui:cutoff-image} and
\eqref{r7:inverse-response} give, with
$\mathfrak G_R=\int g_R\chi^{k_g}$ and
$\mathfrak F_R=\int f_R\chi^{k_f}$,
\begin{equation}
 \mathfrak G_R\le C\sum_{Q\in\{q_0,q_*\}}
                     \mathfrak F_R^{1/Q},\qquad
 \mathfrak F_R\le C\sum_{P\in\{p_0,p_*\}}
                     \mathfrak G_R^{1/P}.
 \label{r7:compact-response-feedback}
\end{equation}
For $Q>1$ use H\"older on the unit annulus; the remaining cutoff
power $k_g-\sigma_1-k_f/Q$ is nonnegative because $Q\ge q_0$.
For $Q=1$ integrate directly, and treat $P=1$ in the same way.
If one mass is large, so is the other, and the large-mass branches
of \eqref{r7:compact-response-feedback} give
$\mathfrak G_R\le C\mathfrak G_R^{1/(p_0q_0)}$.
Thus both masses are uniformly bounded.

We record the amplitude decay needed below.  For $\rho>1$, put
$S=s_{R\rho}/s_R$, $T=t_{R\rho}/t_R$,
$x=\sigma_1+A>0$, and $y=\sigma_2+B>0$.
The weak Euler inequalities and \eqref{r7:inverse-response}'s lower
response bounds imply $S,T<1$: if $S\ge1$, then
$T\ge\rho^y S^{q_0}>1$ and
$S\ge\rho^xT^{p_0}\ge\rho^{x+p_0y}S^{p_0q_0}$, a contradiction.
The other case is symmetric.  Using the small-argument bounds now gives
\[
 S\ge m_G\rho^xT^{p_*},\qquad
 T\ge m_F\rho^yS^{q_*},
\]
and hence
\begin{equation}
 \frac{s_{R\rho}}{s_R}\le C\rho^{-\alpha_*},\quad
 \frac{t_{R\rho}}{t_R}\le C\rho^{-\beta_*},\quad
 \alpha_*:=\frac{x+p_*y}{p_*q_*-1}>0,\quad
 \beta_*:=\frac{y+q_*x}{p_*q_*-1}>0.
 \label{r7:pde-amplitude-decay}
\end{equation}
The constants are independent of $R,\rho$.
Returning to the original variables in the unit-ball source bounds
therefore gives
\begin{equation}
 \int_{B_R}g\le CR^{\lambda_1}s_R\le CR^{\lambda_1-\alpha_*},\qquad
 \int_{B_R}f\le CR^{\lambda_2}t_R\le CR^{\lambda_2-\beta_*}
 \label{r7:pde-source-masses}
\end{equation}
for all sufficiently large $R$.  A negative exponent would already
contradict a nonzero compact source mass.  Otherwise these bounds
supply source moments strictly below the respective $\lambda_i$ by
dyadic summation.
The source bounds on $B_{2R}$ and the inverse response inequality on
$R<|x|<2R$ give
\[
 \int_{R<|x|<2R}u\le CR^n s_R,\qquad
 \int_{R<|x|<2R}v\le CR^n t_R.
\]
For clarity, source bounds on the larger ball are obtained with the
cutoff at scale $2R$, before the annular transfer; one does not divide
by a vanishing cutoff at the outer edge.  Equation
\eqref{r7:pde-amplitude-decay} now yields $u,v\in\mathcal L_0$.

In cases~(i) and~(ii), apply Proposition~\ref{gc:source-inversion}
to the actual sources $g,f$, not to auxiliary powers.  It gives
\eqref{r7:nonlinear-pde-exact} and strict positivity at every point.
Local boundedness of $u,v$ times local source integrability gives
$fu+gv\in L^1_{\mathrm{loc}}$.  The positive origin values in the
exact identities imply the two kernel-weighted coefficient integrals
by lower bounds for $F(u),G(v)$ near zero.

\emph{Case (iii): positive minorants and constant remainders.}
Lemma~\ref{pp:automatic-source} gives $u\ge I_{\sigma_1}g$,
$v\ge I_{\sigma_2}f$.  Use the normalization
\eqref{r7:pde-amplitudes}.  On the fixed patch
$D=B_{1/8}(3e_0/4)$ the orbit hypothesis gives uniform positive
lower coefficient bounds.  If $m_g=\int_{B_1}g_R$,
$m_f=\int_{B_1}f_R$, the positive minorants give
$U_R\ge cm_g$, $V_R\ge cm_f$ on $D$.  Integrating the lower
response bounds gives
\[
 m_f\ge c\min\{m_g^{q_0},m_g^{q_*}\},\qquad
 m_g\ge c\min\{m_f^{p_0},m_f^{p_*}\}.
\]
On the large-mass branch, $p_0q_0>1$ bounds both masses uniformly.
This step uses no H\"older inequality and allows either lower
exponent to be less than one.  The algebraic source transformations
and \eqref{r7:pde-amplitude-decay} then give
\eqref{r7:pde-source-masses}, and hence strict source moments.

Set $P=I_{\sigma_1}g$, $Q=I_{\sigma_2}f$.  Source inversion puts
$P,Q\in\mathcal L_0$ and gives their entire source equations.
The remainders $u-P,v-Q$ are nonnegative, entire
$L_{\sigma_1}$- and $L_{\sigma_2}$-harmonic functions in their
respective natural-tail classes.
Theorem~1.1 of \cite{Fall2016} makes them affine, and
nonnegativity makes them constants $C_u,C_v\ge0$.
If $C_u>0$, then $F(u)\ge F(C_u)>0$ and the exterior lower bound
for $b$ gives
\[
 \int_{B_R}f\ge cR^{n+B}
 \quad\hbox{for large }R.
\]
But $n+B>\lambda_2$ whereas the upper bound in
\eqref{r7:pde-source-masses} has exponent
$\lambda_2-\beta_*<\lambda_2$.  This is a contradiction.  The other
constant vanishes in the same way.  The identities hold at every
nonzero point because the potentials of the locally bounded sources
are continuous there.  Thus $u,v\in\mathcal L_0$, and a final
application of Proposition~\ref{gc:source-inversion} identifies the
origin as well.  The coefficient integrals follow from the same
positive-origin argument.  This proves all assertions.
\end{proof}

\begin{corollary}
\label{r7:nonlinear-pde-nonexistence}
Under the hypotheses of Theorem~\ref{r7:nonlinear-pde-representation},
assume $\sigma_1=\sigma_2=\sigma$. If
\begin{equation}
 (n+A)m_G+(n+B)m_F>n-\sigma,
 \label{r7:pde-primitive-gap}
\end{equation}
there is no nonnegative nonzero pair in any of its three PDE classes.
% In particular, the low-order class allows a sublinear response;
% the even-integer class requires no initial tail assumption.
\end{corollary}
\begin{proof}
The representation theorem first supplies exact potentials, strict
positivity, and finite local nonlinear energy.  Apply Theorem~\ref{r7:master-front}, using
Remark~\ref{r7:lipschitz-response-extension} if necessary.
\end{proof}

\section{Universal interior bounds without regular variation}
\label{r7:universal-bound-section}

We apply the nonlinear PDE result to a class closed under amplitude
normalization.  The doubling and rescaling principle is classical;
see \cite{PolacikQuittnerSouplet2007}.  Universal bounds for
non-scale-invariant second-order systems with regularly varying
responses were developed in \cite[Section~4]{QuittnerSouplet2025}.
The conclusion below uses finite logarithmic derivative windows on
the full response range, and permits responses without regular
variation at either endpoint.

\begin{theorem}
\label{r7:universal-interior-bound}
Let $m\ge1$ be an integer, $\sigma=2m<n$, and fix finite constants
\begin{equation}
 1\le p_0\le p_1,\qquad 1\le q_0\le q_1,\qquad
 p_0q_0>1,\qquad
 \frac{n}{p_1+1}+\frac{n}{q_1+1}>n-\sigma.
 \label{r7:universal-response-window}
\end{equation}
Let $F,G$ be continuous on $[0,\infty)$, locally Lipschitz on
$(0,\infty)$, positive on $(0,\infty)$, and zero at zero, with
\begin{equation}
 q_0\le\frac{sF'(s)}{F(s)}\le q_1,\qquad
 p_0\le\frac{sG'(s)}{G(s)}\le p_1
 \quad\hbox{for a.e. }s>0.
 \label{r7:universal-elasticities}
\end{equation}
For $r>0$ let $s_r,t_r$ be the unique positive solution of
\begin{equation}
 s_r=r^\sigma G(t_r),\qquad t_r=r^\sigma F(s_r).
 \label{r7:interior-amplitudes}
\end{equation}
There is a constant $C$, depending only on
$n,m,p_0,p_1,q_0,q_1$, such that every nonnegative classical solution
of
\begin{equation}
 (-\Delta)^m u=G(v),\qquad (-\Delta)^m v=F(u)
 \quad\hbox{in a proper domain }\Omega\subset\mathbb R^n
 \label{r7:interior-pde}
\end{equation}
satisfies
\begin{equation}
 u(x)\le Cs_{d(x)},\qquad v(x)\le Ct_{d(x)},\qquad
 d(x)=\operatorname{dist}(x,\partial\Omega).
 \label{r7:universal-interior-estimate}
\end{equation}
% No boundary condition, boundedness assumption, or sign condition on
% intermediate Laplacians is imposed.
\end{theorem}

\begin{proof}
The same strictly monotone scalar equation as in
\eqref{r7:pde-amplitudes} gives the amplitudes.  Their scale bounds
will also specify the normalization used in the proof.  For $R>r$,
write $L=\log(R/r)$, $X=\log(s_R/s_r)$, and
$Y=\log(t_R/t_r)$.  Integrating
\eqref{r7:universal-elasticities} on the amplitude intervals gives
$P\in[p_0,p_1]$, $Q\in[q_0,q_1]$ such that
\[
 X=\sigma L+PY,\qquad Y=\sigma L+QX.
\]
Thus
\[
 -X=\frac{\sigma(1+P)}{PQ-1}L,\qquad
 -Y=\frac{\sigma(1+Q)}{PQ-1}L.
\]
Choose $0<a_*\le a^*<\infty$ bounding both displayed positive
quotients on the compact exponent rectangle.  Then
\begin{equation}
 (R/r)^{a_*}\le\frac{s_r}{s_R}\le(R/r)^{a^*},\qquad
 (R/r)^{a_*}\le\frac{t_r}{t_R}\le(R/r)^{a^*}.
 \label{r7:interior-scale-comparison}
\end{equation}
In particular, $s_r,t_r$ are continuous strictly decreasing
bijections of $(0,\infty)$ onto itself.

Let $s^{-1},t^{-1}$ denote their inverses, and define
\[
 M(x)=\max\{1/s^{-1}(u(x)),\,1/t^{-1}(v(x))\},
\]
with each term set to zero when its argument is zero.  We claim that
$M(x)d(x)\le C_0$ uniformly over the response class and the domains.
Otherwise the doubling lemma from
\cite{PolacikQuittnerSouplet2007} gives a sequence of solutions and
points $y_j$, with $r_j=1/M(y_j)$, such that
\[
 B_{jr_j}(y_j)\subset\Omega_j,\qquad
 M(y)\le2M(y_j)\quad(|y-y_j|<jr_j).
\]
The responses may also vary with $j$ while retaining the fixed
windows.  Set
\[
\begin{aligned}
 U_j(x)&=\frac{u_j(y_j+r_jx)}{s_{r_j}},&
 V_j(x)&=\frac{v_j(y_j+r_jx)}{t_{r_j}},\\
 F_j(z)&=\frac{F(s_{r_j}z)}{F(s_{r_j})},&
 G_j(z)&=\frac{G(t_{r_j}z)}{G(t_{r_j})}.
\end{aligned}
\]
where each amplitude pair and response belongs to the corresponding
member of the sequence.  Then, on $B_j$,
\begin{equation}
 (-\Delta)^m U_j=G_j(V_j),\qquad
 (-\Delta)^m V_j=F_j(U_j),\qquad
 \max\{U_j(0),V_j(0)\}=1.
 \label{r7:doubled-system}
\end{equation}
The doubling inequality implies $u_j(y_j+r_jx)\le s_{r_j/2}$
and $v_j(y_j+r_jx)\le t_{r_j/2}$; hence
\eqref{r7:interior-scale-comparison} gives
$0\le U_j,V_j\le2^{a^*}$ throughout $B_j$.

The normalized responses retain their elasticity windows and have
value one at one.  They are uniformly Lipschitz on every compact
interval $[0,L]$, including zero.  Indeed, for $0<z\le1$,
$F_j(z)\le z^{q_0}$ and
$F_j'(z)\le q_1z^{q_0-1}\le q_1$ a.e.; on $[1,L]$ use the
upper power bound, and argue similarly for $G_j$.
Arzel\`a--Ascoli and a diagonal subsequence give locally uniform
limits $F_\infty,G_\infty$ on $[0,\infty)$.  The integrated
response ratio inequalities pass to the limits, so these responses
are locally Lipschitz, positive off zero, zero at zero, normalized at
one, and satisfy \eqref{r7:universal-elasticities} almost everywhere.
They need not be $C^1$; this is why
Remark~\ref{r7:lipschitz-response-extension} was recorded explicitly.

The right sides of \eqref{r7:doubled-system} are uniformly bounded
on compact balls.  Interior $W^{2m,r}$ estimates with $r>n$ and
compact embedding therefore give locally uniform convergence of a
subsequence to bounded nonnegative entire functions $U,V$, with local
$C^{1,\alpha}$ convergence after fixing $0<\alpha<1-n/r$.
The uniform response convergence passes the equations to
\[
 (-\Delta)^m U=G_\infty(V),\qquad
 (-\Delta)^m V=F_\infty(U),\qquad
 \max\{U(0),V(0)\}=1.
\]
The limiting responses are Lipschitz on bounded intervals, so
$F_\infty(U)$ and $G_\infty(V)$ are locally H\"older continuous.
Interior Schauder estimates give $U,V\in C^{2m,\alpha}_{\rm loc}$
for a possibly smaller $\alpha>0$, hence a classical pair.  At this point the
pair is merely nonnegative and nonzero; no high-order maximum
principle is being used.  Boundedness supplies its natural tails.
Theorem~\ref{r7:nonlinear-pde-representation}(i), which was stated
for nonnegative nonzero pairs, gives exact positive Riesz potentials
and therefore strict positivity of both components.  Its primitive
hypotheses hold with
$m_F=(q_1+1)^{-1}$, $m_G=(p_1+1)^{-1}$.
Corollary~\ref{r7:nonlinear-pde-nonexistence} and
\eqref{r7:universal-response-window} now contradict the nonzero
limit.  This proves the claim about $M(x)d(x)$.

Increase $C_0$ to at least one.  From $M(x)\le C_0/d(x)$ and
monotonicity of the amplitude functions,
\[
 u(x)\le s_{d(x)/C_0}\le C_0^{a^*}s_{d(x)},\qquad
 v(x)\le t_{d(x)/C_0}\le C_0^{a^*}t_{d(x)}.
\]
This is \eqref{r7:universal-interior-estimate}.
\end{proof}

\begin{example}
\label{r7:nonrv-interior-example}
The estimate applies to
\[
 F(s)=s^q\exp\{\varepsilon\sin(\log s)\},\qquad
 G(s)=s^p\exp\{\delta\sin(\sqrt2\log s)\},\qquad s>0,
\]
extended by zero at zero, whenever
\[
 q-|\varepsilon|\ge1,\qquad p-\sqrt2|\delta|\ge1,
 \qquad (q-|\varepsilon|)(p-\sqrt2|\delta|)>1,
\]
and
\[
 \frac{n}{p+\sqrt2|\delta|+1}
 +\frac{n}{q+|\varepsilon|+1}>n-2m.
\]
Their elasticities are respectively
$q+\varepsilon\cos(\log s)$ and
$p+\sqrt2\delta\cos(\sqrt2\log s)$.
For nonzero perturbation parameters the corresponding response is
not regularly varying at zero or infinity: for a generic fixed
$c>0$, its ratio at $cs$ and $s$ retains a nonconstant oscillation
in $\log s$.  Thus this example is outside an endpoint
regular-variation assumption.
\end{example}

The upper elasticity bounds are used here for compactness of the
normalized response family.  A positive primitive ratio by itself
does not supply this equicontinuity: normalized responses can have
increasingly steep transitions and discontinuous subsequential
limits.  The primitive nonexistence theorem therefore does not,
without an additional compactness hypothesis, yield the same
universal estimate for its entire response class.

\vspace{1cm}
\textbf{AI assistance.}
ChatGPT assisted with mathematical exploration, manuscript organization,
English editing and computational checks. The authors are responsible for
the statements, proofs, citations and final text.
% References are embedded to match the supplied single-file template.
\providecommand{\bysame}{\leavevmode\hbox to3em{\hrulefill}\thinspace}
\providecommand{\MR}{\relax\ifhmode\unskip\space\fi MR }
% \MRhref is called by the amsart/book/proc definition of \MR.
\providecommand{\MRhref}[2]{%
  \href{http://www.ams.org/mathscinet-getitem?mr=#1}{#2}
}
\providecommand{\href}[2]{#2}

\end{document}